\documentclass[12pt,thmsa,sw20lart]{article}
\usepackage{amsmath,amsthm,amsfonts,amssymb,lineno}
\usepackage{graphicx}
 \usepackage{xcolor}
   \usepackage{epsfig}
     \usepackage[english]{babel}
      \usepackage[latin1]{inputenc}
       \usepackage{cite}

\newtheorem{theorem}{Theorem}[section]
\newtheorem{corollary}{Corollary}[section]
\newtheorem{definition}{Definition}[section]
\newtheorem{lemma}{Lemma}[section]
\newtheorem{proposition}{Proposition}[section]
\newtheorem{remark}{Remark}[section]
\numberwithin{equation}{section} \numberwithin{theorem}{section}
\numberwithin{lemma}{section} \numberwithin{corollary}{section}
\numberwithin{definition}{section} \numberwithin{remark}{section}
\numberwithin{proposition}{section}

\newcommand{\vx}{\boldsymbol{x}}
\newcommand{\vq}{\boldsymbol{q}}
\newcommand{\vvarphi}{\boldsymbol{\varphi}}

\newcommand{\RR}{{\mathbb R}}
\def\supp{\mathop{\mathrm{supp}}}

\title{\sc Singular limits of the quasi-linear Kolmogorov-type equation\\ with a source term}

\author{{\bf Ivan~V.~Kuznetsov} and {\bf Sergey~A.~Sazhenkov}
\footnote{Lavrentyev Institute of Hydrodynamics, Siberian Branch of Russian Academy of Sciences, Prospekt Lavrentyeva 15, Novosibirsk 630090, Russia; Novosibirsk State University, Mechanical \& Mathematical Department, Pirogova Street 2, Novosibirsk 630090, Russia.} \footnote{{\it E-mail addresses:} kuznetsov\vbox{\hrule width 1mm}i@hydro.nsc.ru (Ivan Kuznetsov), sazhenkovs@yandex.ru (Sergey Sazhenkov).}}

\date{}

\begin{document}
\maketitle
\begin{abstract}
Existence, uniqueness and stability of kinetic and entropy solutions to the boundary value
problem for the Kolmogorov-type genuinely nonlinear ultra-parabolic equation with a smooth source term is established. After this, we consider the case when the source term contains a small positive parameter and collapses to the Dirac delta-function, as this parameter tends to zero. In this case, the limiting passage from the original equation with the smooth source to the impulsive ultra-parabolic equation is fulfilled and rigorously justified. The proofs rely on the method of kinetic equation and on the compensated
compactness techniques for genuinely nonlinear equations.\\[1ex]



\end{abstract}

\section*{Introduction}
In this article, we study the second-order quasi-linear
ultra-parabolic equation with partial diffusivity and a smooth
distributed source supplemented by a set of initial, final and
boundary conditions. The article consists of two major parts. The
first part (Sections \ref{Pi-gamma}--\ref{sec.8}) is devoted to
development of the existence and uniqueness theory for the
boundary-value problem under consideration in suitable classes of
kinetic and entropy solutions. The research in the second part
(Sections \ref{SingLim}--\ref{Limited}) is related with the case
when the source term contains a small positive parameter and
collapses to the Dirac delta-function as this small parameter tends
to zero. In this case, we focus on justification of the claim that
the singular limit of the family of kinetic and entropy solutions of
the original problem exists and resolves the boundary-value problem
for the impulsive Kolmogorov-type equation.

Let us recall that ultra-parabolic equations with partial
diffusivity arise in fluid dynamics, physics of particles,
combustion theory, mathematical biology and financial mathematics
\cite{LPP-2002,OL-RAD-1973}. They describe, in particular,
non-stationary transport of matter or energy in cases when effects
of diffusion in some spatial directions are negligible as compared
to convection. In line with the famous works of A.~N.~Kolmogorov
\cite{KOLM-1931,KOLM-1934,KOLM-1937}, these equations are commonly
called {\it Kolmogorov-type equations}. Worth noticing that they
also have some other names in literature. For example, they are
called {\it the Graetz--Nusselt equations} in description of partial
diffusion of heat in fluid dynamics
\cite{MP-1956,PLOT-SAZH-2005,SAZH-2006} or {\it Fokker--Planck
(Fokker--Planck--Kolmogorov) equations} in description of stochastic
diffusive processes modelling Brownian motion \cite{BOG-2015},
\cite[Chapter II, \S\,21]{LIF-PIT}.

The nowadays theory of linear and quasi-linear ultra-parabolic
equations is rather vast. Our study in this article is somewhat
close to the works
\cite{AM-2013,AN-2003,EVZ-1994,Kozh-2010,TAS-2005,TAS-2013,TSA-2001}
devoted to well-posedness topics for boundary-value problems.
Besides the presence of the source term, the peculiarity of the
ultra-parabolic equation, considered in the present article, is that
the flux  in a purely convective direction (assigned to variable
$s$) is non-monotonous. This circumstance makes the question of
proper setting of initial and final conditions with respect to $s$
rather sophisticated. Our approach to overcoming this difficulty
ascends to the works \cite{BLN-1979,O-1996}. More precisely, we
formulate the initial and final conditions in the form of boundary
entropy inequalities or, equivalently, boundary kinetic equalities,
which eventually helps to establish the well-posedness results. The
proofs in the article are heavily based on {\it the genuine
nonlinearity (non-degeneracy) condition} imposed on the equation
under study. We systematically use the techniques elaborated for
genuinely nonlinear ultra-parabolic equations in
\cite{AM-2013,Kuz-2015,Kuz-2016,Kuz-2017,Panov2009,Panov2011,SAZH-2006}.

In the second part of the article, we encounter with the
ultra-parabolic equation, which involves a Dirac delta-function as
the source term. Such equations are called {\it impulsive
equations}. They also can be equivalently written as systems
consisting of the equation with zero source term and additional {\it
impulsive condition}. From the physical viewpoint, an impulsive
source (or condition) reflects phenomena of instantaneous loading,
i.e., drastic change of mass, energy, impulse, etc., at a moment.
Studies of impulsive ordinary differential equations have long
history and cover a wide range of topics in natural sciences (see,
for example, monographs
\cite{Bainov-1993,Laksh-1989,Samoilenko-1995} and references
therein). At the same time, the theory of impulsive partial
differential equations is rather new and far from complete.
Recently, some progress has been achieved in constructing
well-posedness theory for impulsive hyperbolic, elliptic, parabolic,
and abstract equations (see, for example, articles
\cite{Bainov-1995,Bainov-1996,H-Li-2011,Hern-Aki-2008,Hern-OReg-2013}
and references therein). In the present article, we develop the idea
from \cite{V-2002}. Namely, we deduce the equation with the Dirac
delta-function-type source from the equation with the smooth
distributed source. In other words, we fulfill and rigorously
justify the limiting passage from the delayed ultra-parabolic
equation to the impulsive ultra-parabolic equation as the time of
delay tends to zero.

This paper develops and is essentially based on the results
\cite{Kuz-2015,Kuz-2016,Kuz-2017,Kuz-2018,Kuzn-Sazh-2017,Kuzn-Sazh-1-2018,Kuzn-Sazh-2-2018}
due to the authors.

\section{Formulation of the basic problem incorporating\\ a distributed source} \label{Pi-gamma}
Let $\Omega$ be a bounded domain of spatial variables
$\boldsymbol{x}\in\mathbb{R}^d$ with a smooth boundary
$\partial\Omega$ ($\partial\Omega\in C^2$). Let $t\in[0,T]$ and
$s\in[0,S]$ be two independent time-like variables. Here $T$ and $S$
are given positive constants. Denote
\begin{eqnarray*}
 & & \displaystyle G_{T,S}:=\Omega\times(0,T)\times(0,S),\\[1ex]
 & & \displaystyle \Xi^1:=\overline{\Omega}\times[0,S], \quad \Xi^2:=\overline{\Omega}\times[0,T],\\[1ex]
& & \displaystyle  \Gamma_l:=\partial\Omega\times[0,T]\times[0,S],\\[1ex]
& & \displaystyle \Gamma_0^1:=\overline{\Omega}\times\{t=0\}\times[0,S], \quad
\Gamma_T^1:=\overline{\Omega}\times\{t=T\}\times[0,S],\\[1ex]
& & \displaystyle \Gamma_0^2:=\overline{\Omega}\times[0,T]\times\{s=0\},\quad \Gamma_S^2:=\overline{\Omega}\times[0,T]\times\{s=S\}.
\end{eqnarray*}
 In this article, we study the following Cauchy --- Dirichlet problem for the ultra-parabolic equation with partial diffusivity and a source term.\\[1ex]
\noindent \textbf{Problem} $\boldsymbol{\Pi}_{\boldsymbol{\gamma}}$. \emph{It is necessary to
find a function $u$: $G_{T,S}\mapsto\mathbb{R}$ satisfying
the quasi-linear ultra-parabolic Kolmogorov-type equation
\begin{subequations}  \label{e1.01}
\begin{equation}
\partial_tu+\partial_sa(u)+{\rm div}_x\boldsymbol{\varphi}(u)=
\Delta_xu+Z_\gamma(\boldsymbol{x},t,s,u),\quad(\boldsymbol{x},t,s)\in
G_{T,S},
\label{e1.01a}%
\end{equation}
the initial condition with respect to time-like variable $t$
\begin{equation}
u\vert_{t=0}=u_0^{(1)}(\boldsymbol{x},s),
\quad (\boldsymbol{x},s)\in\Xi^1,
\label{e1.01b}%
\end{equation}
the initial and final conditions with respect to time-like variable $s$
\begin{equation}
u\vert_{s=0}\approx u_0^{(2)}(\boldsymbol{x},t),\quad
u\vert_{s=S}\approx u_{S}^{(2)}(\boldsymbol{x},t),
\quad (\boldsymbol{x},t)\in\Xi^2,
\label{e1.01c}%
\end{equation}
and the homogeneous boundary condition
\begin{equation}
u\vert_{\Gamma_l}=0.
\label{e1.01d}%
\end{equation}
\end{subequations}}
In the formulation of Problem $\Pi_\gamma$, we suppose that the
initial and final data $u_0^{(1)}$, $u_0^{(2)}$, $u_S^{(2)}$, the
    nonlinearities $a=a(\lambda)$, $\boldsymbol{\varphi}=(\varphi_1(\lambda),\ldots,\varphi_d(\lambda))$, and
the source term $Z_\gamma=Z_\gamma(\boldsymbol{x},t,s,\lambda)$ are
given and satisfy the conditions stated further.

In \eqref{e1.01c} the relation sign $\approx$ means that $u_0^{(2)}$
and $u_S^{(2)}$ may be unattained by a solution $u$ on some parts of
the sets $\Gamma_0^2$ and $\Gamma_S^2$, respectively. The fact
whether $\approx$ becomes equality $(=)$, or not, is figured out
\textbf{\emph{a posteriori}}, i.e., after a solution of equation \eqref{e1.01a} is constructed somehow.\\[1ex]
\noindent{\bf Conditions on $\boldsymbol{u}_{\boldsymbol{0}}^{{\bf (1)}} \boldsymbol{\&}\boldsymbol{u}_{\boldsymbol{0}}^{\bf (2)}\boldsymbol{\&} \boldsymbol{u}_{\boldsymbol{S}}^{\bf (2)}$.} The
initial and final data meet the regularity requirements
\begin{equation}
\label{u1.01}%
u_0^{(1)}\in C^{2+\alpha}(\Xi^1),\quad u_0^{(2)},u_S^{(2)}\in
C^{2+\alpha}(\Xi^2),\quad\alpha\in(0,1),
\end{equation}
and the consistency conditions
\begin{equation}
\label{u1.02}%
u_0^{(1)}=0\text{ in a neighborhood of }\partial\Xi^1,
\end{equation}
\begin{equation}
\label{u1.03}%
u_0^{(2)}=0,\,u_S^{(2)}=0\text{ in a neighborhood of
}\partial\Xi^2.
\end{equation}
By $C^{2+\alpha}(\overline{\mathcal{O}})$
($\overline{\mathcal{O}}\subset\mathbb{R}^N$) we standardly denote
the space of twice-differentiable functions, whose second
derivatives are H\"{o}lder-continuous with exponent
$\alpha\in(0,1)$, equipped with the norm
\[
\left\Vert\Phi\right\Vert_{C^{2+\alpha}(\overline{\mathcal{O}})}=
\left\Vert\Phi\right\Vert_{C^2(\overline{\mathcal{O}})}+\sum\limits_{\vert\boldsymbol{k}\vert=2}
\sup\limits_{{\,}^{\boldsymbol{\zeta},\boldsymbol{\eta}\in\mathcal{O}}_{\boldsymbol{\zeta}\neq\boldsymbol{\eta}}}
\frac{\left\vert
D^{\boldsymbol{k}}\Phi(\boldsymbol{\zeta})-D^{\boldsymbol{k}}\Phi(\boldsymbol{\eta})
\right\vert}{\left\vert\boldsymbol{\zeta}-\boldsymbol{\eta}\right\vert^\alpha}.
\]

\noindent{\bf Conditions on $\boldsymbol{a}\boldsymbol{\&}\boldsymbol{\varphi}\boldsymbol{\&}\boldsymbol{Z}_{\boldsymbol{\gamma}}$.}
{\bf (i)} Functions $a,Z_\gamma$ and $\varphi_i$
$(i=1,\ldots,d)$ meet the regularity requirements
\begin{equation}
\label{u1.04}%
a\in C_{\mathrm{loc}}^2(\mathbb R),\quad a(0)=0,\quad Z_\gamma\in
L^\infty(0,T;C_{\mathrm{loc}}^1(\Xi^1\times\mathbb{R})),\quad \varphi_i\in
C_{\mathrm{loc}}^2(\mathbb{R}),\quad \varphi_i(0)=0.
\end{equation}
{\bf (ii)}
Function $a$ satisfies {\it the genuine nonlinearity
condition}:
\begin{equation}
\label{e1.02}%
{\rm
meas}\left\{\lambda\in\mathbb{R}\!:\,\,\xi_1+a'(\lambda)\xi_2=0\right\}=0 \quad \mbox{for every fixed } (\xi_1,\xi_2)\in\mathbb{S}^{1}.
\end{equation}
{\bf (iii)}  Function $Z_\gamma$ satisfies the following growth condition:\\[1ex]
there exist constants $b_\gamma^{(1)},b_\gamma^{(2)}\geqslant0$ such
that for a.e. $(\boldsymbol{x},t,s)\in G_{T,S}$ and
$\lambda\in\mathbb{R}$ the inequality
\begin{equation}
\label{u1.05}%
\lambda Z_\gamma(\boldsymbol{x},t,s,\lambda)\leqslant
b_\gamma^{(1)}\lambda^2+ b_\gamma^{(2)}
\end{equation}
holds.\\[1ex]
\indent In item (ii) and further in the article, by $\mathbb{S}^{1}$ we
denote the unit circle in $\mathbb{R}^{2}$ centered at the origin,
$\mathbb{S}^{1}:=\{(\xi_1,\xi_2)\in\mathbb{R}^{2}\!:\,\xi_1^2+\xi_2^2=1\}$, and by ${\rm meas}\, {\mathcal Q}$ we denote the Lebesgue measure of any Lebesgue-measurable set $\mathcal Q$.

\begin{remark}
\label{rem.1.1}%
According to the existing theory of genuinely nonlinear
ultra-parabolic equations \cite{Panov2009}, condition \eqref{e1.02}
can be generalized in the following way:
$$\mbox{set } \left\{\lambda\in\mathbb{R}\!:\,\xi_1+a'(\lambda)\xi_2=0\right\}
\mbox{ has empty interior for each } (\xi_1,\xi_2)\in\mathbb{S}^1.$$
\end{remark}

\begin{remark}
\label{rem.1.2}%
According to the well-known theory of quasi-linear parabolic
equations \cite[Chapter 1, Section 2]{LSU-1968}, the growth condition
\eqref{u1.05} along with the regularity demands \eqref{u1.01} provides the
maximum principle for a solution (if any) of Problem $\Pi_\gamma$.
  \end{remark}

\begin{remark}
\label{rem.1.3}%
Note that, with $u\big\vert_{s=0}=u_0^{(2)}$ and
$u\big\vert_{s=S}=u_S^{(2)}$ on $\Xi^2$ on the place of
\eqref{e1.01c}, Problem $\Pi_\gamma$ becomes ill-posed. Indeed, since
function $a=a(\lambda)$ is nonlinear and, in general,
non-monotonous, it may be impossible to equate a solution $u$ of
Problem $\Pi_\gamma$  to $u_0^{(2)}$ and $u_S^{(2)}$ on the
\textbf{entire} sets $\Gamma_0^2$ and $\Gamma_S^2$. Therefore, we
permit that  a possible weak solution of Problem $\Pi_\gamma$ may
deviate from $u_0^{(2)}$ and $u_S^{(2)}$ on $\Gamma_0^2$ and
$\Gamma_S^2$, respectively. We set up a more loose non-classical
condition \eqref{e1.01c} following the original ideas presented in
\cite{BLN-1979,O-1996}, see also
\cite[Sections 2.6--2.8]{MNRR-1996}.
  \end{remark}

In the formulation of Problem $\Pi_\gamma$ and in Conditions on
$a\&\boldsymbol\varphi\& Z_\gamma$ above, the label `$\gamma$' is
`dumb' so far.
Our first main result in this article consists of the proof of
existence, uniqueness and stability of entropy and kinetic solutions to Problem
$\Pi_\gamma$. Thus we generalize the results established in
\cite{Kuz-2017} onto the case of equations with the source term. In
order to do this, we introduce into considerations and
systematically study a strictly parabolic regularized formulation.

After this, in Sections \ref{SingLim}--\ref{Limited} we consider the case when the
source term has the specific form
\begin{equation}
\label{u1.06}%
Z_\gamma(\boldsymbol{x},t,s,\lambda)=K_\gamma(t,\tau)\beta(\boldsymbol{x},s,\lambda),
  \end{equation}
where $\tau\in(0,T)$ is a given fixed value and functions $K_\gamma$
and $\beta$ satisfy the following requirements.\\[1ex]
\noindent{\bf Conditions on $\boldsymbol{K}_{\boldsymbol{\gamma}}\boldsymbol{\&}\boldsymbol{\beta}$.}
{\bf (i)} Function $K_\gamma$: $\mathbb{R}^2\mapsto\mathbb{R}^+$ is
defined by the formula
\begin{subequations}
\begin{equation}
\label{u1.07i}%
K_\gamma(t,\tau)=\mathbf{1}_{(t\leqslant\tau)}\, \frac{2}{\gamma}\,\omega\left(\frac{t-\tau}{\gamma}\right)
\quad(\gamma>0),
\end{equation}
where $\omega$: $\mathbb{R}\to\mathbb{R}^+$ is a standard regularizing
kernel having the properties
\begin{equation}
\label{u1.07ii}%
\omega\in C_0^\infty(\mathbb{R}),\quad\omega(t)\geqslant 0,\quad
\omega(-t)=\omega(t)\quad\forall\, t\in\mathbb{R},
  \end{equation}
\begin{equation}
\label{u1.07iii}%
\mathrm{supp\,}\omega\subset[-1,1],\quad\int\limits_{\mathbb{R}}\omega(t)\,dt=1.
  \end{equation}
{\bf (ii)} Function
$\beta$: $\Xi^1\times\mathbb{R}_\lambda\mapsto\mathbb{R}$ belongs to
the space  $C_0^1(\Xi^1\times\mathbb{R}_\lambda)$ and
satisfies the growth condition:
\begin{equation}
\label{u1.07iv}%
\max\limits_{\Xi^1\times\mathbb{R}_\lambda}
\left\vert\partial_\lambda\beta(\boldsymbol{x},s,\lambda)\right\vert\leqslant
b_0=\mathrm{const}<+\infty,
  \end{equation}
and the localization condition:
\begin{equation}
\label{u1.07v}%
\begin{split}
& \beta=0 \mbox{ in some neighborhood of } \partial \Xi^1 \mbox{ for all } \lambda\in \mathbb{R},\\
& \mbox{there is } b_1=\mbox{const}>0 \mbox{ such that } \beta=0
\mbox{ for all } |\lambda| > b_1.
\end{split}
\end{equation}
    \end{subequations}

\begin{remark}
\label{rem.1.4}%
On the strength of item $\mathrm{(i)}$ in Conditions on
$K_\gamma\&\beta$, we easily deduce that
\begin{equation}
\label{u1.08}%
K_\gamma(\cdot,\tau)\underset{\gamma\to0+}{\longrightarrow}\delta_{(t=\tau-0)}\text{
weakly}^*\text{ in }\mathcal{M}(\mathbb{R}),
  \end{equation}
i.e.,
\begin{equation}
\label{u1.09}%
\lim\limits_{\gamma\to0+}\int\limits_{\mathbb{R}}\phi(t)K_\gamma(t,\tau)\,dt=\phi(\tau-0)
  \end{equation}
for any integrable in a neighborhood of
$\{t=\tau\}\subset\mathbb{R}$ function $\phi$ having the trace at the point
$t=\tau$ from the left:
\[
\phi(\tau-0)=\lim\limits_{t\to\tau-0}\phi(t).
\]
\end{remark}

In \eqref{u1.08} and further in the article,
$\mathcal{M}(\mathbb{R})$ is the normed space of Radon measures and
$\delta_{(t=\tau)}$ is the Dirac delta-function on $\mathbb{R}$
concentrated at the point $t=\tau$.

As an example of proper $\omega$ in \eqref{u1.07i}, we can take the
classical regularizing kernel
\[
\omega(t)=\left\{
\begin{array}{lll} \displaystyle
C_*^{-1}\exp\left(\frac{-1}{1-t^2}\right) &\text{ for
}&0\leqslant|t|<1,
\\
0 &\text{ for }&|t|\geqslant1,
  \end{array}
\right.
\]
where $C_*=\mathrm{const}>0$ ($C_*\approx 2,2522$) is such that
condition \eqref{u1.07iii}$_2$ holds.

\begin{remark}
\label{rem.1.5}%
 Using the Lagrange mean value theorem and \eqref{u1.07iv}, we find
\begin{equation}
\label{u1.10}%
\beta(\boldsymbol{x},s,\lambda)\lambda=\left(\partial_{\lambda'}\beta(\boldsymbol{x},s,\lambda')
\big\vert_{\lambda'=\theta\lambda}\right)\lambda^2-\beta(\boldsymbol{x},s,0)\lambda\leqslant
b_0\lambda^2-\beta(\boldsymbol{x},s,0)\lambda,\;\forall\,(\boldsymbol{x},s)\in\Xi^1,\,\forall\,\lambda\in\mathbb{R},
  \end{equation}
where $\theta\in(0,1)$. Using \eqref{u1.07i}--\eqref{u1.07iii}, by
rather simple evaluation from \eqref{u1.10} we deduce that the
source term $Z_\gamma=K_\gamma\beta$ satisfies the growth condition
\eqref{u1.05} with
$$\beta_\gamma^{(1)}=2\gamma^{-1}
\|\omega\|_{C[-1,1]}\left(b_0+2^{-1}\|\beta(\cdot,\cdot,0)\|_{C(\Xi^1)}\right),\quad
\beta_\gamma^{(2)}=\gamma^{-1}\Vert\omega\Vert_{C[-1,1]}\|\beta(\cdot,\cdot,0)\|_{C(\Xi^1)}.$$

Thus, Conditions on $K_\gamma\&\beta$ are consistent with Conditions
on $a\&\boldsymbol{\varphi}\& Z_\gamma$.
  \end{remark}

Our second main result in this article consists of limiting passage as $\gamma\to0+$
in Problem $\Pi_\gamma$, with $Z_\gamma$ of the form \eqref{u1.06}. We prove that the family $\{u_\gamma\}$ of kinetic
solutions to Problem $\Pi_\gamma$ has the unique limit $u_*\in
L^\infty(G_{T,S})\cap
L^2((0,T)\times(0,S);\text{\it\r{W}}{}_2^{\,1}(\Omega))$ in
$L^1$-strong sense as $\gamma\to0+$, and that this limiting
function $u_*$ is the unique kinetic and entropy solution of the
limiting problem. This limiting problem is formulated in Section
\ref{SingLim} further. It is called {\it Problem $\Pi_0$}.

\section{Parabolic Regularization}
\label{sec.2}

We construct a kinetic solution of Problem $\Pi_\gamma$ as a
singular limit of the family of classical solutions $u_\varepsilon$ of the following
strictly parabolic model.\\[1ex]
\noindent{\bf Problem $\boldsymbol{\Pi}_{\boldsymbol{\gamma}\boldsymbol{\varepsilon}}$.} \emph{For
arbitrarily given boundary data satisfying Conditions on
$u_0^{(1)}\&u_0^{(2)}\&u_{S}^{(2)}$, it is necessary to find a
function $u_\varepsilon$: $G_{T,S}\mapsto\mathbb{R}$ satisfying
\begin{subequations}
\label{e2.01}
the quasi-linear parabolic equation
\begin{equation}
\label{e2.01a}
\partial_{t}u_{\varepsilon}+\partial_{s}a(u_\varepsilon)+
{\rm div}_x\boldsymbol{\varphi}(u_\varepsilon)=\Delta_x
u_\varepsilon
+\varepsilon\partial_{ss}^2u_\varepsilon+Z_\gamma(\boldsymbol{x},t,s,u_\varepsilon),\quad
(\boldsymbol{x},t,s)\in G_{T,S},
\end{equation}
the initial condition
\begin{equation}
u_\varepsilon\vert_{t=0}=u_0^{(1)}(\boldsymbol{x},s),
\,\,(\boldsymbol{x},s)\in\Xi^1,
\label{e2.01b}%
\end{equation}
the initial and final conditions
\begin{equation}
u_\varepsilon\vert_{s=0}=u_0^{(2)}(\boldsymbol{x},t),\,\,
u_\varepsilon\vert_{s=S}=u_S^{(2)}(\boldsymbol{x},t),
\,\,(\boldsymbol{x},t)\in\Xi^2,
\label{e2.01c}%
\end{equation}
and the homogeneous boundary condition
\begin{equation}
u_\varepsilon\vert_{\Gamma_l}=0.
\label{e2.01d}%
\end{equation}
\end{subequations} Here $\varepsilon\in(0,1]$ is an arbitrarily fixed
small parameter. }

\begin{proposition}
\label{prop.2.0}%
Whenever Conditions on $u_0^{(1)}\&u_0^{(2)}\&u_S^{(2)}$ and
Conditions on $a\&\boldsymbol{\varphi}\&Z_\gamma$ hold, for any
fixed $\varepsilon>0$ there exists the unique classical solution
$u_\varepsilon=u_\varepsilon(\boldsymbol{x},t,s)$ of Problem
$\Pi_{\gamma\varepsilon}$ such that $u_\varepsilon\in
H^{\alpha',\frac{\alpha'}{2}}(\overline{G}_{T,S})\cap
H^{2+\alpha'',1+\frac{\alpha''}2}(G_{T,S})$, where
$\alpha',\alpha''\in(0,1)$ depend on $\alpha$ and $\varepsilon$.

Moreover, the maximum principle
\begin{multline}
\label{e2.03}%
\!\!\!\!\!\!\left\Vert u_{\varepsilon}\right\Vert
_{L^{\infty}(\Xi_1\times(0,t'))}\leqslant \\
\inf\limits_{\xi>b_\gamma^{(1)}} \left(e^{\xi t'}\max\left\{\Vert
u_{0}^{(1)}\Vert_{L^{\infty}(\Xi^1)}, \Vert
u_{0}^{(2)}\Vert_{L^{\infty}(\Omega\times(0,t'))},\Vert
u_{S}^{(2)}\Vert_{L^{\infty}(\Omega\times(0,t'))},
\sqrt{\frac{b_\gamma^{(2)}}{\xi-b_\gamma^{(1)}}}\,\right\}\right)\overset{\mathrm{def}}{=}
\\
M\left(t',\|u_0^{(1)}\|_{L^\infty(\Xi^1)},\|u_0^{(2)}\|_{L^\infty(\Omega\times(0,t'))}
\|u_S^{(2)}\|_{L^\infty(\Omega\times(0,t'))}\right)\leqslant
M\Big\vert_{t'=T}\equiv M_0<+\infty,\\ \forall\, t'\in(0,T],
  \end{multline}
and the energy estimate
\begin{equation}
\label{e2.04}%
\int\limits_{G_{T,S}}\left(  \left\vert \nabla_x
u_{\varepsilon}\right\vert^2+\varepsilon\left\vert
\partial_{s}u_{\varepsilon}\right\vert^2\right)
~d\boldsymbol{x}dtds\leqslant C_1
\end{equation}
hold. The constant $C_1$ does not depend on $\varepsilon$.
\end{proposition}

In the formulation of Proposition \ref{prop.2.0}, by
$H^{\alpha',\frac{\alpha'}2}(\overline{G}_{T,S})$ we standardly
denote the spaces of H\"{o}lder-continuous functions of $(\boldsymbol{x},s)$ and $t$, with exponents $\alpha'$ and
$\frac{\alpha'}2$, respectively. By
$H^{2+\alpha'',\,1+\frac{\alpha''}2}(G_{T,S})$ we denote the space of
differentiable functions $\phi$ on $G_{T,S}$ such that
$\partial_t\phi$, $\partial_{x_i}\phi$, $\partial_s\phi$,
$\partial_{x_ix_j}^2\phi$, $\partial_{ss}^2\phi$,
$\partial_{sx_i}^2\phi\in H^{\alpha'',\frac{\alpha''}2}(G_{T,S})$,
$i,j=1,\ldots,d$.\\[1ex]
{\it Proof} of Proposition \ref{prop.2.0}
directly follows from the well-known theory of quasi-linear parabolic
equations of the second order \cite[Chapter 1, Theorem 2.9; Chapter 5,
Theorem 6.2]{LSU-1968}. \qed\\[1ex]
\indent In view of the forthcoming consideration of impulsive equation in
Sec. \ref{SingLim}--\ref{Limited}, it is suitable also to introduce
the notion of weak solutions to Problem $\Pi_{\gamma\varepsilon}$
for the case of weaker restrictions on initial and final
data than the ones that take place in Conditions on
$u_0^{(1)}\&u_0^{(2)}\&u_S^{(2)}$.

Let us suppose that the initial and final data meet the weakened
regularity requirements
\begin{equation}
\label{u1.01w}%
u_0^{(1)}\in L^\infty(\Xi^1),\quad u_0^{(2)},u_S^{(2)}\in
C^{2+\alpha}(\Xi^2),\quad\alpha\in(0,1),
  \end{equation}
and the consistency conditions \eqref{u1.03}.

Let $\hat{u}\in L^\infty(G_{T,S})\cap C(0,T;W_2^1(\Xi^1))$ be an
arbitrary extension of $u_0^{(1)}$, $u_0^{(2)}$ and $u_S^{(2)}$ into
$G_{T,S}$ such that $\hat{u}\vert_{\Gamma_l}=0$.

\begin{remark}
\label{rem.2.01}%
Such $\hat{u}$ exists due to the well-known Inverse Trace Theorem
(see, for example, \cite[Theorem 5.1]{Lions-Mag-1961} or \cite[Theorem
2.22]{Steinbach-2008}) and requirements \eqref{u1.01w} and
\eqref{u1.03}.
\end{remark}

\begin{definition}
\label{def.2.01}%
 Function $u_{\varepsilon}\in L^\infty(G_{T,S})\cap
L^2(0,T;\text{\it\r{W}}{}_2^{\,1}(\Xi^1))$ is called a weak solution of Problem
$\Pi_{\gamma\varepsilon}$, if it satisfies the following demands.
\begin{enumerate}
\item[1.]%
The equality $u_\varepsilon-\hat{u}=0$ holds on
$\Gamma_0^1\cup\Gamma_0^2\cup\Gamma_S^2\cup\Gamma_l$ in the trace
sense.
\item[2.]%
The integral equality
\begin{equation}
\label{e2.02}%
\int\limits_{G_{T,S}}\big(-u_\varepsilon\partial_t\phi-a(u_\varepsilon)\partial_s\phi-
\boldsymbol{\varphi}(u_\varepsilon)\cdot\nabla_x\phi+
\nabla_xu_\varepsilon\cdot\nabla_x \phi
+\varepsilon\partial_{s}u_\varepsilon\partial_{s}\phi
-Z_\gamma(\boldsymbol{x},t,s,u_\varepsilon)\phi\big)\,d\boldsymbol{x}dtds=0
  \end{equation}
holds for every $\phi\in
L^{\infty}(G_{T,S})\cap\text{\it\r{W}}{}_2^{\,1}(G_{T,S})$.
\end{enumerate}
\end{definition}

\begin{proposition}
\label{prop.2.1}%
Whenever $u_0^{(1)}$, $u_0^{(2)}$ and $u_S^{(2)}$ satisfy the
requirements \eqref{u1.01w} and \eqref{u1.03}, and Conditions on
$a\&\boldsymbol{\varphi}\&Z_\gamma$ hold, for any fixed
$\varepsilon\in(0,1]$ there exists the unique weak solution
$u_\varepsilon=u_\varepsilon(\boldsymbol{x},t,s)$ of Problem
$\Pi_{\gamma\varepsilon}$ such that $u_\varepsilon\in
W_2^{2,1}(G')$, where $G'$ is an arbitrary strictly interior
subdomain  of $G_{T,S}$ with a smooth boundary.

Moreover, the maximum principle \eqref{e2.03} and the energy estimate
\eqref{e2.04} hold and $u_\varepsilon$ possesses the additional
regularity property: $u_\varepsilon\in
H^{\alpha',\frac{\alpha'}2}(G_{T,S}\cup\Xi_*^2)$, where $\Xi_*^2$ is
an arbitrary closed subset of $\Xi^2$ such
that $\Xi_*^2\cap\{t=0\}=\emptyset$ and $\partial\,\Xi_*^2$ is smooth.
\end{proposition}

In the formulation of Proposition \ref{prop.2.1}, exponent
$\alpha'\in(0,1)$ depends on $\alpha$ and $\varepsilon$.  By
$W_2^{2,1}(G')$ we standardly denote the Sobolev space of
measurable integrable functions $\phi$ on $G'$ such
that $\partial_t\phi$, $\partial_{x_i}\phi$, $\partial_s\phi$,
$\partial_{x_ix_j}^2\phi$, $\partial_{ss}^2\phi$,
$\partial_{sx_i}^2\phi\in L^2(G')$, $i,j=1,\ldots,d$.\\[1ex]
\noindent {\emph{Proof}} of Proposition \ref{prop.2.1}
follows from \cite[Chapter 1, Theorem 2.9; Chapter 5, Theorem
1.1]{LSU-1968}.\qed
\begin{remark} \label{classic-weak} Clearly, a classical solution of Problem $\Pi_{\gamma\varepsilon}$ is a weak solution in the sense of Definition \ref{def.2.01}.
\end{remark}

\section{Notions of kinetic and entropy solutions\\ to Problem $\boldsymbol{\Pi}_{\boldsymbol{\gamma}}$}
\label{sec.3}%
In this section we set up notions of kinetic and entropy
solutions to Problem $\Pi_\gamma$. To this end, let us introduce a few necessary concepts firstly.

Set
\begin{equation}
\label{chi-f}%
\chi(\lambda;\upsilon)= \left\{\begin{array}{rl} 1& \text{for }
0<\lambda<\upsilon,
\\
-1& \text{for }\upsilon<\lambda<0,
\\
0& \text{otherwise.}
\end{array}\right.
\end{equation}

\begin{definition}
\label{def.3.01}%
 Let $N\in\mathbb{N}$, $L>0$, and $\mathcal{O}$ be an open
set in  $\mathbb{R}^N$. Let $h\in
L^\infty(\mathcal{O}\times(-L,L))$  satisfy the inequality
$$0\leqslant h(\boldsymbol{z},\lambda)\mathrm{sgn}(\lambda)\leqslant1 \mbox{ for a.e. }
(\boldsymbol{z},\lambda)\in\mathbb{R}^{N+1}.$$
We say that $h$ is a
$\chi$-function, if there exists a function $\upsilon\in
L^\infty(\mathcal{O})$ such that
\begin{equation*}
h(\boldsymbol{z},\lambda)=\chi(\lambda;\upsilon(\boldsymbol{z}))\text{
for a.e. }\boldsymbol{z}\in\mathcal{O}.
\end{equation*}
\end{definition}

\begin{lemma}
\label{lem.3.01}%
(\cite[Lemma 2.1.1]{Per-2002}.) The following two identities hold true
(say, for $\Psi$ locally Lipschitz continuous, i.e. $\Psi'\in L_{\rm
loc}^\infty(\mathbb{R})$):
\begin{enumerate}
\item[$\mathrm{(i)}$]
$\displaystyle \int\limits_{\mathbb{R}_\lambda}\Psi'(\lambda)\chi(\lambda;v)\,d\lambda=\Psi(v)-\Psi(0)$,
in particular,
$\displaystyle \int\limits_{\mathbb{R}_\lambda}\chi(\lambda;v)\,d\lambda=v$;
\item[$\mathrm{(ii)}$]
$\displaystyle \int\limits_{\mathbb{R}_\lambda}\Psi'(\lambda)
\left\vert\chi(\lambda;v)-\chi(\lambda;w)\right\vert\,d\lambda={\rm
sgn}(v-w)(\Psi(v)-\Psi(w))$, in particular,\\
$\displaystyle \int\limits_{\mathbb{R}_\lambda}\left\vert\chi(\lambda;v)-\chi(\lambda;w)\right\vert\,d\lambda
=\left\vert v-w\right\vert$.
\end{enumerate}
Additionally,
\begin{enumerate}
\item[$\mathrm{(iii)}$] $\left\vert\chi(\lambda;v)-\chi(\lambda;w)\right\vert=
\left\vert\chi(\lambda;v)-\chi(\lambda;w)\right\vert^2$.
\end{enumerate}
\end{lemma}

The following lemma establishes the link between sequences of
$\chi$-functions and their limits.
\begin{lemma}
\label{lem.3.02} (\cite{V-2001}.) Let $\mathcal{O}$ be an open set of
$\mathbb{R}^N$ and $h_n\in L^\infty(\mathcal{O}\times(-L,L))$ be a
sequence of $\chi$-functions converging weakly* to $h\in
L^\infty(\mathcal{O}\times(-L,L))$. Set
$$\upsilon_n(\cdot)=\int\limits_{-L}^Lh_n(\cdot,\lambda)\,d\lambda,\quad
\upsilon(\cdot)=\int\limits_{-L}^L h(\cdot,\lambda)\,d\lambda.$$

Then the three assertions are equivalent to each other:
\begin{itemize}
\item $h_n$ converges strongly to $h$ in $L^1_{\rm loc}(\mathcal{O}\times(-L,L))$,
\item $\upsilon_n$ converges strongly to $\upsilon$ in $L^1_{\rm loc}(\mathcal{O})$,
\item $h$ is a $\chi$-function.
\end{itemize}
\end{lemma}

In order to study topics regarding attainability of initial and final data, introduce the definition of one-sided essential limits.
\begin{definition} \label{def.esslim}
Let $\phi$: $[0,Y]\mapsto \RR$ be a measurable function, $Y={\rm const}>0$. We say that $K\in \RR$ is {\rm the essential limit from the left} of $\phi$ at a point $y_0\in (0,Y]$ and write $\displaystyle K=\underset{y\to y_0-0}{\mbox{\rm ess\,lim\,}} \phi(y)$, if there exists a set $E\subset (0,y_0)$ of full Lebesgue measure such that $\displaystyle \lim\limits_{{}^{y\to y_0}_{y\in E}} |\phi(y)-K|=0$.

Analogously, we say that $K\in \RR$ is {\rm the essential limit from the right} of $\phi$ at a point $y_0\in [0,Y)$ and write $\displaystyle K=\underset{y\to y_0+0}{\mbox{\rm ess\,lim\,}} \phi(y)$, if there exists a set $E\subset (y_0,Y)$ of full Lebesgue measure such that $\displaystyle \lim\limits_{{}^{y\to y_0}_{y\in E}} |\phi(y)-K|=0$.
\end{definition}

Now we are in a position to set up notions of kinetic and entropy solutions to Problem $\Pi_\gamma$.
\begin{definition}
\label{def.3.02}%
\begin{subequations}
Function $u\in L^\infty(G_{T,S})\cap
L^2((0,T)\times(0,S);\text{\it\r{W}}{}_2^{\,1}(\Omega))$ is called a
kinetic solution of Problem $\Pi_\gamma$, if it satisfies the
 kinetic equation
\begin{multline}
\label{e3.01a}%
\partial_{t}\chi(\lambda;u(\boldsymbol{x},t,s))
+a^{\prime}(\lambda)\partial_{s}\chi(\lambda;u(\boldsymbol{x},t,s))
+\\ \boldsymbol{\varphi}^{\prime}(\lambda)\cdot\nabla_x\chi(\lambda;u(\boldsymbol{x},t,s))
+Z_\gamma(\boldsymbol{x},t,s,\lambda)\partial_\lambda\chi(\lambda;u(\boldsymbol{x},t,s))=
\\
\Delta_x\chi(\lambda;u(\boldsymbol{x},t,s))+\delta_{(\lambda=0)}
Z_\gamma(\boldsymbol{x},t,s,\lambda)+\partial_\lambda(m(\boldsymbol{x},t,s,\lambda)+n(\boldsymbol{x},t,s,\lambda)),\\
(\boldsymbol{x},t,s,\lambda) \in G_{T,S}\times [-M_0,M_0],
\end{multline}
 the kinetic initial condition
\begin{equation}
\label{e3.01b}%
\underset{t\to0+}{\mathrm{ess\,lim}}\int\limits_{-M_0}^{M_0}\int\limits_{\Xi^1}
\left\vert\chi(\lambda;u(\boldsymbol{x},t,s))-
\chi(\lambda;u_0^{(1)}(\boldsymbol{x},s))\right\vert\,d\boldsymbol{x}dsd\lambda=0,
\end{equation}
the limiting relations
\begin{equation}
\label{e3.02}%
\underset{s\to0+}{\rm ess\,lim}\int\limits_{\Xi^2}\left\vert
u(\boldsymbol{x},t,s)-u_0^{{\rm
tr},(2)}(\boldsymbol{x},t)\right\vert\,d\boldsymbol{x}dt=0,
  \end{equation}
\begin{equation}
\label{e3.03}%
\underset{s\to S-0}{\rm ess\,lim}\int\limits_{\Xi^2}\left\vert
u(\boldsymbol{x},t,s)-u_S^{{\rm
tr},(2)}(\boldsymbol{x},t)\right\vert\,d\boldsymbol{x}dt=0,
  \end{equation}
with some functions $u_0^{{\rm tr},(2)}, u_S^{{\rm tr},(2)} \in L^\infty(\Xi^2)$,
and the kinetic boundary conditions
\begin{multline}
\label{e3.01c}%
a^{\prime}(\lambda)\big(\chi(\lambda;u_0^{{\rm
tr},(2)}(\boldsymbol{x},t))-
\chi(\lambda;u_0^{(2)}(\boldsymbol{x},t))\big)
\\
-\delta_{(\lambda=u_0^{(2)}(\boldsymbol{x},t))} \big(a(u_0^{{\rm
tr},(2)}(\boldsymbol{x},t)) -a(u_0^{(2)}(\boldsymbol{x},t))\big)=
\partial_\lambda\mu_0^{(2)}(\boldsymbol{x},t,\lambda),\\
(\boldsymbol{x},t,\lambda) \in \Xi^2 \times [-M_0,M_0],
  \end{multline}
\begin{multline}
\label{e3.01d}%
a^{\prime}(\lambda)\big(\chi(\lambda;u_{S}^{{\rm
tr},(2)}(\boldsymbol{x},t))
-\chi(\lambda;u_{S}^{(2)}(\boldsymbol{x},t))\big)
\\
-\delta_{(\lambda=u_{S}^{(2)}(\boldsymbol{x},t))} \big(a(u_{S}^{{\rm
tr},(2)}(\boldsymbol{x},t)) -a(u_{S}^{(2)}(\boldsymbol{x},t))\big)=
-\partial_\lambda\mu_{S}^{(2)}(\boldsymbol{x},t,\lambda),\\
(\boldsymbol{x},t,\lambda) \in \Xi^2 \times [-M_0,M_0].
  \end{multline}
  \end{subequations}
\end{definition}
In \eqref{e3.01a}--\eqref{e3.01d} we have
$m\in\mathcal{M}^+(G_{T,S}\times[-M_0,M_0])$,
$\mu_{0}^{(2)},\mu_{S}^{(2)}\in\mathcal{M}^+(\Xi^2\times[-M_0,M_0])$,
$n=\delta_{(\lambda=u)}\left\vert\nabla_x u\right\vert^2$, where
$\delta_{(\lambda=u)}$ is the Dirac measure on $\mathbb{R}_\lambda$,
concentrated at the point\linebreak $\lambda=u(\boldsymbol{x},t,s)$, and $M_0$ is the constant defined in \eqref{e2.03}.
Correspondingly, in \eqref{e3.01a}, $\delta_{(\lambda=0)}$ is the
Dirac measure on $\mathbb{R}_\lambda$ concentrated at the origin. By
$\mathcal{M}^+$ we denote the space of finite positive Radon
measures. Functions $u_0^{{\rm tr},(2)}$ and $u_S^{{\rm tr},(2)}$
are $L^1(\Xi^2)$-strong traces of $u=u(\boldsymbol{x},t,s)$ on the planes $\{s=0\}$ and
$\{s=S\}$, respectively.

The kinetic equation \eqref{e3.01a} and the kinetic boundary
conditions \eqref{e3.01c} and \eqref{e3.01d} are understood in the
sense of distributions: see further integral equality \eqref{e4.12} in Section \ref{sec.4} and integral equality \eqref{e6.15} in Section \ref{sec.6}.

\begin{definition} \label{def.3.03}%
\begin{subequations}
Function $u\in L^\infty(G_{T,S})\cap
L^2((0,T)\times(0,S);\text{\it\r{W}}{}_2^{\,1}(\Omega))$ is called
an entropy solution of Problem $\Pi_\gamma$, if it satisfies the
entropy inequality
\begin{equation}
\label{e3.04a}%
\partial_{t}\eta(u)+\partial_{s}q_a(u)+{\rm
div}_x\boldsymbol{q}_\varphi(u)-Z_\gamma(\boldsymbol{x},t,s,u)\eta'(u)-\Delta_x\eta(u)\leqslant
-\eta^{\prime\prime}(u)|\nabla_x u|^2,
\end{equation}
the maximum principle
\begin{equation}
\label{e3.04b}%
\left\Vert u\right\Vert_{L^\infty(G_{T,S})}\leqslant
M_0<+\infty,
   \end{equation}
the initial condition
\begin{equation}
\label{e3.04c}%
\underset{t\to0+}{\mathrm{ess\,lim}}\int\limits_{\Xi^1} \left\vert
u(\boldsymbol{x},t,s)-u_0^{(1)}(\boldsymbol{x},s)\right\vert\,d\boldsymbol{x}ds=0,
\end{equation}
the limiting relations \eqref{e3.02} and \eqref{e3.03} with some functions $u_0^{{\rm tr},(2)}, u_S^{{\rm tr},(2)} \in L^\infty(\Xi^2)$,
and the entropy boundary conditions
\begin{equation}
\label{e3.04d}%
q_{a}(u_0^{{\rm tr},(2)}(\boldsymbol{x},t))
-q_{a}(u_0^{(2)}(\boldsymbol{x},t))
-\eta^{\prime}(u_0^{(2)}(\boldsymbol{x},t)) (a(u_0^{{\rm
tr},(2)}(\boldsymbol{x},t))
-a(u_0^{(2)}(\boldsymbol{x},t)))\leqslant0,
\quad(\boldsymbol{x},t)\in \Xi^2,
  \end{equation}
\begin{equation}
\label{e3.04e}%
q_a(u_{S}^{{\rm tr},(2)}(\boldsymbol{x},t))-
q_a(u_{S}^{(2)}(\boldsymbol{x},t))
-\eta^{\prime}(u_{S}^{(2)}(\boldsymbol{x},t)) (a(u_{S}^{{\rm
tr},(2)}(\boldsymbol{x},t))
-a(u_{S}^{(2)}(\boldsymbol{x},t)))\geqslant0,
\quad(\boldsymbol{x},t)\in \Xi^2.
  \end{equation}
In \eqref{e3.04a}, \eqref{e3.04d} and \eqref{e3.04e}, $\eta\in
C^2(\mathbb{R})$ is an arbitrary convex  test-function:
$\eta''(z)\geqslant 0$ $\forall z\in\mathbb{R}$, and
$(\eta,q_a,\boldsymbol{q}_\varphi)$ is a convex entropy flux triple:
\begin{equation*}
q_a^{\prime}(z)=a^{\prime}(z)\eta^{\prime}(z),\quad
\boldsymbol{q}_\varphi^{\prime}(z)=\boldsymbol{\varphi}^{\prime}(z)\eta^{\prime}(z),\quad
z\in\mathbb{R}.
  \end{equation*}
  \end{subequations}
\end{definition}
Entropy inequality \eqref{e3.04a} is understood in the sense of
distributions. Entropy boundary conditions \eqref{e3.04d} and
\eqref{e3.04e} hold a.e. in $\Xi^2$.

The following theorem on well-posedness of Problem $\Pi_\gamma$ is
the first main result of the article.
\begin{theorem} \label{theo.3.1}
{\bf 1.\;Existence, uniqueness and stability of kinetic solutions.} Under
Conditions on $u_0^{(1)}\&u_0^{(2)}\&u_S^{(2)}$ and Conditions on
$a\&\boldsymbol{\varphi}\& Z_\gamma$, Problem $\Pi_\gamma$ has the
unique kinetic solution $u=u(\boldsymbol{x},t,s)$ in the sense of
Definition \ref{def.3.02}.

Moreover, let $u_1$ and $u_2$ be two kinetic solutions corresponding
to two given sets of data
$(u_{1,0}^{(1)},u_{1,0}^{(2)},u_{1,S}^{(2)})$ and
$(u_{2,0}^{(1)},u_{2,0}^{(2)},u_{2,S}^{(2)})$, respectively. Then
the estimate
\begin{multline}
\label{e3.05}%
\|u_1(\cdot,t,\cdot)-u_2(\cdot,t,\cdot)\|_{L^1(\Xi^1)}\leqslant\\
e^{\mathfrak{G}_\gamma(t)}\Biggl[ \Vert
u_{1,0}^{(1)}-u_{2,0}^{(1)}\Vert_{L^1(\Xi^1)}+
\max\limits_{\lambda\in[-M_1(t),M_1(t)]}\vert a'(\lambda)\vert
\int\limits_0^t e^{-\mathfrak{G}_\gamma(t')} \Big(\Vert
u_{1,0}^{(2)}(\cdot,t')-u_{2,0}^{(2)}(\cdot,t')\Vert_{L^1(\Omega)}+\\
\Vert
u_{1,S}^{(2)}(\cdot,t')-u_{2,S}^{(2)}(\cdot,t')\Vert_{L^1(\Omega)}\Big)dt'
\Biggr], \quad \forall\, t\in(0,T],
\end{multline}
holds true, where
\begin{equation}
\label{e3.06}%
\mathfrak{G}_\gamma(t)=\int\limits_0^t\max\limits_{(\boldsymbol{x},s,\lambda)\in\Xi^1\times[-M_1(t),M_1(t)]}|\partial_\lambda
Z_\gamma(\boldsymbol{x},t',s,\lambda)|dt',
  \end{equation}
\begin{multline}
\label{e7.01bis}%
M_1(t)=\max\bigl\{M(t,\|u_{1,0}^{(1)}\|_{L^\infty(\Xi^1)},\|u_{1,0}^{(2)}\|_{L^\infty(\Omega\times(0,t))},
\|u_{1,S}^{(2)}\|_{L^\infty(\Omega\times(0,t))}),\\
M(t,\|u_{2,0}^{(1)}\|_{L^\infty(\Xi^1)},
\|u_{2,0}^{(2)}\|_{L^\infty(\Omega\times(0,t))},
\|u_{2,S}^{(2)}\|_{L^\infty(\Omega\times(0,t))})\bigr\},
  \end{multline}
and $M$ is defined in \eqref{e2.03}.\\[1ex]
{\bf 2.\;Existence, uniqueness and stability of entropy solutions.}
Under Conditions on $u_0^{(1)}\&u_0^{(2)}\&u_S^{(2)}$ and Conditions
on $a\&\boldsymbol{\varphi}\& Z_\gamma$, Problem $\Pi_\gamma$ has the
unique entropy solution\linebreak $u=u(\boldsymbol{x},t,s)$ in the sense of
Definition \ref{def.3.03}.

Moreover, let $u_1$ and $u_2$ be two entropy solutions corresponding
to two given sets of data
$(u_{1,0}^{(1)},u_{1,0}^{(2)},u_{1,S}^{(2)})$ and
$(u_{2,0}^{(1)},u_{2,0}^{(2)},u_{2,S}^{(2)})$, respectively, then
the estimate \eqref{e3.05} holds true for them.\\[1ex]
{\bf 3.\;Equivalency of the notions of kinetic and entropy solutions.}
Function $u=u(\boldsymbol{x},t,s)$ is an entropy solution of Problem
$\Pi_\gamma$ in the sense of Definition \ref{def.3.03}, if and only if it is a kinetic solution in the sense of Definition \ref{def.3.02}.
\end{theorem}
Estimate \eqref{e3.05} manifests uniqueness and $L^1$-strong
stability (with respect to given initial and final data) of kinetic
and entropy solution.

The proof of assertion 1 of Theorem \ref{theo.3.1} is fulfilled
further in Sections \ref{sec.4}--\ref{sec.7}, followed by the proofs of
assertions 2 and 3 of Theorem \ref{theo.3.1} in Section \ref{sec.8}.

\section{Relative compactness of the family of classical\\ solutions
${\bf \{u_{\boldsymbol{\varepsilon}}\}_{\boldsymbol{\varepsilon}>0}}$ to Problem
$\boldsymbol{\Pi}_{\boldsymbol{\gamma}\boldsymbol{\varepsilon}}$ in $\bf L^1(G_{T,S})$.\\ Derivation of the
kinetic equation \textbf{(\ref{e3.01a})}} \label{sec.4}

Our first aim is to prove relative compactness of the family of weak
solutions to Problem $\Pi_{\gamma\varepsilon}$. To this end, we use
the Perthame --- Souganidis and Lazar --- Mitrovi\'c averaging compactness
theorems \cite[Theorem 6]{PS-1998}, \cite[Theorem 7]{LM-2012} (see also
\cite[Theorem 5.2.1]{Per-2002}). We formulate the corollary of these
theorems in the reduced form adapted for studying solutions of Problem
$\Pi_{\gamma\varepsilon}$ as $\varepsilon\to0+$.

\begin{proposition}
\label{prop.4.1}%
{\bf (Corollary of the Perthame --- Souganidis and Lazar --- Mitrovi\'c
Theorems.)}
Let functions $f_n, f, g_{0,n}\in L^2(\mathbb{R}_x^d\times
\mathbb{R}_t^+\times\mathbb{R}_s^+\times\mathbb{R}_\lambda)$ and
Radon measures\linebreak $k_n, k\in\mathcal{M}(\mathbb{R}_x^d\times
\mathbb{R}_t^+\times\mathbb{R}_s^+\times\mathbb{R}_\lambda)$
 $(n\in\mathbb{N})$ satisfy the following
relations:
\begin{itemize}
\item[$\mathrm{(i)}$]
$f_n\underset{n\to\infty}{\longrightarrow}f$ weakly in
$L^2(\mathbb{R}_x^d\times\mathbb{R}_t^+\times\mathbb{R}_s^+\times\mathbb{R}_\lambda)$;
\item[$\mathrm{(ii)}$] $g_{0,n}\underset{n\to\infty}{\longrightarrow}0$ strongly in
$L^2(\mathbb{R}_x^d\times\mathbb{R}_t^+\times\mathbb{R}_s^+\times\mathbb{R}_\lambda)$;
\item[$\mathrm{(iii)}$] $k_n\underset{n\to\infty}{\longrightarrow}k$ weakly$^*$ in
$\mathcal{M}(\mathbb{R}_x^d\times\mathbb{R}_t^+\times\mathbb{R}_s^+\times\mathbb{R}_\lambda)$,\\
$\phantom{k_n\underset{n\to\infty}{\longrightarrow}k}$
strongly in $W_{\rm
loc}^{-1,p'}(\mathbb{R}_x^d\times\mathbb{R}_t^+\times\mathbb{R}_s^+\times\mathbb{R}_\lambda)$,
$\displaystyle \forall\, p'\in \Bigl[1,\frac{d+3}{d+2}\Bigr)$;
\item[$\mathrm{(iv)}$]  the triple
$\{f_n,g_{0,n},k_n\}$ $(n\in\mathbb{N})$ resolves the kinetic
equation
\begin{equation*}
\partial_t f_n+a'(\lambda)\partial_s f_n+
\boldsymbol{\varphi}'(\lambda)\cdot\nabla_x f_n-\Delta_x
f_n=\partial_\lambda k_n+\partial_s\partial_\lambda g_{0,n}
\end{equation*}
\end{itemize}
in the sense of distributions, where coefficients $a,\varphi_1,\ldots,\varphi_d$ obey
Conditions on $a\&\boldsymbol{\varphi}\&Z_\gamma$.

Then, for all compactly supported $\psi\in L^2(\mathbb{R}_\lambda)$,
the following limiting relation holds true:
\begin{equation*}
\int\limits_{\mathbb{R_\lambda}}
f_n(\boldsymbol{x},t,s,\lambda)\psi(\lambda)\,d\lambda\underset{n\to\infty}
\longrightarrow \int\limits_{\mathbb{R_\lambda}}
f(\boldsymbol{x},t,s,\lambda)\psi(\lambda)\,d\lambda \text{ strongly
in } L_{\rm
loc}^2(\mathbb{R}_x^d\times\mathbb{R}_t^+\times\mathbb{R}_s^+).
  \end{equation*}
\end{proposition}
In item (iii) by
$W^{-1,p'}(\mathbb{R}_x^d\times\mathbb{R}_t^+\times\mathbb{R}_s^+\times\mathbb{R}_\lambda)$
we denote the dual of the Sobolev space
$W_p^1(\mathbb{R}_x^d\times\mathbb{R}_t^+\times\mathbb{R}_s^+\times\mathbb{R}_\lambda)$.

Now, using Propositions \ref{prop.2.1} and \ref{prop.4.1} and Lemmas
\ref{lem.3.01} and \ref{lem.3.02}, we establish the compactness
result for $\{u_\varepsilon\}_{\varepsilon>0}$.

\begin{lemma}
\label{lem.4.1}%
The family of classical solutions $\{u_\varepsilon\}_{\varepsilon>0}$ of
Problem $\Pi_{\gamma\varepsilon}$ is relatively compact in $L^1(G_{T,S})$.
\end{lemma}
\proof We justify Lemma \ref{lem.4.1} similarly to the claim of relative
compactness of families of weak solution to the regularized problems
outlined in \cite[Section 3]{Kuz-2016}, \cite[Section 3]{Kuz-2017}.

Recall that, due to Remark \ref{classic-weak}, the classical solution is the weak solution in the sense of Definition
\ref{def.2.01}.
In \eqref{e2.02} take an arbitrary admissible test-function $\phi$
vanishing in a neighborhood of the planes $\{t=0\}$, $\{t=T\}$,
$\{s=0\}$, $\{s=S\}$ and boundary $\partial\Omega$, and integrate
by parts with respect to $t$, $s$ and $\boldsymbol{x}$ in the first
{three} summands to get
\begin{equation}
\label{e4.01}%
\int\limits_{G_{T,S}}\big(\phi\partial_t u_\varepsilon+\phi\partial_s
a(u_\varepsilon)+\phi{\rm
div}_x\boldsymbol{\varphi}(u_\varepsilon)+\nabla_x
u_\varepsilon\cdot\nabla_x\phi
+\varepsilon\partial_su_\varepsilon\partial_s\phi
-Z_\gamma(\boldsymbol{x},t,s,u_\varepsilon)\phi\big)\,d\boldsymbol{x}dtds=0.
  \end{equation}

Taking $\phi:=\eta'(u_\varepsilon)\zeta$ with arbitrary $\eta\in
C^\infty(\mathbb{R}_\lambda)$ and $\zeta\in C_0^\infty(G_{T,S})$ in
\eqref{e4.01}, using the chain rule and integration by parts in the
first five summands, we arrive at the integral equality
\begin{multline}
\label{e4.02}%
\int\limits_{G_{T,S}} \Big(-\eta(u_\varepsilon)\partial_t\zeta-
q_a(u_\varepsilon)\partial_s\zeta-\boldsymbol{q}_\varphi(u_\varepsilon)\cdot\nabla_x\zeta
-\eta(u_\varepsilon)\Delta_x\zeta+
\varepsilon\eta'(u_\varepsilon)\partial_su_\varepsilon\partial_s\zeta
\\
+\eta''(u_\varepsilon)(\vert\nabla_xu_\varepsilon\vert^2+\varepsilon\vert\partial_su_\varepsilon\vert^2)\zeta
-Z_\gamma(\boldsymbol{x},t,s,u_\varepsilon)\eta'(u_\varepsilon)\zeta\Big)\,d\boldsymbol{x}dtds=0.
  \end{multline}

Using item (i) in Lemma \ref{lem.3.01}, we rewrite \eqref{e4.02} as
the integral equality for the function $\chi(\lambda;u_\varepsilon)$
having the form \eqref{chi-f}. More precisely, considering appropriate groups of terms in \eqref{e4.02} one by one, we get
\begin{multline}
\label{e4.03i}%
\int\limits_{G_{T,S}} \Big(-\eta(u_\varepsilon)\partial_t\zeta-
q_a(u_\varepsilon)\partial_s\zeta-\boldsymbol{q}_\varphi(u_\varepsilon)\cdot\nabla_x\zeta
-\eta(u_\varepsilon)\Delta_x\zeta\Big)\,d\boldsymbol{x}dtds=
\\
-\int\limits_{G_{T,S}}\int\limits_{\mathbb{R}_\lambda}\chi(\lambda;u_\varepsilon)\Big(
\partial_t(\zeta \eta'(\lambda))+a'(\lambda)\partial_s(\zeta
\eta'(\lambda))+\boldsymbol{\varphi}'(\lambda)\cdot\nabla_x(\zeta\eta'(\lambda))
+\Delta_x(\zeta\eta'(\lambda))\Big)\,d\lambda d\boldsymbol{x}dtds,
  \end{multline}
\begin{equation}
\label{e4.03ii}%
\int\limits_{G_{T,S}}\varepsilon\eta'(u_\varepsilon)\partial_su_\varepsilon
\partial_s\zeta\,d\boldsymbol{x}dtds=
\int\limits_{G_{T,S}}\int\limits_{\mathbb{R}_\lambda}
\varepsilon(\chi(\lambda;u_\varepsilon)-\mathbf{1}_{(\lambda\geqslant
0)})\partial_s
u_\varepsilon\partial_s(\zeta\eta''(\lambda))\,d\lambda
d\boldsymbol{x}dtds,
  \end{equation}
\begin{equation}
\label{e4.03iii}%
\int\limits_{G_{T,S}}\eta''(u_\varepsilon)\left(|\nabla_x
u_\varepsilon|^2+\varepsilon |\partial_s
u_\varepsilon|^2\right)\zeta\,d\boldsymbol{x}dtds=\left\langle\left(
|\nabla_x u_\varepsilon|^2+\varepsilon |\partial_s
u_\varepsilon|^2\right)\delta_{(\lambda=u_\varepsilon)},\zeta\eta''
\right\rangle,
  \end{equation}
  \begin{multline}
\label{e4.03iv}%
-\int\limits_{G_{T,S}}Z_\gamma(\boldsymbol{x},t,s,u_\varepsilon)\eta'(u_\varepsilon)\zeta\,d\boldsymbol{x}dtds=\\
-\int\limits_{G_{T,S}}\int\limits_{\mathbb{R}_\lambda}Z_\gamma(\boldsymbol{x},t,s,u_\varepsilon)
(\chi(\lambda;u_\varepsilon)-\mathbf{1}_{(\lambda\geqslant0)})\zeta\eta''(\lambda)\,d\boldsymbol{x}dtds=
\\
-\left\langle Z_\gamma(\cdot,\cdot,\cdot,u_\varepsilon)
(\chi(\cdot;u_\varepsilon)-\mathbf{1}_{(\lambda\geqslant0)}),\zeta\eta''
\right\rangle.
  \end{multline}
In \eqref{e4.03iii}, \eqref{e4.03iv} and further in the paper, by
$\langle\cdot,\cdot\rangle$ we standardly denote the duality bracket
between
$\mathcal{M}(\mathbb{R}_x^d\times\mathbb{R}_t^+\times\mathbb{R}_s^+\times\mathbb{R}_\lambda)$
and
$C_0(\mathbb{R}_x^d\times\mathbb{R}_t^+\times\mathbb{R}_s^+\times\mathbb{R}_\lambda)$.

Now, combining \eqref{e4.02}--\eqref{e4.03iv} we arrive at the desired integral
equality:
  \begin{multline}
\label{e4.03}%
-\int\limits_{G_{T,S}}\int\limits_{\mathbb{R}_\lambda}\chi(\lambda;u_\varepsilon)\Big(
\partial_t(\zeta \eta'(\lambda))+a'(\lambda)\partial_s(\zeta
\eta'(\lambda))+\boldsymbol{\varphi}'(\lambda)\cdot\nabla_x(\zeta\eta'(\lambda))
+\Delta_x(\zeta\eta'(\lambda))\Big)\,d\lambda d\boldsymbol{x}dtds
\\
+\int\limits_{G_{T,S}}\int\limits_{\mathbb{R}_\lambda}g_\varepsilon(\boldsymbol{x},t,s,\lambda)
\partial_s(\zeta\eta''(\lambda))\,d\lambda d\boldsymbol{x}dtds
+\langle k_\varepsilon,\zeta\eta''\rangle=0,
  \end{multline}
where
\begin{equation*}
g_\varepsilon(\boldsymbol{x},t,s,\lambda):=\varepsilon
\left(\chi(\lambda;u_\varepsilon(\boldsymbol{x},t,s))-\mathbf{1}_{(\lambda\geqslant0)}\right)
\partial_su_\varepsilon(\boldsymbol{x},t,s),
  \end{equation*}
\begin{multline*}
k_\varepsilon(\boldsymbol{x},t,s,\lambda):=(|\nabla_x
u_\varepsilon(\boldsymbol{x},t,s)|^2+\varepsilon |\partial_s
u_\varepsilon(\boldsymbol{x},t,s)|^2)\delta_{(\lambda=u_\varepsilon(\boldsymbol{x},t,s))}-
\\
Z_\gamma(\boldsymbol{x},t,s,u_\varepsilon(\boldsymbol{x},t,s))\left(\chi(\lambda;u_\varepsilon(\boldsymbol{x},t,s))-
\mathbf{1}_{(\lambda\geqslant0)}\right),\quad
  k_\varepsilon\in\mathcal{M}(\mathbb{R}_x^d\times\mathbb{R}_t^+\times\mathbb{R}_s^+\times\mathbb{R}_\lambda).
  \end{multline*}
Notice that ${\rm supp\,}\chi(\lambda;u_\varepsilon)$ lays in the
layer $\{|\lambda|\leqslant M_0\}$, since the maximum principle
\eqref{e2.03} holds for $\{u_{\varepsilon}\}_{\varepsilon>0}$.
Therefore integration with respect to $\lambda$ in \eqref{e4.03} is
fulfilled, in fact, merely over $[-M_0,M_0]$. Since the linear span of the set
$\{\zeta(\boldsymbol{x},t,s)\eta'(\lambda)\}$ is dense in
$C_0^2(G_{T,S}\times\mathbb{R}_\lambda)$, equality \eqref{e4.03} is
equivalent to the kinetic equation
\begin{equation*}
\partial_t\chi(\lambda;u_\varepsilon)+a'(\lambda)
\partial_s\chi(\lambda;u_\varepsilon)+\boldsymbol{\varphi}'(\lambda)\cdot
\nabla_x\chi(\lambda;u_\varepsilon)-\Delta_x\chi(\lambda;u_\varepsilon)=\partial_\lambda
k_\varepsilon -\partial_s\partial_\lambda g_\varepsilon
  \end{equation*}
in the sense of distributions on $G_{T,S}\times[-M_0,M_0]$.

On the strength of the energy estimate \eqref{e2.04}, we have that
\begin{equation}
\label{e4.04}%
\|g_\varepsilon\|_{L^2(G_{T,S}\times[-M_0,M_0])}\leqslant
\sqrt{2\varepsilon}C_1\underset{\varepsilon\to0+}{\longrightarrow}0,
\end{equation}
\begin{equation}
\label{e4.05}%
\|k_\varepsilon\|_{L_w^1(G_{T,S};\mathcal{M}[-M_0,M_0])}\leqslant
C_1.
\end{equation}
Here $L_w^1(G_{T,S};\mathcal{M}[-M_0,M_0])$ is the space of weakly
measurable functions from $G_{T,S}$ into\linebreak $\mathcal{M}[-M_0,M_0]$
equipped with the norm
\[
\|\nu\|_{L_w^1(G_{T,S};\mathcal{M}[-M_0,M_0])}=\int\limits_{G_{T,S}}
\|\nu(\boldsymbol{x},t,s,\cdot)\|_{\mathcal{M}[-M_0,M_0]}d\boldsymbol{x}dtds.
\]
Recall that the Dirac delta-function $\delta_{(\lambda=u_\varepsilon(\boldsymbol{x},t,s))}$ (in the expression of $k_\varepsilon$) is a positive
Radon measure having the unit mass, i.e., it is a probability measure.
Clearly,
\[
L_w^1(G_{T,S};\mathcal{M}[-M_0,M_0])\hookrightarrow\mathcal{M}(G_{T,S}\times[-M_0,M_0]).
\]
Therefore there exist a subsequence $\{\varepsilon'\to0+\}$ and measure $k\in\mathcal{M}(G_{T,S}\times[-M_0,M_0])$ such
that
\begin{equation}
\label{e4.06}%
k_{\varepsilon'}\underset{\varepsilon'\to0+}{\longrightarrow}k\text{
weakly* in }\mathcal{M}(G_{T,S}\times[-M_0,M_0]).
  \end{equation}

On the strength of Sobolev's embedding theorem \cite[Chapter I, Section
8]{Sob-2008}, the space of functionals
$\mathcal{M}(G_{T,S}\times[-M_0,M_0])$ is compactly embedded into
$W_{\rm loc}^{-1,p'}(G_{T,S}\times[-M_0,M_0])$ for any
$p'\in[1,\frac{d+3}{d+2})$. Hence
\begin{equation}
\label{e4.07}%
k_{\varepsilon'}\underset{\varepsilon'\to0+}{\longrightarrow} k\text{
strongly in } W_{\rm loc}^{-1,p'}(G_{T,S}\times[-M_0,M_0]).
  \end{equation}
Finally, there exist a subsequence $\{\varepsilon''\to0+\}$ from
$\{\varepsilon'\to0+\}$ and $f\in
L^2(G_{T,S}\times[-M_0,M_0])$ such that
\begin{equation}
\label{e4.08}%
\chi(\lambda;u_{\varepsilon''})\underset{\varepsilon''\to0+}{\longrightarrow}
f\text{ weakly in } L^2(G_{T,S}\times[-M_0,M_0]),
  \end{equation}
since, obviously, $\|\chi(\lambda;u_{\varepsilon''})\|_{L^2(G_{T,S}\times[-M_0,M_0])}\leqslant
\sqrt{M_0\,T S\,{\rm meas\,}\Omega}$, which means that the family $\{\chi(\lambda;u_{\varepsilon''})\}$ is uniformly bounded in $\varepsilon''$.
On the strength of Proposition \ref{prop.4.1}, from \eqref{e4.04}--\eqref{e4.08} it follows that
\begin{equation}
\label{e4.09}%
\int\limits_{-M_0}^{M_0}\chi(\lambda;u_{\varepsilon''})\,d\lambda\underset{\varepsilon''\to0+}{\longrightarrow}
u=\int\limits_{-M_0}^{M_0}f\,d\lambda\text{ strongly in }
L^2(G_{T,S}).
  \end{equation}
In view of Lemma \ref{lem.3.02} and item (i) in Lemma
\ref{lem.3.01}, this yields that
\begin{equation}
\label{e4.10}%
u_{\varepsilon''}\underset{\varepsilon''\to0+}{\longrightarrow}
u\text{ strongly in } L^2(G_{T,S})\,\,(\text{and in }L^1(G_{T,S}))
  \end{equation}
and that $f$ is the $\chi$-function:
\begin{equation}
\label{e4.11}%
f(\boldsymbol{x},t,s,\lambda)=\chi(\lambda;u(\boldsymbol{x},t,s)).
  \end{equation}
Lemma \ref{lem.4.1} is proved. \qed\\[1ex]
\indent Furthermore, passing to the limit along the subsequence
$\varepsilon''\to0+$ in \eqref{e4.03}, we get
\begin{multline}
\label{e4.12i}%
-\int\limits_{G_{T,S}}\int\limits_{\mathbb{R}_\lambda}\chi(\lambda;u_{\varepsilon''})
\Bigl(\partial_t(\zeta\eta'(\lambda))+a'(\lambda)
\partial_s(\zeta\eta'(\lambda))+\\ \boldsymbol{\varphi}'(\lambda)\cdot\nabla_x
(\zeta\eta'(\lambda))+\Delta_x(\zeta\eta'(\lambda))\Bigr)d\lambda
d\boldsymbol{x}dtds \underset{\varepsilon''\to0+}{\longrightarrow}
\\
-\int\limits_{G_{T,S}}\int\limits_{\mathbb{R}_\lambda}
\chi(\lambda;u)\Bigl(\partial_t(\zeta\eta'(\lambda))+a'(\lambda)
\partial_s(\zeta\eta'(\lambda))+\boldsymbol{\varphi}'(\lambda)\cdot\nabla_x
(\zeta\eta'(\lambda))+\Delta_x(\zeta\eta'(\lambda))\Bigr)d\lambda
d\boldsymbol{x}dtds
\end{multline}
due to \eqref{e4.08} and \eqref{e4.11},
\begin{equation}
\label{e4.12ii}%
-\int\limits_{G_{T,S}}\int\limits_{\mathbb{R}_\lambda}
g_{\varepsilon''}(\boldsymbol{x},t,s,\lambda)\partial_s(\zeta\eta''(\lambda))\,d\lambda
d\boldsymbol{x}dtds \underset{\varepsilon''\to0+}{\longrightarrow}0
  \end{equation}
due to \eqref{e4.04}, and
\begin{multline}
\label{e4.12iii}%
\langle k_{\varepsilon''},\zeta\eta''\rangle\equiv
\int\limits_{G_{T,S}}\Bigl(\eta''(u_\varepsilon)(|\nabla_x
u_{\varepsilon''}|^2+\varepsilon''|\nabla_s
u_{\varepsilon''}|^2)-\eta'(u_{\varepsilon''})Z_\gamma(\boldsymbol{x},t,s,u_{\varepsilon''})\Bigr)\zeta\,d\boldsymbol{x}dtds
\underset{\varepsilon''\to0+}{\longrightarrow}
\\
\left\langle\lim\limits_{\varepsilon''\to0+}\delta_{(\lambda=u_{\varepsilon''})}
\left(|\nabla_x u_{\varepsilon''}|^2+\varepsilon''|\nabla_s
u_{\varepsilon''}|^2\right),\eta''(\cdot)\zeta\right\rangle- \hspace{7cm}
\\
\int\limits_{G_{T,S}}\int\limits_{\mathbb{R}_\lambda}
\chi(\lambda;u)
\partial_\lambda(\eta'(\lambda)Z_\gamma(\boldsymbol{x},t,s,\lambda))\zeta\,d\lambda d\boldsymbol{x}dtds-
\bigl\langle
\delta_{(\lambda=0)}Z_\gamma(\boldsymbol{x},t,s,\cdot),\eta'(\cdot)\zeta\bigr\rangle
  \end{multline}
due to \eqref{e4.03iii}, \eqref{e4.03iv}, \eqref{e4.08},
\eqref{e4.11} and item (i) in Lemma \ref{lem.3.01}.

Denote
\begin{equation}
\label{e4.12iv}%
n:=\delta_{(\lambda=u)}|\nabla_x u|^2, \quad
m:=\lim\limits_{\varepsilon''\to0+}\Bigl(
\delta_{(\lambda=u_{\varepsilon''})}\left\vert\nabla_x
u_{\varepsilon''}\right\vert^2-\delta_{(\lambda=u)}\left\vert\nabla_x
u\right\vert^2+\delta_{(\lambda=u_{\varepsilon''})}\varepsilon''\left\vert\partial_s
u_{\varepsilon''}\right\vert^2\Bigr).
  \end{equation}
Notice that $m,n\in\mathcal{M}^+(G_{T,S}\times[-M_0,M_0])$.

Using \eqref{e4.12i}--\eqref{e4.12iv}, taking
into account that integration in $\lambda$ is fulfilled merely over\linebreak
$[-M_0,M_0]$ and that the linear span of the set
$\{\zeta(\boldsymbol{x},t,s)\eta'(\lambda)\}$ is dense in
$C_0^2(G_{T,S}\times\mathbb{R}_\lambda)$, from \eqref{e4.03} we derive the integral
equality
\begin{multline}
\label{e4.12}%
\int\limits_{G_{T,S}}\int\limits_{-M_0}^{M_0} \chi(\lambda;u)\Bigl(
\partial_t\Phi+a'(\lambda)\partial_s\Phi+
\boldsymbol{\varphi}'(\lambda)\cdot\nabla_x\Phi+\Delta_x\Phi+
\partial_\lambda(Z_\gamma(\boldsymbol{x},t,s,\lambda)\Phi)\Bigr)\,d\lambda d\boldsymbol{x}dtds +
\\
\langle\delta_{(\lambda=0)}Z_\gamma(\boldsymbol{x},t,s,\cdot),\Phi\rangle=\langle
m+n,\partial_\lambda\Phi\rangle\quad\forall\Phi\in
C_0^2(G_{T,S}\times\mathbb{R}_\lambda).
  \end{multline}
In the sense of distributions, \eqref{e4.12} is equivalent to the
kinetic equation \eqref{e3.01a}.

Thus, we have proved the following lemma.

\begin{lemma}
\label{lem.4.2}%
The limit $u\in L^\infty(G_{T,S}\times[-M_0,M_0])$ of the
subsequence $\{u_{\varepsilon''}\}_{\varepsilon''\to0+}$ of classical
solutions to Problem $\Pi_{\gamma\varepsilon}$, along with measures\;
$m,n\in\mathcal{M}^+(G_{T,S}\times[-M_0,M_0])$, resolves the kinetic equation \eqref{e3.01a} in the sense of
distributions.
  \end{lemma}

\section{Traces of solutions of the kinetic equation \textbf{(\ref{e3.01a})}}
\label{sec.5}
In this section, we establish existence of $L^1$-strong one-sided
traces $u_0^{{\rm tr},(1)}$, $u_{\tau-0}^{{\rm tr},(1)}$,
$u_{\tau+0}^{{\rm tr},(1)}$, $u_{0}^{{\rm tr},(2)}$, and $u_{S}^{{\rm
tr},(2)}$ of solutions of the kinetic equation \eqref{e3.01a} on the
sections $\Gamma_{0+}^1=\overline{\Omega}\times\{t=0+\}\times[0,S]$,
$\Gamma_{\tau-0}^1=\overline{\Omega}\times\{t=\tau-0\}\times[0,S]$,
$\Gamma_{\tau+0}^1=\overline{\Omega}\times\{t=\tau+0\}\times[0,S]$,
$\Gamma_{0+}^2=\overline{\Omega}\times[0,T]\times\{s=0+\}$, and
$\Gamma_{S-0}^2=\overline{\Omega}\times[0,T]\times\{s=S-0\}$,
respectively. Here $\tau\in(0,T)$ is an arbitrarily fixed value.

At first, we prove that there exists the trace on $\Gamma_{0+}^1$.
To this end, we use the theory  elaborated in \cite{AM-2013}.

\begin{definition}
\label{def.5.1}%
(\cite[Definition 1]{AM-2013}.) We say that function $f\in
L^\infty(\mathbb{R}_x^d\times\mathbb{R}_t^+\times
\mathbb{R}_s^+\times\mathbb{R}_\lambda)$ admits {\rm an averaged
trace}, if there exist function $f_0\in
L^\infty(\mathbb{R}_x^d\times\mathbb{R}_s^+\times\mathbb{R}_\lambda)$
and set $E\subset\mathbb{R}_t^+$ of full Lebesgue measure such
that
\begin{equation}
\label{e5.01}%
\lim\limits_{{}_{t\in
E}^{t\to0}}\int\limits_{\mathcal K}\left\vert\,\int\limits_{\mathbb{R}_\lambda}
\left(f(\boldsymbol{x},t,s,\lambda)-f_0(\boldsymbol{x},s,\lambda)\right)
\rho(\lambda)\,d\lambda\right\vert d\boldsymbol{x}ds=0
  \end{equation}
for any function $\rho\in C_0^1(\mathbb{R}_\lambda)$ and any
relatively compact set ${\mathcal K}\subset\mathbb{R}_x^d\times\mathbb{R}_s^+$.
\end{definition}

\begin{remark} \label{rem.esslim}
In view of Definition \ref{def.esslim}, relation \eqref{e5.01} can be equivalently written as follows:
\begin{equation*}
\underset{t\to0+}{\mbox{\rm ess\,lim}}\int\limits_{\mathcal K}\left\vert\,\int\limits_{\mathbb{R}_\lambda}
\left(f(\boldsymbol{x},t,s,\lambda)-f_0(\boldsymbol{x},s,\lambda)\right)
\rho(\lambda)\,d\lambda\right\vert d\boldsymbol{x}ds=0.
  \end{equation*}
\end{remark}

\begin{proposition}
\label{prop.5.1}%
{\bf (Corollary of the  Aleksi\'c --- Mitrovi\'c
Strong Trace Existence Theorem \cite[Theorem 4]{AM-2013}.)} Let function $f\in
L^\infty\left(\mathbb{R}_x^d\times\mathbb{R}_t^+
\times\mathbb{R}_s^+\times\mathbb{R}_\lambda\right)$ and locally
finite Borel measure
$k\in\mathcal{M}\left(\mathbb{R}_x^d\times\mathbb{R}_t^+
\times\mathbb{R}_s^+\times\mathbb{R}_\lambda\right)$ resolve the
kinetic equation
\begin{equation}
\label{e5.02}%
\partial_t f+a'(\lambda)\partial_sf+\boldsymbol{\varphi}'(\lambda)\cdot
\nabla_xf-\Delta_xf=\partial_\lambda k
  \end{equation}
in the sense of distributions, where the coefficients
$a,\varphi_1,\ldots,\varphi_d$ obey Conditions on
$a\&\boldsymbol{\varphi}\&Z_\gamma$.

Additionally, let
\begin{equation}
\label{e5.03}%
{\rm supp}f\text{ lay in the layer }\{|\lambda|\leqslant M\}, \text{
with some }M={\rm const}>0,
  \end{equation}
and
\begin{equation*}
\partial_{x_i}f\in L_w^2\left(\mathbb{R}_x^d\times\mathbb{R}_t^+
\times\mathbb{R}_s^+;\mathcal{M}(\mathbb{R}_\lambda)\right),\quad
i=1,\ldots,d.
  \end{equation*}

Then there exists an average trace $f_0\in
L^\infty(\mathbb{R}_x^d\times\mathbb{R}_s^+\times
\mathbb{R}_\lambda)$ of function $f$ on\linebreak
$\mathbb{R}_x^d\times\{t=0+\}\times\mathbb{R}_s^+\times
\mathbb{R}_\lambda$ in the sense of Definition \ref{def.5.1}.
\end{proposition}

\begin{remark}
\label{rem.5.1} It is worth to notice that the genuine nonlinearity
condition \eqref{e1.02} is excessive in Proposition \ref{prop.5.1}.
(See \cite[Proof of Theorem 4]{AM-2013}.)
 Also remark that, in \cite{AM-2013}, the space $L_w^2(\mathbb{R}^N;\mathcal{M}(\mathbb{R}))$
is denoted by $L^2(\mathbb{R}^N;\mathcal{M}(\mathbb{R}))$.
  \end{remark}

Now we establish the following lemma.
\begin{lemma}
\label{lem.5.1}%
Let triple $u\in L^\infty(G_{T,S})\cap
L^2((0,T)\times(0,S);\text{\it\r{W}}{}_2^{\,1}(\Omega))$, $m\in
\mathcal{M}^+(G_{T,S}\times[-M_0,M_0])$ and
$n=\delta_{(\lambda=u)}|\nabla_xu|^2$ resolve the kinetic equation
\eqref{e3.01a} in the sense of the integral equality \eqref{e4.12}.
(As the matter of fact, we consider the case when
$u=\lim\limits_{\varepsilon''\to0+}u_{\varepsilon''}$).

Then there exists $u_0^{{\rm tr},(1)}\in L^\infty(\Xi^1)$, the trace
of $u$ on $\Gamma_{0+}^1$, such that
\begin{equation}
\label{e5.04}%
\underset{t\to0+}{\rm ess\,lim}\int\limits_{\Xi^1} \left\vert
u(\boldsymbol{x},t,s)-u_0^{{\rm
tr},(1)}(\boldsymbol{x},s)\right\vert\,d\boldsymbol{x}ds=0,
\end{equation}
\begin{equation}
\label{e5.05}%
\underset{t\to0+}{\rm
ess\,lim}\int\limits_{\Xi^1}\int\limits_{-M_0}^{M_0}
\left\vert\chi(\lambda;u(\boldsymbol{x},t,s))- \chi(\lambda;
u_0^{{\rm tr},(1)}(\boldsymbol{x},s))\right\vert\,d\lambda
d\boldsymbol{x}ds=0.
\end{equation}
\end{lemma}
\proof In \eqref{e4.12} rewrite the last two terms in the left-hand
side as follows:
\begin{multline}
\label{e5.06i}%
\int\limits_{G_{T,S}}\int\limits_{-M_0}^{M_0}\chi(\lambda;u)
\partial_\lambda(Z_\gamma(\boldsymbol{x},t,s,\lambda)\Phi(\boldsymbol{x},t,s,\lambda))d\lambda d\boldsymbol{x}dtds
+\langle\delta_{(\lambda=0)}
Z_\gamma(\boldsymbol{x},t,s,\cdot),\Phi(\boldsymbol{x},t,s,\cdot)\rangle=
\\
\int\limits_{G_{T,S}}
Z_\gamma(\boldsymbol{x},t,s,u(\boldsymbol{x},t,s))\Phi(\boldsymbol{x},t,s,u(\boldsymbol{x},t,s))\,d\boldsymbol{x}dtds=\\-
\int\limits_{G_{T,S}}\left(\,\int\limits_{-M_0}^{M_0}\boldsymbol{1}_{(\lambda\geqslant
u(\boldsymbol{x},t,s))}\partial_\lambda(Z_\gamma(\boldsymbol{x},t,s,\lambda)\Phi(\boldsymbol{x},t,s,\lambda))d\lambda\right)\,d\boldsymbol{x}dtds=
\\
-\int\limits_{G_{T,S}}\int\limits_{-M_0}^{M_0}\Bigl(\mathbf{1}_{(\lambda\geqslant
u(\boldsymbol{x},t,s))}Z_\gamma(\boldsymbol{x},t,s,\lambda)-\\ \int\limits_0^\lambda
\mathbf{1}_{(\lambda'\geqslant
u(\boldsymbol{x},t,s))}\partial_{\lambda'}Z_\gamma(\boldsymbol{x},t,s,\lambda')\,d\lambda'\Bigr)\partial_\lambda
\Phi(\boldsymbol{x},t,s,\lambda)\,d\lambda d\boldsymbol{x}dtds,
\\
\forall\,\Phi\in C_0^2(G_{T,S}\times[-M_0,M_0]).
\end{multline}
In this chain of equalities, at first, we use assertion (i) in Lemma
\ref{lem.3.01}, at second, we use the evident identity
\begin{equation*}
\phi(u(\boldsymbol{x},t,s))=-\int\limits_{\mathbb
R}  \mathbf{1}_{(\lambda\geqslant
u(\boldsymbol{x},t,s))} \phi'(\lambda)d\lambda\quad \forall\,\phi\in
C_0^1(\mathbb{R}),
  \end{equation*}
and, finally, we integrate by parts in $\lambda$.

Substituting \eqref{e5.06i} into \eqref{e4.12} we obtain the
integral equality equivalent to \eqref{e4.12}. Consequently, the
kinetic equation \eqref{e3.01a} takes the following equivalent form:
\begin{multline} \label{e3.01bis}
\partial_{t}\chi(\lambda;u(\boldsymbol{x},t,s))
+a^{\prime}(\lambda)\partial_{s}\chi(\lambda;u(\boldsymbol{x},t,s))
+\boldsymbol{\varphi}^{\prime}(\lambda)\cdot\nabla_x\chi(\lambda;u(\boldsymbol{x},t,s))
-\Delta_x\chi(\lambda;u(\boldsymbol{x},t,s))=
\\
\partial_\lambda\Biggl(m(\boldsymbol{x},t,s,\lambda)+n(\boldsymbol{x},t,s,\lambda)
+\mathbf{1}_{(\lambda\geqslant u(\boldsymbol{x},t,s))}
Z_\gamma(\boldsymbol{x},t,s,\lambda)-\int\limits_0^\lambda
\mathbf{1}_{(\lambda'\geqslant
u(\boldsymbol{x},t,s))}\partial_{\lambda'}Z_\gamma(\boldsymbol{x},t,s,\lambda')\,d\lambda'\Biggr).
\end{multline}
Hence we conclude that the kinetic equation \eqref{e3.01a} is the
kinetic equation of  the form \eqref{e5.02} with $f=\chi(\lambda;u)$
and
\begin{equation}
\label{e5.06ii}%
k:=k_*\equiv m+n+\mathbf{1}_{(\lambda\geqslant
u)}Z_\gamma-\int\limits_0^\lambda \mathbf{1}_{(\lambda'\geqslant
u)}\partial_{\lambda'}Z_\gamma(\boldsymbol{x},t,s,\lambda')\,d\lambda'.
  \end{equation}
Clearly, $\chi$ is bounded and satisfies \eqref{e5.03}, and $k_*$ is a locally finite Borel measure belonging to
$\mathcal{M}(\mathbb{R}_x^d\times\mathbb{R}_t^+\times\mathbb{R}_s^+\times\mathbb{R}_\lambda)$.
Furthermore, using assertion (i) in Lemma \ref{lem.3.01} and Green's
theorem \cite[Lemma 2.20]{MNRR-1996}, we establish the representation
\begin{multline}
\label{e5.06iii}%
\left\langle\partial_{x_i}\chi,\psi\right\rangle=-\int\limits_{G_{T,S}}
\int\limits_{-M_0}^{M_0}\chi(\lambda;u(\boldsymbol{x},t,s))\partial_{x_i}\psi(\boldsymbol{x},t,s,\lambda)\,d\lambda
d\boldsymbol{x}dtds=
\\
\int\limits_{G_{T,S}}\left[(\partial_{x_i}\Psi(\boldsymbol{x},t,s,\lambda))\big\vert_{\lambda=
u(\boldsymbol{x},t,s)}-\partial_{x_i}\Psi(\boldsymbol{x},t,s,0)\right]\,d\boldsymbol{x}dtds=
\\
\int\limits_{G_{T,S}}\left[\partial_{x_i}\Psi(\boldsymbol{x},t,s,u)
-\psi(\boldsymbol{x},t,s,u)\partial_{x_i}u-\partial_{x_i}\Psi(\boldsymbol{x},t,s,0)\right]\,d\boldsymbol{x}dtds=
\\
-\int\limits_{G_{T,S}}
\psi(\boldsymbol{x},t,s,u)\partial_{x_i}u\,d\boldsymbol{x}dtds+
\int\limits_{\Gamma_l}\left(\Psi(\boldsymbol{\sigma},t,s,u(\boldsymbol{\sigma},t,s))
-\Psi(\boldsymbol{\sigma},t,s,0)\right)n_i(\boldsymbol{\sigma})\,d\boldsymbol{\sigma}dtds=
\\
-\int\limits_{G_{T,S}}
\psi(\boldsymbol{x},t,s,u)\partial_{x_i}u\,d\boldsymbol{x}dtds,\quad \forall\,i=1,\ldots,d,\;\forall\,\psi\in
C_0^\infty(G_{T,S};C_0(\mathbb{R}_\lambda)),
  \end{multline}
where $\Psi$ is the primitive of $\psi$ with respect to $\lambda$:
\[
\partial_\lambda\Psi(\boldsymbol x,t,s,\lambda)=\psi(\boldsymbol
x,t,s,\lambda),
\]
 $n_i$ is the $i$-th component of the unit outward
normal to $\partial\Omega$, $d\boldsymbol{\sigma}$ is an
infinitesimal element of $\partial\Omega$, and
$\partial_{x_i}\Psi(\boldsymbol{x},t,s,u)$ is the full derivative of
$\Psi$ with respect to $x_i$, i.e.,
\[
\partial_{x_i}\Psi(\boldsymbol{x},t,s,u)=\left(\partial_{x_i}\Psi(\boldsymbol{x},t,s,\lambda)\right)\big\vert_{\lambda= u(\boldsymbol{x},t,s)}
+\psi(\boldsymbol{x},t,s,u)\partial_{x_i}u.
\]
The last equality in \eqref{e5.06iii} is valid due to the property
of finiteness of $\psi$. More certainly, we have
\[
\Psi(\boldsymbol{\sigma},t,s,u(\boldsymbol{\sigma},t,s))=
\Psi(\boldsymbol{\sigma},t,s,0)=0\quad \text{ for
}\boldsymbol{\sigma}\in \partial\Omega.
\]

Now using the Cauchy-Buniakovskii inequality, from \eqref{e5.06iii}
we derive the bound
\begin{multline*}
\left\vert\left\langle\partial_{x_i}\chi,\psi\right\rangle\right\vert
\equiv \left\vert-\int\limits_{G_{T,S}}
\psi(\boldsymbol{x},t,s,u)\partial_{x_i}u\,d\boldsymbol{x}dtds\right\vert\leqslant \\
\|\partial_{x_i}u\|_{L^2(G_{T,S})}\|\psi\|_{L^2(G_{T,S};\,C_0(\mathbb{R}_\lambda))}\equiv
C_u\|\psi\|_{L^2(G_{T,S};\,C_0(\mathbb{R}_\lambda))},\quad\forall\,
\psi\in C_0^1(G_{T,S};C_0(\mathbb{R}_\lambda)).
\end{multline*}
Since $C_0^1(G_{T,S};\,C_0(\mathbb{R}_\lambda))$ is dense in
$L^2(G_{T,S};\,C_0(\mathbb{R}_\lambda))$, this bound yields that
\[
\|\partial_{x_i}\chi\|_{L_w^2(G_{T,S};\mathcal{M}(\mathbb{R}_\lambda))}\leqslant
C_u.
\]
Thus, all assumptions of Proposition \ref{prop.5.1} hold for $f=\chi$
and $k=k_*$ (see \eqref{e5.06ii}). Therefore, the limiting relation
\eqref{e5.01} is valid for $\chi$, i.e., there exists $f_0\in
L^\infty(G_{T,S}\times[-M_0,M_0])$ such that
\begin{equation}
\label{e5.06}%
\underset{t\to0+}{\rm ess\,lim}\int\limits_{\Xi^1}\left\vert\,
\int\limits_{-M_0}^{M_0}\left(\chi(\lambda;u(\boldsymbol{x},t,s))-f_0(\boldsymbol{x},s,\lambda)
\right)\rho(\lambda)\,d\lambda\right\vert d\boldsymbol{x}ds=0,\quad \forall\,\rho\in C_0^1(\mathbb{R}_\lambda).
  \end{equation}
 On the strength of assertion
(i) in Lemma \ref{lem.3.01}, taking $\rho\equiv 1$ on $[-M_0,M_0]$, from \eqref{e5.06} we
immediately derive \eqref{e5.04} with
\begin{equation}
\label{e5.07}%
u_0^{{\rm
tr},(1)}(\boldsymbol{x},s):=\int\limits_{-M_0}^{M_0}f_0(\boldsymbol{x},s,\lambda)\,d\lambda.
  \end{equation}
In turn, on the strength of assertion (ii) in Lemma \ref{lem.3.01},
from \eqref{e5.04} we deduce the limiting relation \eqref{e5.05}.
Lemma \ref{lem.5.1} is proved. \qed

\begin{remark}
\label{rem.5.2}%
In view of Lemma \ref{lem.3.02}, relations
\eqref{e5.04}, \eqref{e5.05}, \eqref{e5.06}, and \eqref{e5.07} yield that
$$f_0(\boldsymbol{x},s,\lambda)=\chi(\lambda;u_0^{{\rm
tr},(1)}(\boldsymbol{x},s)).$$
\end{remark}

Next, we establish existence of traces on $\Gamma_{\tau-0}^1$ and
$\Gamma_{\tau+0}^1$ for any fixed $\tau\in(0,T)$.
\begin{lemma}
\label{lem.5.2} In the assumptions of Lemma \ref{lem.5.1}, for any
fixed $\tau\in(0,T)$ there exist $u_{\tau-0}^{{\rm tr},(1)}$,
$u_{\tau+0}^{{\rm tr},(1)}\in L^\infty(\Xi^1)$ such that
\begin{equation*}
\underset{t\to\tau\pm0}{\rm ess\,lim}\int\limits_{\Xi^1}\left\vert
u(\boldsymbol{x},t,s)-u_{\tau\pm0}^{{\rm
tr},(1)}(\boldsymbol{x},s)\right\vert\, d\boldsymbol{x}ds=0,
  \end{equation*}
\begin{equation*}
\underset{t\to\tau\pm0}{\rm
ess\,lim}\int\limits_{\Xi^1}\int\limits_{-M_0}^{M_0}\left\vert
\chi(\lambda;u(\boldsymbol{x},t,s))-\chi(\lambda;u_{\tau\pm0}^{{\rm
tr},(1)}(\boldsymbol{x},s))\right\vert\,d\lambda
d\boldsymbol{x}ds=0.
  \end{equation*}
\end{lemma}
\noindent {\it Proof} of this lemma is quite analogous to
justification of Lemma \ref{lem.5.1}. It is based on the natural
modification of Definition \ref{def.5.1} and Proposition
\ref{prop.5.1}, which consists in substitution of $t=0+$ by
$t=\tau-0$ and $t=\tau+0$ in the formulations and relative proofs.
After this substitution, we only need to thoroughly repeat track of
the proof of Lemma \ref{lem.5.1}. \qed\\[1ex]
\indent Existence of strong traces of
$u=\lim\limits_{\varepsilon''\to0}u_{\varepsilon''}$ on
$\Gamma_{0+}^2$ and $\Gamma_{S-0}^2$ immediately follows from the
 result on existence of traces from \cite[Section 3.3]{Kuz-2017},
which, in turn, relies on the techniques initially introduced in
\cite{KV-2007,PAN-2005,PAN-2007,V-2001}, and then developed in
\cite[Lemmas 2 and 4]{Kuz-2015}.
Thus we establish the following lemma.
\begin{lemma}
\label{lem.5.3}%
Let triple $u\in L^\infty(G_{T,S})\cap
L^2((0,T)\times(0,S);\text{\it\r{W}}{}_2^{\,1}(\Omega))$,
$m\in\mathcal{M}^+(G_{T,S}\times[-M_0,M_0])$ and
$n=\delta_{(\lambda=u)}|\nabla_x u|^2$ resolve the kinetic equation
\eqref{e3.01a} in the sense of the integral equality \eqref{e4.12}.
Let coefficients $a,\varphi_1,\ldots,\varphi_d$ and function
$Z_\gamma$ satisfy Conditions on
$a\&\boldsymbol{\varphi}\&Z_\gamma$.

 Then there exist $u_{0+}^{{\rm tr},(2)}$, $u_{S-0}^{{\rm
tr},(2)}\in L^\infty(\Xi^2)$, the traces of $u$ on $\Gamma_{0+}^2$
and $\Gamma_{S-0}^2$, satisfying the limiting
relations \eqref{e3.02} and \eqref{e3.03}, respectively.
\end{lemma}

\begin{remark}
\label{rem.5.3}%
The demand \eqref{e1.02} of genuine nonlinearity is
necessary in assumptions of Lemma \ref{lem.5.3}, and it cannot be
discarded.
\end{remark}

\section{Kinetic initial and final conditions} \label{sec.6}
\begin{lemma}
\label{lem.6.1}%
 Let $u$ resolve the kinetic equation \eqref{e3.01a}
(along with measures $m,n\in\mathcal{M}^+(G_{T,S}\times[-M_0,M_0])$) and be a
strong limiting point of the family
$\{u_\varepsilon\}_{\varepsilon>0}$ of solutions to Problem
$\Pi_{\gamma\varepsilon}$: $u=\mbox{\rm s-}\lim\limits_{\varepsilon''\to0}u_{\varepsilon''}$.
Then the kinetic initial condition \eqref{e3.01b} holds true for
$u$.
  \end{lemma}
  \begin{proof}
From the integral equality \eqref{e2.02}, the maximum principle
\eqref{e2.03} and the energy estimate \eqref{e2.04}, using the
standard techniques, we easily derive the bound
\[
\langle\partial_tu_\varepsilon(t),\phi\rangle_{W^{-1,2}(\Xi^1),\text{\it\r{W}}{}_2^{\,1}(\Xi^1)}\leqslant
C_2\|\phi\|_{W_2^1(\Xi^1)},\quad\forall\,\phi\in
\text{\it\r{W}}{}_2^{\,1}(\Xi^1),\quad\forall\, t\in[0,T], \quad\forall\,
\varepsilon\in(0,1],
\]
where
\[
C_2=\sqrt{S{\rm
meas}(\Omega)}\left(\|a\|_{C[-M_0,M_0]}+\|\boldsymbol{\varphi}\|_{C[-M_0,M_0]}
+\Vert\beta\Vert_{C(\Xi^1\times[-M_0,M_0])}\right)+2\sqrt{C_1},
\]
$C_1$ is the constant from \eqref{e2.04}.

This means that the family $\{\partial_t
u_\varepsilon\}_{\varepsilon\in(0,1]}$ is uniformly bounded in
$L^\infty(0,T;W^{-1.2}(\Xi^1))$. Hence the family
$\{u_\varepsilon\!:\;[0,T]\mapsto
W^{-1,2}(\Xi^1)\}_{\varepsilon\in(0,1]}$ is equicontinuous. In
particular,
\begin{equation}
\label{e6.01i}%
u_\varepsilon(\cdot,t,\cdot)\underset{t\to0+}{\longrightarrow}u_0\text{
strongly in }W^{-1,2}(\Xi^1)\text{ uniformly in
}\varepsilon\in(0,1].
  \end{equation}
On the other hand, due to \eqref{e2.03}, values of the mappings
$t\mapsto u_\varepsilon(\cdot,t,\cdot)$ belong to the set
\[
\{\phi\in L^2(\Xi^1)\!:\;\underset{(\boldsymbol{x},s)\in\Xi^1}{\rm
ess\,sup\,}|\phi(\boldsymbol{x},s)|\leqslant M_0\},
\]
 which is a compact
subset in $W^{-1,2}(\Xi^1)$ by the Rellich theorem. Therefore, by
the Arcel theorem, the set $\{u_\varepsilon\}_{\varepsilon\in(0,1]}$
is relatively compact in $C(0,T;W^{-1,2}(\Xi^1))$. Hence,
\begin{equation}
\label{e6.01ii}%
u_{\varepsilon''}(\cdot,t,\cdot)\underset{\varepsilon''\to0}{\longrightarrow}
u(\cdot,t,\cdot)\text{ strongly in }W^{-1,2}(\Xi^1)\text{ uniformly
on the segment }[0,T].
  \end{equation}
(Recall that $\{\varepsilon''\to0\}$ is the subsequence extracted in
the proof of Lemma \ref{lem.4.1}.) Next, from \eqref{e5.04} it
immediately follows that
\begin{equation}
\label{e6.01iii}%
u(\cdot,t,\cdot)\underset{t\to0+}{\longrightarrow}u_0^{\mathrm{tr},(1)}
\text{ strongly in }W^{-1,2}(\Xi^1).
  \end{equation}
From \eqref{e6.01i}--\eqref{e6.01iii}, by the triangle inequality we
deduce that $u_0^{{\rm tr},(1)}=u_0^{(1)}$. Now, inserting $u_0^{(1)}$ on the place of $u_0^{tr,(1)}$ in \eqref{e5.05}, we
conclude the proof of the lemma.
\end{proof}

In order to derive the kinetic initial and final conditions
\eqref{e3.01c} and \eqref{e3.01d}, we will prove and then apply the
following lemma.
\begin{lemma}
\label{lem.6.2}%
Under assumptions of Lemma \ref{lem.6.1}, the limiting relations
\begin{equation}
\label{e6.02}%
\lim\limits_{\varepsilon''\to0+} \int\limits_{\Xi^2}
\kappa(\boldsymbol{x},t)(-\varepsilon''\partial_s
u_{\varepsilon''}(\boldsymbol{x},t,0))\,d\boldsymbol{x}dt=
\int\limits_{\Xi^2} \kappa(\boldsymbol{x},t)\left( a(u_0^{{\rm
tr},(2)}(\boldsymbol{x},t))-a(u_0^{(2)}(\boldsymbol{x},t))\right)\,d\boldsymbol{x}dt,
  \end{equation}
\begin{equation}
\label{e6.03}%
\lim\limits_{\varepsilon''\to0+} \int\limits_{\Xi^2}
\kappa(\boldsymbol{x},t)(-\varepsilon''\partial_s
u_{\varepsilon''}(\boldsymbol{x},t,S))\,d\boldsymbol{x}dt=
\int\limits_{\Xi^2} \kappa(\boldsymbol{x},t)\left( a(u_S^{{\rm
tr},(2)}(\boldsymbol{x},t))-a(u_S^{(2)}(\boldsymbol{x},t))\right)\,d\boldsymbol{x}dt
  \end{equation}
hold true for all test-functions $\kappa\in C_0^1(\Xi^2)$.
  \end{lemma}
\proof At first, notice that the weak traces of $u_\varepsilon$ and
$\partial_s u_\varepsilon$ are well-defined on $\Gamma_{\tilde
s}^2=\bar{\Omega}\times[0,T]\times\{s=\tilde{s}\}$ for all
$\tilde{s}\in[0,S]$ and that the integral equality \eqref{e2.02}
is equivalent to the integral equality
\begin{multline}
\label{e6.04}
\int\limits_{s'}^{s''}\int\limits_{\Xi^2}\big(-u_\varepsilon\partial_t\phi
-a(u_\varepsilon)\partial_s\phi-\boldsymbol{\varphi}(u_\varepsilon)\cdot
\nabla_x\phi+\nabla_xu_\varepsilon\cdot\nabla_x\phi
+\varepsilon\partial_su_\varepsilon\partial_s\phi-Z_\gamma(\boldsymbol{x},t,s,u_\varepsilon)\phi\big)\,d\boldsymbol{x}dtds
\\
+\int\limits_{\Xi^2}a(u_\varepsilon(\boldsymbol{x},t,s''))
\phi(\boldsymbol{x},t,s'')\,d\boldsymbol{x}dt
-\int\limits_{\Xi^2}a(u_\varepsilon(\boldsymbol{x},t,s'))
\phi(\boldsymbol{x},t,s')\,d\boldsymbol{x}dt
\\
+\int\limits_{\Xi^2}(-\varepsilon\partial_s
u_\varepsilon(\boldsymbol{x},t,s'')
\phi(\boldsymbol{x},t,s''))\,d\boldsymbol{x}dt
-\int\limits_{\Xi^2}(-\varepsilon\partial_s
u_\varepsilon(\boldsymbol{x},t,s')
\phi(\boldsymbol{x},t,s'))\,d\boldsymbol{x}dt=0
\end{multline}
for all $s',s''\in[0,S]$ ($s'\leqslant s''$) and for all
test-functions $\phi\in L^\infty(G_{T,S})\cap W_2^1(G_{T,S})$
vanishing in the neighborhood of the planes $\{t=0\}$, $\{t=T\}$
and the boundary $\partial\Omega$.

Existence of weak traces of $u_\varepsilon$ and $\partial_s
u_\varepsilon$ on $\Gamma_{\tilde{s}}^2$ can be proved in a quite
similar way as in \cite[Proposition A.1]{Kuzn-Sazh-1-2018}. Proof of
the equivalence of \eqref{e6.04} to \eqref{e2.02} is quite similar
to the proof in \cite[Chapter 3, Section 1.1]{AKM-1990}. Therefore,
we skip these proofs in this article.

Like in \cite{K-1970} and \cite{DL-1988}, let us introduce a family
of functions
$\{\rho_{\delta}^0:\mathbb{R}^+\mapsto[0,1]\}_{\delta>0}$ such that
\begin{multline}
\label{e6.05}%
\{\rho_\delta^0\}\subset C^2(\mathbb{R}_s^+),\quad
\rho_\delta^0(s)=0,\quad \forall s>\delta,\quad \rho_\delta^0(0)=1,
\quad |(\rho_\delta^0)'(\cdot)\vert\leqslant\frac c\delta
\\
\text{~with a constant~}c>0 \text{~independent of~}\delta,
  \end{multline}
and
\begin{equation}
\label{e6.07-Phi}%
\lim\limits_{\delta\to0+}\int\limits_0^\delta\Phi(s)(\rho_\delta^0)'(s)\,ds=-\Phi(0+)
  \end{equation}
for any integrable in a neighborhood of
$\{s=0\}\subset\mathbb{R}^+$ function $\Phi$ having the trace on
$s=0$ from the right, $\Phi(0+)=\lim\limits_{s\to0+}\Phi(s)$.

In \eqref{e6.04} fix $\delta>0$, $s'=0$, $s''>\delta$ and take
$\phi(\boldsymbol{x},t,s)=\kappa(\boldsymbol{x},t)\rho_\delta^0(s)$,
where $\kappa\in C_0^1(\Xi^2)$ is an arbitrary test-function. On the
strength of Proposition \ref{prop.2.1}, Lemma \ref{lem.4.1} and
properties \eqref{e6.05}, passing to the limit as
$\varepsilon:=\varepsilon''\to0+$, from \eqref{e6.04} we derive the
relation
\begin{multline}
\label{e6.02i}%
\int\limits_0^\delta\rho_\delta^0(s)\int\limits_{\Xi^2}\left(-u\partial_t\kappa
-\boldsymbol{\varphi}(u)\cdot\nabla_x\kappa+\nabla_x u\cdot\nabla_x
\kappa-Z_\gamma(\boldsymbol{x},t,s,u)\kappa\right)\,d\boldsymbol{x}dtds
-\\
\int\limits_0^\delta\left(\rho_\delta^0\right)'(s)
\int\limits_{\Xi^2}a(u)\kappa\,d\boldsymbol{x}dtds
-\int\limits_{\Xi^2}a(u_0^{(2)})\kappa\,d\boldsymbol{x}dt
-\\
\lim\limits_{\varepsilon''\to0+}\int\limits_{\Xi^2}
\left(-\varepsilon''\partial_s
u_{\varepsilon''}(\boldsymbol{x},t,0)\right)\kappa(\boldsymbol{x},t)\,
d\boldsymbol{x}dt=0.
  \end{multline}

Further, \eqref{e6.02} directly follows from \eqref{e6.02i} as
$\delta\to0+$, thanks to Lemma \ref{lem.5.3} and property
\eqref{e6.07-Phi}. And, finally, \eqref{e6.03} is deduced similarly
by means of the test-function $\rho_\delta^S(s)=\rho_\delta^0(S-s)$,
choosing $s'<S-\delta$ and $s''=S$ in \eqref{e6.04}. This
observation completes the proof. \qed\\[1ex]
With the help of Lemma \ref{lem.6.2}, we establish the following
result.
\begin{lemma}
\label{lem.6.3}%
Under assumptions of Lemma \ref{lem.6.1}, kinetic initial and final
conditions \eqref{e3.01c} and \eqref{e3.01d} are valid.
  \end{lemma}
\proof In \eqref{e6.04} fix $s'=0$, $s''>\delta$, where $\delta$ is
an arbitrary rather small positive constant, and take the
test-function
\[
\phi(\boldsymbol{x},t,s):=\eta'(u_\varepsilon(\boldsymbol{x},t,s))\Theta(\boldsymbol{x},t)
\rho_\delta^0(s),
\]
where $\eta\in C^2(\mathbb{R})$ is convex, $\Theta\in C_0^1(\Xi^2)$
is nonnegative, and $\rho_\delta^0$ is the same as in the proof of
Lemma \ref{lem.6.2}. Due to sufficient regularity of
$u_\varepsilon$, such choice of $\phi$ is legitime. Using
integration by parts and the chain rule, from \eqref{e6.04} we
deduce the integral equality
\begin{multline}
\label{e6.07}%
\int\limits_0^\delta \int\limits_{\Xi^2}\big(
-\eta(u_\varepsilon)\rho_\delta^0\partial_t\Theta
-q_a(u_\varepsilon)(\rho_\delta^0)'\Theta-
\rho_\delta^0\boldsymbol{q}_\varphi(u_\varepsilon)\cdot\nabla_x\Theta
-Z_\gamma(\boldsymbol{x},t,s,u_\varepsilon)\eta'(u_\varepsilon)\Theta\rho_\delta^0+
\\
\eta''(u_\varepsilon)\left\vert\nabla_xu_\varepsilon\right\vert^2\Theta\rho_\delta^0
+\rho_\delta^0\nabla_x\eta(u_\varepsilon)\cdot\nabla_x\Theta+
\varepsilon\eta''(u_\varepsilon)
\left\vert\partial_su_\varepsilon\right\vert^2\Theta\rho_\delta^0+
\varepsilon\left(\rho_\delta^0\right)'
\partial_s\eta(u_\varepsilon)\Theta\big)\,d\boldsymbol{x}dtds-
\\
\int\limits_{\Xi^2}q_a(u_0^{(2)})\Theta\,d\boldsymbol{x}dt
-\int\limits_{\Xi^2}(-\varepsilon\partial_s
u_\varepsilon(\boldsymbol{x},t,0))\eta'(u_0^{(2)})\Theta\,d\boldsymbol{x}dt=0.
\end{multline}
Remark that
\begin{equation}
\label{e6.08}%
\eta''(u_\varepsilon)\left(\left\vert\nabla_xu_\varepsilon\right\vert^2+
\varepsilon\left\vert\partial_su_\varepsilon\right\vert^2
\right)\Theta\rho_\delta^0\geqslant 0\quad \text{in
}\Xi^2\times(0,\delta),
  \end{equation}
due to the choice of $\eta$, $\Theta$ and $\rho_\delta^0$, and that
\begin{multline}
\label{e6.09i}%
\lim\limits_{\varepsilon''\to0}
\int\limits_0^\delta\int\limits_{\Xi^2}
 \big(-\eta(u_{\varepsilon''})\rho_\delta^0\partial_t\Theta-q_a(u_{\varepsilon''})(\rho_\delta^0)'\Theta
-\rho_\delta^0\boldsymbol{q}_\varphi(u_{\varepsilon''})\cdot\nabla_x\Theta
-Z_\gamma(\boldsymbol{x},t,s,u_{\varepsilon''})\eta'(u_{\varepsilon''})\Theta\rho_\delta^0
\\
+\rho_\delta^0\nabla_x\eta(u_{\varepsilon''})\cdot\nabla_x\Theta+\varepsilon''(\rho_\delta^0)'
\partial_s\eta(u_{\varepsilon''})\Theta\big)\,d\boldsymbol{x}dtds=
\int\limits_0^\delta\int\limits_{\Xi^2}
\big(-\eta(u)\rho_\delta^0\partial_t\Theta
-q_a(u)(\rho_\delta^0)'\Theta
\\
-\rho_\delta^0\boldsymbol{q}_\varphi(u)\cdot\nabla_x\Theta
-Z_\gamma(\boldsymbol{x},t,s,u)\eta'(u)\Theta\rho_\delta^0
+\rho_\delta^0\nabla_x\eta(u)\cdot\nabla_x\Theta\big)\,d\boldsymbol{x}dtds
\end{multline}
due to Lemma \ref{lem.4.1} and the structure of function
$\rho_\delta^0$.

On the strength of \eqref{e6.08} and \eqref{e6.09i}, from
\eqref{e6.07} we derive the inequality
\begin{multline}
\label{e6.09} \lim\limits_{\varepsilon''\to0}\int\limits_{\Xi^2}
(-\varepsilon''\partial_s
u_{\varepsilon''}(\boldsymbol{x},t,0))\eta'(u_0^{(2)})\Theta\,d\boldsymbol{x}dt
 -\int\limits_0^\delta\int\limits_{\Xi^2}
 \big(-\eta(u)\rho_\delta^0\partial_t\Theta-q_a(u)(\rho_\delta^0)'\Theta
-\rho_\delta^0\boldsymbol{q}_\varphi(u)\cdot\nabla_x\Theta
\\
-Z_\gamma(\boldsymbol{x},t,s,u)\eta'(u)\Theta\rho_\delta^0
+\rho_\delta^0\nabla_x\eta(u)\cdot\nabla_x\Theta\big)\,d\boldsymbol{x}dtds
+\int\limits_{\Xi^2}q_a(u_0^{(2)})\Theta\,d\boldsymbol{x}dt\geqslant0,
\end{multline}
as $\varepsilon''\to0+$. Further, remark that
\begin{equation}
\label{e6.10i}%
\lim\limits_{\varepsilon''\to0}\int\limits_{\Xi^2}
(-\varepsilon''\partial_s
u_{\varepsilon''}(\boldsymbol{x},t,0))\eta'(u_0^{(2)})\Theta\,d\boldsymbol{x}dt
= \int\limits_{\Xi^2}\left(a(u_0^{{\rm
tr},(2)})-a(u_0^{(2)})\right)\eta'(u_0^{(2)})\Theta\,d\boldsymbol{x}dt
\end{equation}
due to Lemma \ref{lem.6.2},
\begin{equation}
\label{e6.10ii}%
\lim\limits_{\delta\to0+}\int\limits_0^\delta\int\limits_{\Xi^2}
q_a(u)(\rho_\delta^0)'\Theta\,d\boldsymbol{x}dtds=
-\int\limits_{\Xi^2}
q_a(u_0^{\mathrm{tr},(2)})\Theta\,d\boldsymbol{x}dt
\end{equation}
due to property \eqref{e6.07-Phi} of the derivative $(\rho_\delta^0)'$.

On the strength of \eqref{e6.10i}, \eqref{e6.10ii}, and the bounds
$|u|\leqslant M_0$ (a.e. in $G_{T,S}$) and
$0\leqslant\rho_\delta^0\leqslant 1$, as $\delta\to0+$ from \eqref{e6.09} we deduce
\begin{equation}
\label{e6.10}%
\int\limits_{\Xi^2}\left((a(u_0^{{\rm
tr},(2)})-a(u_0^{(2)}))\eta'(u_0^{(2)})-(q_a(u_0^{{\rm
tr},(2)})-q_a(u_0^{(2)}))\right)\Theta\,d\boldsymbol{x}dt\geqslant0.
\end{equation}
Quite analogously, we derive the integral inequality
\begin{equation}
\label{e6.11}%
\int\limits_{\Xi^2}\left(\left(a(u_S^{{\rm
tr},(2)})-a(u_S^{(2)})\right)\eta'(u_S^{(2)})-\left(q_a(u_S^{{\rm
tr},(2)})-q_a(u_S^{(2)})\right)\right)\Theta\,d\boldsymbol{x}dt\leqslant0.
\end{equation}
To this end, it suffices to substitute $s'=0$, $s''>\delta$ and
$\phi=\eta'(u_\varepsilon)\Theta\rho_\delta^0$ by $s'<S-\delta$,
$s''=S$ and $\phi=\eta'(u_\varepsilon)\Theta\rho_\delta^S$ (with
$\rho_\delta^S(s)=\rho_\delta^0(S-s)$), respectively, and to keep
track of the above outline.

Now, let us prove that \eqref{e6.10} is equivalent to \eqref{e3.01c}.
This procedure is rather standard and simple. Introduce the linear
functional $\nu_0^{(2)}$ by the formula
\begin{multline}
\label{e6.12}%
\left\langle\nu_0^{(2)},\Phi\right\rangle
_{\mathcal{M}(\Xi^2\times\mathbb{R}_\lambda),C_0(\Xi^2\times\mathbb{R}_\lambda)}:=\\
\int\limits_{\Xi^2}\Big[\bigl(a(u_0^{{\rm
tr},(2)}(\boldsymbol{x},t))-a(u_0^{(2)}(\boldsymbol{x},t))\bigr)\left\langle
\delta_{(\cdot=u_0^{(2)}(\boldsymbol{x},t))},\Phi(\boldsymbol{x},t,\cdot)
\right\rangle_{\mathcal{M}(\mathbb{R}_\lambda),C_0(\mathbb{R}_\lambda)}-
\\
\int\limits_{-M_0}^{M_0}a'(\lambda)\bigl(\chi(\lambda;u_0^{{\rm
tr},(2)}(\boldsymbol{x},t))-\chi(\lambda;u_0^{(2)}(\boldsymbol{x},t))\bigr)\Phi(\boldsymbol{x},t,\lambda)\,d\lambda\Big]\,d\boldsymbol{x}dt
 \end{multline}
for all $\Phi\in C_0(\Xi^2\times\mathbb{R}_\lambda)$.
On the strength of the maximum principle \eqref{e2.03}, we have that $\vert u_0^{{\rm tr},(2)}\vert\leqslant M_0$ a.e. on
$\Xi^2$. Hence, ${\rm supp}\,\nu_0^{(2)}$ lays in the layer
$\{-M_0\leqslant \lambda \leqslant M_0\}$, and there exists $C_3={\rm const}>0$
independent of $\Phi$ such that
\[
\left\vert\left\langle\nu_0^{(2)},\Phi\right\rangle\right\vert\leqslant
C_3\left\Vert\Phi\right\Vert_{L^1(\Xi^2;C[-M_0,M_0])},\quad\forall\Phi\in
L^1(\Xi^2;C[-M_0,M_0]),
\]
 i.e. $\nu_0^{(2)}\in L_w^\infty(\Xi^2;\mathcal{M}[-M_0,M_0])$.

Further, introduce the primitive $\mu_0^{(2)}$ of $\nu_0^{(2)}$ with
respect to $\lambda$:
\begin{equation}
\label{e6.13}%
\left\langle\mu_0^{(2)},\Phi'\right\rangle_{\mathcal{M}(\mathbb{R}_\lambda),C_0(\mathbb{R}_\lambda)}=
-\left\langle\nu_0^{(2)},\Phi\right\rangle_{\mathcal{M}(\mathbb{R}_\lambda),C_0(\mathbb{R}_\lambda)}
\text{ a.e. in }\Xi^2,\quad\forall\,\Phi\in C_0^1(\mathbb{R}_\lambda).
  \end{equation}

Clearly, $\mu_0^{(2)}\in
L_w^\infty(\Xi^2;\mathcal{M}(\mathbb{R}_\lambda))$, and ${\rm
supp}\,\mu_0^{(2)}$ lays in the layer $\{-M_0\leqslant \lambda \leqslant M_0\}$.
Using assertion (i) in Lemma \ref{lem.3.01} and formulas
\eqref{e6.12} and \eqref{e6.13}, we rewrite \eqref{e6.10} in terms
of the mapping $\mu_0^{(2)}$ as follows:
\begin{equation}
\label{e6.14}%
\left\langle\partial_\lambda\mu_0^{(2)},\Theta\eta'\right\rangle\leqslant0,\quad
\forall\,\Theta\in C_0^1(\Xi^2),\quad\forall\,\eta\in
C_0^2(\mathbb{R}_\lambda),\quad\eta\text{ is convex on }[-M_0,M_0].
\end{equation}

In the sense of distributions, \eqref{e6.14} is equivalent to the
inequality $\left\langle\mu_0^{(2)},\Theta\eta''\right\rangle
\geqslant 0$. In turn, since $\Theta\geqslant0$, $\eta''\geqslant0$
and the linear span of the set $\left\{\Theta\eta''\!:\,\Theta\in
C_0^1(\Xi^2),\,\eta\in C_0^2(\mathbb{R}_\lambda)\right\}$ is dense
in $L^1(\Xi^2;C[-M_0,M_0])$, from \eqref{e6.12}--\eqref{e6.14} it
follows that $\mu_0^{(2)}$ is a positive Radon measure on
$[-M_0,M_0]$ for a.e. $(\boldsymbol{x},t)\in\Xi^2$, i.e.,
\begin{equation*}
\mu_0^{(2)}\in L_w^\infty(\Xi^2;\mathcal{M}^+[-M_0,M_0]),
  \end{equation*}
and the integral equality
\begin{multline}
\label{e6.15}%
\int\limits_{\Xi^2}\Biggl(\,\int\limits_{-M_0}^{M_0} a'(\lambda)\left(
\chi(\lambda;u_0^{{\rm
tr},(2)})-\chi(\lambda;u_0^{(2)})\right)\Phi\,d\lambda-\\
\left(a(u_0^{{\rm
tr},(2)})-a(u_0^{(2)})\right)\left\langle
\delta_{(\cdot\,=u_0^{(2)})},\Phi
\right\rangle_{\mathcal{M}(\mathbb{R}_\lambda),C_0(\mathbb{R}_\lambda)}\Biggr)\,d\boldsymbol{x}dt=
\\
-\int\limits_{\Xi^2}\left\langle\mu_0^{(2)}(\boldsymbol{x},t,\cdot),\partial_\lambda
\Phi(\boldsymbol{x},t,\cdot)\right\rangle_{\mathcal{M}(\mathbb{R}_\lambda),\,
C_0(\mathbb{R}_\lambda)}\,d\boldsymbol{x}dt,\quad\forall\Phi\in
L^1(\Xi^2;C_0^1(\mathbb{R}_\lambda))
  \end{multline}
holds true. In the sense of distributions, \eqref{e6.15} is
equivalent to \eqref{e3.01c}. Quite analogously, we verify that
\eqref{e6.11} is equivalent to \eqref{e3.01d}. Lemma \ref{lem.6.3}
is proved. \qed

\begin{remark}
\label{rem.6.1}%
Collecting altogether the results of Lemmas \ref{lem.4.1},
\ref{lem.4.2}, \ref{lem.6.1}, and \ref{lem.6.3}, we establish that
there is an $L^1$-strong limiting point $u$ of the family
$\{u_\varepsilon\}_{\varepsilon>0}$ of solutions to Problem
$\Pi_{\gamma\varepsilon}$ such that $u$ is a kinetic solution of
Problem $\Pi_\gamma$ in the sense of Definition \ref{def.3.02}.
  \end{remark}

\section{Stability and uniqueness of kinetic solutions} \label{sec.7}
Let $u_1$ and $u_2$ be two kinetic solutions of Problem $\Pi_\gamma$
corresponding to two given sets of data $(u_{1,0}^{(1)}$,
$u_{1,0}^{(2)},$ $u_{1,S}^{(2)})$ and $(u_{2,0}^{(1)},$
$u_{2,0}^{(2)},$ $u_{2,S}^{(2)})$, respectively. In order to verify
inequality \eqref{e3.05}, we revisit, modify, and somewhat improve
the procedure outlined in \cite[Section 4]{Kuz-2017}. We start with
fulfilling the normalization procedure for the difference of two
kinetic equations for $u_1$ and $u_2$. More precisely, we prove the
following lemma.
\begin{lemma}
\label{prop.7.1}%
For an arbitrary fixed $t'\in(0,T]$ and for all nonnegative
functions $\xi\in C^2(\Omega)$ the renormalized inequality
\begin{multline}
\label{e7.01}
\int\limits_{\Xi^1\times(-M_1,M_1)}|\chi(\lambda;u_{1,t'-0}^{{\rm
tr},(1)}(\boldsymbol{x},s))-\chi(\lambda;u_{2,t'-0}^{{\rm
tr},(1)}(\boldsymbol{x},s))|^2\xi(\boldsymbol{x})\,
d\boldsymbol{x}dsd\lambda
-\\ \int\limits_{(0,t')\times\Xi^1\times(-M_1,M_1)}
\Big(\boldsymbol{\varphi}'(\lambda)\cdot\nabla_x\xi(\boldsymbol{x})
+\Delta_x\xi(\boldsymbol{x})+\partial_\lambda
Z_\gamma(\boldsymbol{x},t,s,\lambda)\xi(\boldsymbol{x})\Big)\times \\ \left\vert\chi(\lambda;u_1(\boldsymbol{x},t,s))
-\chi(\lambda;u_2(\boldsymbol{x},t,s))\right\vert^2\,d\boldsymbol{x}dtdsd\lambda
\leqslant
\\
\int\limits_{\Xi^1\times(-M_1,M_1)}|\chi(\lambda;u_{1,0}^{{\rm
tr},(1)}(\boldsymbol{x},s))-\chi(\lambda;u_{2,0}^{{\rm
tr},(1)}(\boldsymbol{x},s))|^2\xi(\boldsymbol{x})\,d\boldsymbol{x}dsd\lambda+
\\
\int\limits_{(0,t')\times\Omega\times(-M_1,M_1)}
a'(\lambda)\Bigl(|\chi(\lambda;u_{1,0}^{{\rm
tr},(2)}(\boldsymbol{x},t))-\chi(\lambda;u_{2,0}^{{\rm
tr},(2)}(\boldsymbol{x},t))|^2-
\\
|\chi(\lambda;u_{1,S}^{{\rm
tr},(2)}(\boldsymbol{x},t))-\chi(\lambda;u_{2,S}^{{\rm
tr},(2)}(\boldsymbol{x},t))|^2\Bigr)\xi(\boldsymbol{x})\,
d\boldsymbol{x}dtd\lambda
\end{multline}
is valid, where $M_1=M_1(t')$ is given by \eqref{e7.01bis}.
\end{lemma}
\proof The proof is divided into five steps.\\[1ex]
{\it Step 1. The smoothing of the kinetic equation.}
 Introduce the regularizing kernel $\omega\in
C_0^\infty(\mathbb{R})$, $\|\omega\|_{L^1(\mathbb{R})}=1$ that is a
nonnegative smooth function with a compact support on $[0,1]$. For
any measurable function or enough regular distribution
$f$: $\mathbb{R}_x^d\times\mathbb{R}_t^+\times\mathbb{R}_s^+\times\mathbb{R}_\lambda\mapsto\mathbb{R}$
we denote
\begin{equation*}
\begin{split}
 f_{\varepsilon_0}(\boldsymbol{x},t,s,\cdot)=\omega_{\varepsilon_0}*f(\boldsymbol{x},t,s,\cdot),
\quad
f_{\varepsilon_1}(\boldsymbol{x},t,\cdot,\lambda)=\omega_{\varepsilon_1}*f(\boldsymbol{x},t,\cdot,\lambda),\\
 f_{\varepsilon_2}(\boldsymbol{x},\cdot,s,\lambda)=\omega_{\varepsilon_2}*f(\boldsymbol{x},\cdot,s,\lambda),
\quad
f_{\varepsilon_3}(\cdot,t,s,\lambda)=\omega_{\varepsilon_3}*f(\cdot,t,s,\lambda),
\end{split}
  \end{equation*}
where $\varepsilon_0$, $\varepsilon_1$, $\varepsilon_2$ and
$\varepsilon_3$ are small positive parameters, and
\begin{equation*}
\begin{split}
& \omega_{\varepsilon_0}(\lambda)=\frac1{\varepsilon_0}\omega\left(\frac{\lambda}{\varepsilon_0}\right),
\quad\omega_{\varepsilon_1}(s)=\frac1{\varepsilon_1}\omega\left(\frac{s}{\varepsilon_1}\right),\\
& \omega_{\varepsilon_2}(t)=\frac1{\varepsilon_2}\omega\left(\frac{t}{\varepsilon_2}\right),
\quad\omega_{\varepsilon_3}(\boldsymbol{x})=\frac1{\varepsilon_3^d}
\omega\left(\frac{x_1}{\varepsilon_3}\right)\cdot\ldots\cdot\omega\left(\frac{x_d}{\varepsilon_3}\right).
\end{split}
  \end{equation*}
Further we write $f_{\alpha\beta}$ instead of $(f_{\alpha})_{\beta}$
for
$\alpha,\beta=\varepsilon_0,\varepsilon_1,\varepsilon_2,\varepsilon_3$
and denote
$f_{\vec{\varepsilon}}=f_{\varepsilon_0,\varepsilon_1,\varepsilon_2,\varepsilon_3}$
for the sake of brevity. Set
\[
\Phi(\boldsymbol{x},t,s,\lambda)=\omega_{\varepsilon_0}(\bar{\lambda}-\lambda)
\omega_{\varepsilon_1}(\bar{s}-s)\omega_{\varepsilon_2}(\bar{t}-t)
\omega_{\varepsilon_3}(\bar{\boldsymbol{x}}-\boldsymbol{x})
\]
for $\bar{\lambda}\in\mathbb{R}_\lambda$,
$\bar{s}\in[\varepsilon_1,S]$, $\bar{t}\in[\varepsilon_2,T]$ and
$\bar{\boldsymbol{x}}\in\overline{\Omega}_{\varepsilon_3}\overset{\mathrm{def}}{=}
\{\boldsymbol{x}'\in\Omega:\,
\mathrm{dist}(\boldsymbol{x}',\partial\Omega)\geqslant\varepsilon_3\}$.
Notice that this is a legitimate test function for \eqref{e4.12}.

Substituting $\Phi(\boldsymbol{x},t,s,\lambda)$ on the place of a
test function in \eqref{e4.12}, we obtain the following equation for
the smoothed $\chi$-function, where we write
$\boldsymbol{x}$, $t$, $s$, and $\lambda$ instead of
$\bar{\boldsymbol{x}}$, $\bar{t}$, $\bar{s}$ and $\bar{\lambda}$:
\begin{multline}
\label{e7.01i}
\partial_t\chi_{\vec{\varepsilon}}(\lambda;u)+a'(\lambda)
\partial_s\chi_{\vec{\varepsilon}}(\lambda;u)+\boldsymbol{\varphi}'(\lambda)\cdot
\nabla_x\chi_{\vec{\varepsilon}}(\lambda;u)
-\Delta_x\chi_{\vec{\varepsilon}}(\lambda;u)+Z_\gamma(\boldsymbol{x},t,s,\lambda)
\partial_\lambda\chi_{\vec{\varepsilon}}(\lambda;u)
\\
-(Z_\gamma)_{\varepsilon_1,\varepsilon_2,\varepsilon_3}(\boldsymbol{x},t,s,0)\omega_{\varepsilon_0}(\lambda)=
\partial_\lambda(m_{\vec{\varepsilon}}+n_{\vec{\varepsilon}})+R_1^{(\vec{\varepsilon})}+
R_2^{(\vec{\varepsilon})}+R_3^{(\vec{\varepsilon})}+
R_4^{(\vec{\varepsilon})}
\\
\text{ in
}\overline{\Omega}_{\varepsilon_3}\times[\varepsilon_2,T]\times[\varepsilon_1,S]\times\mathbb{R}_\lambda,
\end{multline}
where
\[
\chi_{\vec{\varepsilon}}(\lambda;u)=\chi_{\varepsilon_0,\varepsilon_1,\varepsilon_2,\varepsilon_3}(\cdot;u(\cdot,\cdot,\cdot)),
\]
and the rest terms are given by the formulas
\begin{equation*}
\begin{split}
& R_1^{(\vec{\varepsilon})}=a'(\lambda)\partial_s\chi_{\vec{\varepsilon}}(\lambda;u)-
\partial_s(a'\chi)_{\vec{\varepsilon}},\quad
R_2^{(\vec{\varepsilon})}={\rm
div}_x\left(\boldsymbol{\varphi}'(\lambda)\chi_{\vec{\varepsilon}}(\lambda;u)-
(\boldsymbol{\varphi}'(\cdot)\chi)_{\vec{\varepsilon}}\right),\\
& R_3^{(\vec{\varepsilon})}=(\chi\partial_\lambda
Z_\gamma)_{\vec{\varepsilon}}-\chi_{\vec{\varepsilon}}\partial_\lambda
Z_\gamma(\boldsymbol{x},t,s,\lambda),\quad
R_4^{(\vec{\varepsilon})}=
\partial_\lambda\left(\chi_{\vec{\varepsilon}}(\lambda;u)Z_\gamma(\boldsymbol{x},t,s,\lambda)-(\chi
Z_\gamma)_{\vec{\varepsilon}}\right).
\end{split}
\end{equation*}
{\it Step 2. Renormalization of the smoothed kinetic equation.}
Now subtract \eqref{e7.01i} with
$$\chi_{\vec{\varepsilon}}(\lambda;u)=\chi_{\vec{\varepsilon}}(\lambda;u_2),\quad
m_{\vec{\varepsilon}}=m_{2\vec{\varepsilon}},\quad
n_{\vec{\varepsilon}}=n_{2\vec{\varepsilon}},\quad
R_k^{(\vec{\varepsilon})}=R_{2k}^{(\vec{\varepsilon})} \quad (k=1,2,3,4)$$
from \eqref{e7.01i} with
$$\chi_{\vec{\varepsilon}}(\lambda;u)=\chi_{\vec{\varepsilon}}(\lambda;u_1),\quad
m_{\vec{\varepsilon}}=m_{1\vec{\varepsilon}},\quad
n_{\vec{\varepsilon}}=n_{1\vec{\varepsilon}},\quad
R_k^{(\vec{\varepsilon})}=R_{1k}^{(\vec{\varepsilon})}\quad
(k=1,2,3,4)$$
and multiply the both sides of the resulting equation
by
$2\left(\chi_{\vec{\varepsilon}}(\lambda;u_1)-\chi_{\vec{\varepsilon}}(\lambda;u_2)\right)$
to get
\begin{multline}
\label{e7.01ii}
\partial_t\vert\chi_{\vec{\varepsilon}}(\lambda;u_1)-\chi_{\vec{\varepsilon}}(\lambda;u_2)\vert^2
+a'(\lambda)\partial_s\vert\chi_{\vec{\varepsilon}}(\lambda;u_1)-\chi_{\vec{\varepsilon}}(\lambda;u_2)\vert^2
+\\ \boldsymbol{\varphi}'(\lambda)\cdot\nabla_x\vert\chi_{\vec{\varepsilon}}(\lambda;u_1)-\chi_{\vec{\varepsilon}}(\lambda;u_2)\vert^2
-\Delta_x\vert\chi_{\vec{\varepsilon}}(\lambda;u_1)-\chi_{\vec{\varepsilon}}(\lambda;u_2)\vert^2
+\\
2\vert\nabla_x(\chi_{\vec{\varepsilon}}(\lambda;u_1)-\chi_{\vec{\varepsilon}}(\lambda;u_2))\vert^2+
Z_\gamma(\boldsymbol{x},t,s,\lambda)\partial_\lambda\vert\chi_{\vec{\varepsilon}}(\lambda;u_1)-\chi_{\vec{\varepsilon}}(\lambda;u_2)\vert^2
=\\
2\left(\chi_{\vec{\varepsilon}}(\lambda;u_1)-\chi_{\vec{\varepsilon}}(\lambda;u_2)\right)
\Bigl(\partial_\lambda\left((m_{1\vec{\varepsilon}}+n_{1\vec{\varepsilon}})-
(m_{2\vec{\varepsilon}}+n_{2\vec{\varepsilon}})\right)
+\sum\limits_{k=1}^4\bigl(R_{1k}^{(\vec{\varepsilon})}-R_{2k}^{(\vec{\varepsilon})}\bigr)\Bigr)
\\
\text{ in
}\overline{\Omega}_{\varepsilon_3}\times[\varepsilon_2,T]\times[\varepsilon_1,S]\times\mathbb{R}_\lambda.
\end{multline}
Since $u_1\vert_{\Gamma_l}=u_2\vert_{\Gamma_l}=0$ by \eqref{e1.01d},
equation \eqref{e7.01ii} can be extended as the trivial identity
beyond $\overline{\Omega}_{\varepsilon_3}$ onto the whole space
$\mathbb{R}_x^d$. Let $B$ be an arbitrary open ball in
$\mathbb{R}_x^d$ containing $\overline{\Omega}$. In particular, we
have $\mathrm{dist}(\partial\Omega,\partial B)>0$. Let $\zeta\in
C_0^2(B)$ be arbitrary and not necessarily finite in $\Omega$.
Multiply the both sides of \eqref{e7.01ii} by $\zeta$, integrate
over $(\varepsilon_2,t')\times(\varepsilon_1,s')\times
B\times\mathbb{R}_\lambda$ $(s'\in(0,S)$, $t'\in[\varepsilon_2,T))$
and integrate by parts in $\boldsymbol{x}$ and $\lambda$:
\begin{multline}
\label{e7.01iii}%
\int\limits_{B\times(\varepsilon_2,t')\times(\varepsilon_1,s')\times\mathbb{R}_\lambda}
\zeta(\boldsymbol{x})\left(\partial_t\vert\chi_{\vec{\varepsilon}}(\lambda;u_1)-\chi_{\vec{\varepsilon}}(\lambda;u_2)\vert^2
+a'(\lambda)\partial_s\vert\chi_{\vec{\varepsilon}}(\lambda;u_1)-\chi_{\vec{\varepsilon}}(\lambda;u_2)\vert^2
\right)\,d\boldsymbol{x}dtdsd\lambda-
\\
\int\limits_{B\times(\varepsilon_2,t')\times(\varepsilon_1,s')\times\mathbb{R}_\lambda}
\bigl(\boldsymbol{\varphi}'(\lambda)\cdot\nabla_x\zeta(\boldsymbol{x})+
\Delta_x\zeta(\boldsymbol{x})+  \hspace{6cm}  \\ \partial_\lambda
Z_\gamma(\boldsymbol{x},t,s,\lambda)\zeta(\boldsymbol{x})\bigr)
\vert\chi_{\vec{\varepsilon}}(\lambda;u_1)-\chi_{\vec{\varepsilon}}(\lambda;u_2)\vert^2
\,d\boldsymbol{x}dtdsd\lambda+
\\
\int\limits_{B\times(\varepsilon_2,t')\times(\varepsilon_1,s')\times\mathbb{R}_\lambda}
2\zeta(\boldsymbol{x})\vert\nabla_x(\chi_{\vec{\varepsilon}}(\lambda;u_1)-\chi_{\vec{\varepsilon}}(\lambda;u_2))\vert^2
\,d\boldsymbol{x}dtdsd\lambda=
\\
\int\limits_{B\times(\varepsilon_2,t')\times(\varepsilon_1,s')\times\mathbb{R}_\lambda}
\Bigl(
-2\zeta(\boldsymbol{x})\partial_\lambda(\chi_{\vec{\varepsilon}}(\lambda;u_1)
-\chi_{\vec{\varepsilon}}(\lambda;u_2))
((m_{1\vec{\varepsilon}}+n_{1\vec{\varepsilon}})-(m_{2\vec{\varepsilon}}+n_{2\vec{\varepsilon}}))+
\\
2\zeta(\boldsymbol{x})\left(\chi_{\vec{\varepsilon}}(\lambda;u_1)-\chi_{\vec{\varepsilon}}(\lambda;u_2)\right)
\sum\limits_{k=1}^4\left(R_{1k}^{(\vec{\varepsilon})}-R_{2k}^{(\vec{\varepsilon})}\right)\Bigr)
\,d\boldsymbol{x}dtdsd\lambda.
\end{multline}
{\it Step 3. Passage to the limit, as\; $\varepsilon_0\to0+$.}
Further, for the sake of brevity, we denote\linebreak $f_{\hat{\varepsilon}}=
f_{\varepsilon_1,\varepsilon_2,\varepsilon_3}$ for
$f=\chi(\lambda;u_i)$, $f=m_i$, $f=n_i$, etc., and
$R_{ik}^{(\hat{\varepsilon})}=
R_{ik}^{(\varepsilon_1\varepsilon_2\varepsilon_3)}$ $(k=1,2,3,4,\,
i=1,2)$.

\begin{remark}
\label{rem.7.1i}%
We start with the remark that the limiting passage as
$\varepsilon_0\to0+$ in the left-hand side of \eqref{e7.01iii} is simple, and in the limit we arrive at the same
expressions with subscript $\hat{\varepsilon}$ on the
places of $\vec{\varepsilon}.$
  \end{remark}

In the right hand side, due to the standard properties of the
regularizing kernels $\omega_{\varepsilon_0}$,
$\omega_{\varepsilon_1}$, $\omega_{\varepsilon_2}$ and
$\omega_{\varepsilon_3}$, we have
\begin{multline}
\label{e7.01iv}%
R_{ik}^{(\vec{\varepsilon})}
\underset{\varepsilon_0\to0+}{\longrightarrow}
R_{ik}^{(\hat{\varepsilon})}\equiv0\quad
(i,k=1,2),\quad R_{i3}^{(\vec{\varepsilon})}
\underset{\varepsilon_0\to0+}{\longrightarrow}
R_{i3}^{(\hat{\varepsilon})},\quad
\chi_{\vec{\varepsilon}}(\lambda;u_i)
\underset{\varepsilon_0\to0+}{\longrightarrow}
 \chi_{\hat{\varepsilon}}(\lambda;u_i)\quad (i=1,2)
 \\
\text{ strongly in }L_{\rm loc}^r(B\times[\varepsilon_2,T]\times[\varepsilon_1,S]\times\mathbb{R}_\lambda),
\quad \forall\, r\in[1,+\infty).
\end{multline}

Hence
\begin{multline}
\label{e7.01v}%
\int\limits_{B\times(\varepsilon_2,t')\times(\varepsilon_1,s')\times\mathbb{R}_\lambda}
2\zeta(\boldsymbol{x})(\chi_{\vec{\varepsilon}}(\lambda;u_1)-\chi_{\vec{\varepsilon}}(\lambda;u_2))
\sum\limits_{k=1}^3
\left(R_{1k}^{(\vec{\varepsilon})}-R_{2k}^{(\vec{\varepsilon})}\right)
d\boldsymbol{x}dtdsd\lambda
\underset{\varepsilon_0\to0+}{\longrightarrow}
\\
\int\limits_{B\times(\varepsilon_2,t')\times(\varepsilon_1,s')\times\mathbb{R}_\lambda}
2\zeta(\boldsymbol{x})(\chi_{\hat{\varepsilon}}(\lambda;u_1)-\chi_{\hat{\varepsilon}}(\lambda;u_2))
\left(R_{13}^{(\hat{\varepsilon})}-R_{23}^{(\hat{\varepsilon})}\right)
d\boldsymbol{x}dtdsd\lambda.
\end{multline}

According to assertion (i) in Lemma \ref{lem.3.01} we have that
\begin{equation}
\label{e7.01vi}%
\partial_\lambda\chi(\lambda;v)=\delta_{(\lambda=0)}-\delta_{(\lambda=v)},\quad
\forall\, v\in\mathbb{R}\quad (\text{in }\mathcal{M}(\mathbb{R}_\lambda)).
\end{equation}

Using \eqref{e7.01vi}, integrating by parts in $\lambda$, and
passing to the limit, we get
\begin{multline}
\label{e7.01vii}%
\int\limits_{B\times(\varepsilon_2,t')\times(\varepsilon_1,s')\times
\mathbb{R}_\lambda}2\zeta(\boldsymbol{x})
(\chi_{\vec{\varepsilon}}(\lambda;u_1)-\chi_{\vec{\varepsilon}}(\lambda;u_2))
\left(R_{14}^{(\vec{\varepsilon})}-R_{24}^{(\vec{\varepsilon})}\right)
d\boldsymbol{x}dtdsd\lambda=
\\
-\int\limits_{B\times(\varepsilon_2,t')\times(\varepsilon_1,s')\times\mathbb{R}_\lambda}2\zeta(\boldsymbol{x})
\partial_\lambda(\chi_{\vec{\varepsilon}}(\lambda;u_1)-\chi_{\vec{\varepsilon}}(\lambda;u_2))
\Bigl(Z_\gamma(\boldsymbol{x},t,s,\lambda)(\chi_{\vec{\varepsilon}}(\lambda;u_1)-\chi_{\vec{\varepsilon}}(\lambda;u_2))-
\\
(Z_\gamma(\chi(\lambda;u_1)-\chi(\lambda;u_2)))_{\vec{\varepsilon}}\Bigr)
d\boldsymbol{x}dtdsd\lambda
\underset{\varepsilon_0\to0+}{\longrightarrow}
\\
-\int\limits_{B\times(\varepsilon_2,t')\times(\varepsilon_1,s')}2\zeta(\boldsymbol{x})
\bigr\langle\left(\delta_{(\lambda=u_2)}-\delta_{(\lambda=u_1)}\right)_{\hat{\varepsilon}},
Z_\gamma(\boldsymbol{x},t,s,\cdot)(\chi_{\hat{\varepsilon}}(\lambda;u_1)-
\chi_{\hat{\varepsilon}}(\lambda;u_2))-
\\
(Z_\gamma(\chi(\lambda;u_1)-\chi(\lambda;u_2)))_{\hat{\varepsilon}}
\bigr\rangle_{\mathcal{M}(\mathbb{R}_\lambda),C_0(\mathbb{R}_\lambda)}
d\boldsymbol{x}dtds.
\end{multline}

Notice that, since $\lambda\mapsto\chi(\lambda;v)$ is the finite
step-function in $\mathbb{R}_\lambda$, the duality bracket\linebreak
$\bigl\langle\cdot,\cdot\bigr\rangle_{\mathcal{M}(\mathbb{R}_\lambda),C_0(\mathbb{R}_\lambda)}$
is well-defined, although the right-hand side in the bracket is not
continuous in $\lambda$.
Next, integrate \eqref{e7.01i} over the interval
$(-\infty,\lambda_0)$ with respect to $\lambda$. Since
$\chi(\lambda;u_i)$ and $m_i+n_i$ vanish for $\lambda<-M_1(T)$, we
have
\begin{multline}
\label{e7.01viii}%
m_{i\vec{\varepsilon}}(\boldsymbol{x},t,s,\lambda_0)+
n_{i\vec{\varepsilon}}(\boldsymbol{x},t,s,\lambda_0)=\\
\int\limits_{-\infty}^{\lambda_0}\Bigl(
\partial_t\chi_{\vec{\varepsilon}}(\lambda;u_i)+a'(\lambda)
\partial_s\chi_{\vec{\varepsilon}}(\lambda;u_i)+
\boldsymbol{\varphi}'(\lambda)\cdot
\nabla_x\chi_{\vec{\varepsilon}}(\lambda;u_i)
-\Delta_x\chi_{\vec{\varepsilon}}(\lambda;u_i)-
\\
\chi_{\vec{\varepsilon}}(\lambda;u_i)\partial_\lambda
Z_\gamma(\boldsymbol{x},t,s,\lambda)-
Z_{\gamma\hat{\varepsilon}}(\boldsymbol{x},t,s,0)\omega_{\varepsilon_0}(\lambda)
-\sum\limits_{k=1}^4R_{ik}^{(\vec{\varepsilon})}\Bigr)d\lambda+\\
Z_\gamma(\boldsymbol{x},t,s,\lambda_0)\chi_{\vec{\varepsilon}}(\lambda_0;u_i)\quad
(i=1,2).
\end{multline}

Passing to the limit in \eqref{e7.01viii}, we derive
\begin{equation}
\label{e7.01ixa}%
m_{i\vec{\varepsilon}}+n_{i\vec{\varepsilon}}
\underset{\varepsilon_0\to0+}{\longrightarrow}
m_{i\hat{\varepsilon}}+n_{i\hat{\varepsilon}}\text{
strongly in }L_{\rm
loc}^r\left(B\times[\varepsilon_2,T]\times[\varepsilon_1,S]\times\mathbb{R}_{\lambda_0}\right),\quad\forall\,r\in[1,\infty),
\end{equation}
where
\begin{multline}
\label{e7.01ixb}%
m_{i\hat{\varepsilon}}(\boldsymbol{x},t,s,\lambda_0)+n_{i\hat{\varepsilon}}(\boldsymbol{x},t,s,\lambda_0)=\\
\int\limits_{-\infty}^{\lambda_0}\Bigl(
\partial_t\chi_{\hat{\varepsilon}}(\lambda;u_i)+a'(\lambda)
\partial_s\chi_{\hat{\varepsilon}}(\lambda;u_i)+
\boldsymbol{\varphi}'(\lambda)\cdot
\nabla_x\chi_{\hat{\varepsilon}}(\lambda;u_i)
-\Delta_x\chi_{\hat{\varepsilon}}(\lambda;u_i)-
\\
(\chi(\lambda;u_i)\partial_\lambda
Z_\gamma)_{\hat{\varepsilon}}(\boldsymbol{x},t,s,\lambda)
\Bigr)d\lambda+(\chi(\lambda;u_i)Z_\gamma)_{\hat{\varepsilon}}(\boldsymbol{x},t,s,\lambda_0)
-Z_{\gamma\hat{\varepsilon}}(\boldsymbol{x},t,s,0)H(\lambda_0),
\quad (i=1,2),
\end{multline}
$H(\lambda_0)=\mathbf{1}_{(\lambda_0\geqslant0)}$ (the
right-continuous Heaviside function). On the strength of the
structure of function $\chi$ (recall \eqref{chi-f}), properties of
the regularizing kernels $\omega_{\varepsilon_1}$,
$\omega_{\varepsilon_2}$ and $\omega_{\varepsilon_3}$, Conditions on
$a\&\boldsymbol{\varphi}\&Z_\gamma$, and representation
\eqref{e7.01ixb}, we conclude that the bound
\begin{equation}
\label{e7.01x}%
0\leqslant m_{i\hat{\varepsilon}}+
n_{i\hat{\varepsilon}}\leqslant
C_4=C_4(\varepsilon_1,\varepsilon_2,\varepsilon_3)\quad\forall
(\boldsymbol{x},t,s,\lambda_0)\in B\times[\varepsilon_2,T]\times
[\varepsilon_1,S]\times\mathbb{R}_{\lambda_0}
  \end{equation}
holds true and that, for all $(\boldsymbol{x},t,s)\in
B\times[\varepsilon_2,T]\times[\varepsilon_1,S]$, the function
\begin{equation}
\label{e7.01xi}%
 \lambda_0\mapsto
m_{i\hat{\varepsilon}}(\boldsymbol{x},t,s,\lambda_0)+n_{i\hat{\varepsilon}}(\boldsymbol{x},t,s,\lambda_0)
\quad(i=1,2)
  \end{equation}
is absolutely continuous on $\mathbb{R}$ except for the point
$\lambda_0=0$, where it suffers a finite jump.

Due to these properties of
$m_{i\hat{\varepsilon}}+n_{i\hat{\varepsilon}}$, the
limiting relation \eqref{e7.01ixa} and representation
\eqref{e7.01vi}, we establish that
\begin{equation}
\label{e7.01xii}%
\begin{split}
& \int\limits_{B\times(\varepsilon_2,t')\times(\varepsilon_1,s')\times\mathbb{R}_\lambda}\Big(
-2\zeta(\boldsymbol{x})\partial_\lambda(\chi_{\vec{\varepsilon}}(\lambda;u_1)
-\chi_{\vec{\varepsilon}}(\lambda;u_2))
\left((m_{1\vec{\varepsilon}}+n_{1\vec{\varepsilon}})-
(m_{2\vec{\varepsilon}}+n_{2\vec{\varepsilon}})\right)\Big)
\,d\boldsymbol{x}dtdsd\lambda\\
& \underset{\varepsilon_0\to 0+}{\longrightarrow}
\\
& -\int\limits_{B\times(\varepsilon_2,t')\times(\varepsilon_1,s')}2\zeta(\boldsymbol{x})\left\langle(\delta_{(\lambda=u_2)}
-\delta_{(\lambda=u_1)})_{\hat{\varepsilon}},
(m_{1\hat{\varepsilon}}+n_{1\hat{\varepsilon}})-
(m_{2\hat{\varepsilon}}+n_{2\hat{\varepsilon}})
\right\rangle_{\mathcal{M}(\mathbb{R}_\lambda),C_0(\mathbb{R}_\lambda)}\,d\boldsymbol{x}dtds.
\end{split}
  \end{equation}

Now we are in a position to pass to the limit in \eqref{e7.01iii},
as $\varepsilon_0\to0+$. Aggregating \eqref{e7.01v},
\eqref{e7.01vii}, \eqref{e7.01xii} and recalling Remark
\ref{rem.7.1i}, from \eqref{e7.01iii} we derive the integral
equality
\begin{multline*}
\int\limits_{B\times(\varepsilon_2,t')\times(\varepsilon_1,s')\times\mathbb{R}_\lambda}
\Big\{\zeta(\boldsymbol{x})\Big(\partial_t\vert\chi_{\hat{\varepsilon}}(\lambda;u_1)-
\chi_{\hat{\varepsilon}}(\lambda;u_2)\vert^2+
a'(\lambda)\partial_s\left\vert\chi_{\hat{\varepsilon}}(\lambda;u_1)-
\chi_{\hat{\varepsilon}}(\lambda;u_2)\right\vert^2\Big)-
\\
\left(\boldsymbol{\varphi}'(\lambda)\cdot\nabla_x(\boldsymbol{x})+\Delta_x\zeta(\boldsymbol{x})
+\zeta(\boldsymbol{x})\partial_\lambda
Z_\gamma(\boldsymbol{x},t,s,\lambda)\right)
\left\vert\chi_{\hat{\varepsilon}}(\lambda;u_1)-
\chi_{\hat{\varepsilon}}(\lambda;u_2)\right\vert^2+
\\
2\zeta(\boldsymbol{x})\left\vert\nabla_x\left(\chi_{\hat{\varepsilon}}(\lambda;u_1)-
\chi_{\hat{\varepsilon}}(\lambda;u_2)\right)\right\vert^2
 \Big\}\,d\boldsymbol{x}dtdsd\lambda=\\
 -\int\limits_{B\times(\varepsilon_2,t')\times(\varepsilon_1,s')}
2\zeta(\boldsymbol{x})\left\langle(\delta_{(\lambda=u_2)}
-\delta_{(\lambda=u_1)})_{\hat{\varepsilon}},
(m_{1\hat{\varepsilon}}+n_{1\hat{\varepsilon}})-
(m_{2\hat{\varepsilon}}+n_{2\hat{\varepsilon}})
\right\rangle_{\mathcal{M}(\mathbb{R}_\lambda),C_0(\mathbb{R}_\lambda)}d\boldsymbol{x}dtds-
\end{multline*}
 \begin{multline}
\int\limits_{B\times(\varepsilon_2,t')\times(\varepsilon_1,s')}
2\zeta(\boldsymbol{x})\bigl\langle(\delta_{(\lambda=u_2)}
-\delta_{(\lambda=u_1)})_{\hat{\varepsilon}},
Z_\gamma(\boldsymbol{x},t,s,\cdot)\left(\chi_{\hat{\varepsilon}}(\lambda;u_1)-
\chi_{\hat{\varepsilon}}(\lambda;u_2)\right)-
\\
(Z_\gamma(\chi(\lambda;u_1)-
\chi(\lambda;u_2)))_{\hat{\varepsilon}}
\bigr\rangle_{\mathcal{M}(\mathbb{R}_\lambda),C_0(\mathbb{R}_\lambda)}d\boldsymbol{x}dtds+
\\
\int\limits_{B\times(\varepsilon_2,t')\times(\varepsilon_1,s')\times {\mathbb R}_\lambda}2\zeta(\boldsymbol{x})\left(\chi_{\hat{\varepsilon}}(\lambda;u_1)-
\chi_{\hat{\varepsilon}}(\lambda;u_2)\right)
\left(R_{13}^{(\hat{\varepsilon})}-R_{23}^{(\hat{\varepsilon})}\right)
d\boldsymbol{x}dtdsd\lambda.
\label{e7.01xiii}%
  \end{multline}
{\it Step 4. Passage to the limit as $\varepsilon_1,\varepsilon_2,\varepsilon_3\to0+$.}
We pass to the limit in \eqref{e7.01xiii}, as
$\hat{\varepsilon}\to0+$, i.e., as
$\varepsilon_1,\varepsilon_2,\varepsilon_3\to0+$ simultaneously,
provided that ratios $\displaystyle \frac{\varepsilon_1}{\varepsilon_3}$ and
$\displaystyle \frac{\varepsilon_2}{\varepsilon_3}$ tend to zero as well, for the
technical reasons. Repeating arguments from \cite[Chapter 4, proof of
Theorem 4.3.1 (third step)]{Per-2002} with natural modifications, we
deduce that
\begin{equation}
\label{e7.01xiv}%
\lim\limits_{\hat{\varepsilon}\to 0+}
\int\limits_{B\times(\varepsilon_2,t')\times(\varepsilon_1,s')}
2\zeta(\boldsymbol{x}) \left\langle
(\delta_{(\lambda=u_i)})_{\hat{\varepsilon}},
m_{i\hat{\varepsilon}}+n_{i\hat{\varepsilon}}
\right\rangle_{\mathcal{M}(\mathbb{R}_\lambda),C_0(\mathbb{R}_\lambda)}
d\boldsymbol{x}dtds=0\quad(i=1,2).
\end{equation}
(This limiting relation is analogous to \cite[formula
(4.3.5)]{Per-2002}.)

Since $\delta_{(\lambda=u_i)}$, $m_i$ and $n_i$ are nonnegative
measures and $\zeta$ is a nonnegative function, we have that
\begin{multline}
\label{e7.01xv}%
-\int\limits_{B\times(\varepsilon_2,t')\times(\varepsilon_1,s')}
2\zeta(\boldsymbol{x})\Big(
\left\langle(\delta_{(\lambda=u_2)})_{\hat{\varepsilon}},
m_{1\hat{\varepsilon}}+n_{1\hat{\varepsilon}}
\right\rangle_{\mathcal{M}(\mathbb{R}_\lambda),C_0(\mathbb{R}_\lambda)}
\\
+\left\langle(\delta_{(\lambda=u_1)})_{\hat{\varepsilon}},
m_{2\hat{\varepsilon}}+n_{2\hat{\varepsilon}}
\right\rangle_{\mathcal{M}(\mathbb{R}_\lambda),C_0(\mathbb{R}_\lambda)}
\Big)d\boldsymbol{x}dtds\leqslant0,
  \end{multline}
\begin{equation}
\label{e7.01xvi}%
\int\limits_{B\times(\varepsilon_2,t')\times(\varepsilon_1,s')\times\mathbb{R}_\lambda}
2\zeta(\boldsymbol{x})
\vert\nabla_x(\chi_{\hat\varepsilon}(\lambda;u_1)-\chi_{\hat\varepsilon}(\lambda;u_2))\vert^2
d\boldsymbol{x}dtdsd\lambda\geqslant0\quad\forall\varepsilon_1,\varepsilon_2,\varepsilon_3\in(0,1].
  \end{equation}

On the strength of the standard properties of the regularizing
kernels $\omega_{\varepsilon_1}$, $\omega_{\varepsilon_2}$ and
$\omega_{\varepsilon_3}$, we deduce the following limiting relations
\begin{multline}
\label{e7.01xvii}%
\int\limits_{B\times(\varepsilon_2,t')\times(\varepsilon_1,s')\times\mathbb{R}_\lambda}
\zeta(\boldsymbol{x})\Big(
\partial_t\vert\chi_{\hat{\varepsilon}}(\lambda;u_1)-
\chi_{\hat{\varepsilon}}(\lambda;u_2))\vert^2
+a'(\lambda)\partial_s\vert\chi_{\hat{\varepsilon}}(\lambda;u_1)-
\chi_{\hat{\varepsilon}}(\lambda;u_2))\vert^2\Big)
d\boldsymbol{x}dtdsd\lambda\equiv
\\
\int\limits_{B\times(\varepsilon_1,s')\times\mathbb{R}_\lambda}
\zeta(\boldsymbol{x})\big(\vert\chi_{\hat{\varepsilon}}(\lambda;u_1)(\boldsymbol{x},t,s)-
\chi_{\hat{\varepsilon}}(\lambda;u_2)(\boldsymbol{x},t,s)\vert^2
\big)\big\vert_{t=\varepsilon_2}^{t=t'}\,d\boldsymbol{x}dsd\lambda+
\\
\int\limits_{B\times(\varepsilon_2,t')\times\mathbb{R}_\lambda}
\zeta(\boldsymbol{x})a'(\lambda)\big(
\vert\chi_{\hat{\varepsilon}}(\lambda;u_1)(\boldsymbol{x},t,s)-
\chi_{\hat{\varepsilon}}(\lambda;u_2)(\boldsymbol{x},t,s)\vert^2
\big)\big\vert_{s=\varepsilon_1}^{s=s'}\,d\boldsymbol{x}dtd\lambda
\underset{\hat{\varepsilon}\to0+}{\longrightarrow}
\\
\int\limits_{\Omega\times(0,s')\times\mathbb{R}_\lambda}
\zeta(\boldsymbol{x})\Big(
\vert\chi(\lambda;u_{1,t'-0}^{\mathrm{tr},(1)}(\boldsymbol{x},s))-
\chi(\lambda;u_{2,t'-0}^{\mathrm{tr},(1)}(\boldsymbol{x},s))\vert^2-\vspace{8cm}\\
\vert\chi(\lambda;u_{1,0}^{\mathrm{tr},(1)}(\boldsymbol{x},s))-
\chi(\lambda;u_{2,0}^{\mathrm{tr},(1)}(\boldsymbol{x},s))\vert^2
\Big)d\boldsymbol{x}dsd\lambda+
\\
\int\limits_{\Omega\times(0,t')\times\mathbb{R}_\lambda}\zeta(\boldsymbol{x})a'(\lambda)\Big(
\vert\chi(\lambda;u_{1,s'-0}^{\mathrm{tr},(2)}(\boldsymbol{x},t))-
\chi(\lambda;u_{2,s'-0}^{\mathrm{tr},(2)}(\boldsymbol{x},t))\vert^2-\\
\vert\chi(\lambda;u_{1,0}^{\mathrm{tr},(2)}(\boldsymbol{x},t))-
\chi(\lambda;u_{2,0}^{\mathrm{tr},(2)}(\boldsymbol{x},t))\vert^2
\Big)d\boldsymbol{x}dtd\lambda,
  \end{multline}
\begin{multline}
\label{e7.01xviii}%
-\int\limits_{B\times(\varepsilon_2,t')\times(\varepsilon_1,s')\times\mathbb{R}_\lambda}
\Big(\boldsymbol{\varphi}'(\lambda)\cdot\nabla_x\zeta(\boldsymbol{x})+
\Delta_x\zeta(\boldsymbol{x})+\zeta(\boldsymbol{x})\partial_\lambda
Z_\gamma(\boldsymbol{x},t,s,\lambda)\Big)\times \\
\vert\chi_{\hat{\varepsilon}}(\lambda;u_1)-
\chi_{\hat{\varepsilon}}(\lambda;u_2)\vert^2
d\boldsymbol{x}dtdsd\lambda
\underset{\hat{\varepsilon}\to0+}{\longrightarrow}
\\
-\int\limits_{\Omega\times(\varepsilon_2,t')\times(\varepsilon_1,s')\times\mathbb{R}_\lambda}
\Big(\boldsymbol{\varphi}'(\lambda)\cdot\nabla_x\zeta(\boldsymbol{x})+\Delta_x\zeta(\boldsymbol{x})+\zeta(\boldsymbol{x})\partial_\lambda
Z_\gamma(\boldsymbol{x},t,s,\lambda)\Big)\times\\ \vert\chi(\lambda;u_1)-
\chi(\lambda;u_2)\vert^2d\boldsymbol{x}dtdsd\lambda,
  \end{multline}
\begin{multline}
\label{e7.01xix}%
-\int\limits_{B\times(\varepsilon_2,t')\times(\varepsilon_1,s')}
2\zeta(\boldsymbol{x}) \Big\langle
(\delta_{(\lambda=u_2)}-\delta_{(\lambda=u_1)})_{\hat{\varepsilon}},
Z_\gamma(\boldsymbol{x},t,s,\cdot)(\chi_{\hat{\varepsilon}}(\lambda;u_1)-\chi_{\hat{\varepsilon}}(\lambda;u_2))
\\
-(Z_\gamma(\chi(\lambda;u_1)-\chi(\lambda;u_2)))_{\hat{\varepsilon}}
\Big\rangle_{\mathcal{M}(\mathbb{R}_\lambda),C_0(\mathbb{R}_\lambda)}
d\boldsymbol{x}dtds
\\
+\int\limits_{B\times(\varepsilon_2,t')\times(\varepsilon_1,s')\times\mathbb{R}_\lambda}
2\zeta(\boldsymbol{x})(\chi_{\hat{\varepsilon}}(\lambda;u_1)-\chi_{\hat{\varepsilon}}(\lambda;u_2))
(R_{13}^{(\hat{\varepsilon})}-R_{23}^{(\hat{\varepsilon})})
d\boldsymbol{x}dtdsd\lambda\underset{\varepsilon\to0+}{\longrightarrow}0.
  \end{multline}

Combining \eqref{e7.01xiii}, \eqref{e7.01xv} and \eqref{e7.01xvi},
and then passing to the limit in the resulting inequality, on the
strength of \eqref{e7.01xiv}, \eqref{e7.01xvii}, \eqref{e7.01xviii}
and \eqref{e7.01xix}, we derive the integral inequality
\eqref{e7.01} with $s'$ on the place of $S$.\\[1ex]
{\it The last
(fifth) step of the proof consists of the passage to the limit as
$s\to S-0$.}
This limiting transition is quite straightforward thanks to
existence of the strong traces $u_{1,S}^{{\rm tr},(2)}$ and
$u_{2,S}^{{\rm tr},(2)}$ due to Lemma \ref{lem.5.3}.

Lemma \ref{prop.7.1} is proved.
\qed

\begin{remark}
\label{rem.7.1}%
 In view of relation \eqref{e3.01b} and Lemma
\ref{lem.5.2}, it is legitimate to insert
$u_{1,0}^{(1)}(\boldsymbol{x},s)$ and $u_1(\boldsymbol{x},t',s)$ on
the places of $u_{1,0}^{{\rm tr},(1)}(\boldsymbol{x},s)$ and
$u_{1,t'-0}^{{\rm tr},(1)}(\boldsymbol{x},s)$, and to insert
$u_{2,0}^{(1)}(\boldsymbol{x},s)$ and $u_2(\boldsymbol{x},t',s)$ on
the places of $u_{2,0}^{{\rm tr},(1)}(\boldsymbol{x},s)$ and
$u_{2,t'-0}^{{\rm tr},(1)}(\boldsymbol{x},s)$ in \eqref{e7.01} and
in the further considerations.
    \end{remark}

Next, we find the valuable relations between prescribed data
\eqref{e1.01c} and traces of kinetic solutions of Problem
$\Pi_\gamma$ on $\Gamma_0^2$ and $\Gamma_S^2$. Verification of these
relations heavily relies on the kinetic boundary conditions
\eqref{e3.01c} and \eqref{e3.01d}.

\begin{lemma}
\label{prop.7.2}%
For an arbitrarily fixed $t'\in(0,\tau]$ and for all nonnegative
test-functions $\xi\in C(\mathbb{R}^d)$ the inequalities
\begin{multline}
\label{e7.02}%
\int\limits_0^{t'}\int\limits_\Omega \int\limits_{-M_1(t')}^{M_1(t')}
a'(\lambda)\vert\chi(\lambda;u_{1,0}^{{\rm
tr},(2)}(\boldsymbol{x},t))-\chi(\lambda;u_{2,0}^{{\rm
tr},(2)}(\boldsymbol{x},t))\vert^2\xi(\boldsymbol{x})\,d\lambda
d\boldsymbol{x}dt\leqslant
\\
\int\limits_0^{t'}\int\limits_\Omega \int\limits_{-M_1(t')}^{M_1(t')}
a'(\lambda)\vert\chi(\lambda;u_{1,0}^{(2)}(\boldsymbol{x},t))-
\chi(\lambda;u_{2,0}^{(2)}(\boldsymbol{x},t))\vert^2\xi(\boldsymbol{x})\,
d\lambda d\boldsymbol{x}dt
\end{multline}
and
\begin{multline}
\label{e7.03}%
-\int\limits_0^{t'}\int\limits_\Omega \int\limits_{-M_1(t')}^{M_1(t')}
a'(\lambda)\vert\chi(\lambda;u_{1,S}^{{\rm
tr},(2)}(\boldsymbol{x},t))-\chi(\lambda;u_{2,S}^{{\rm
tr},(2)}(\boldsymbol{x},t))\vert^2\xi(\boldsymbol{x})\,d\lambda
d\boldsymbol{x}dt\leqslant
\\
-\int\limits_0^{t'}\int\limits_\Omega \int\limits_{-M_1(t')}^{M_1(t')}
a'(\lambda)\vert\chi(\lambda;u_{1,S}^{(2)}(\boldsymbol{x},t))-
\chi(\lambda;u_{2,S}^{(2)}(\boldsymbol{x},t))\vert^2\xi(\boldsymbol{x})\,d\lambda
d\boldsymbol{x}dt
\end{multline}
are valid. (Values $M_1(t')$ are given by \eqref{e7.01bis}.)
  \end{lemma}
\proof Assertion of Lemma \ref{prop.7.2} is just a minor
modification of \cite[Proposition 2]{Kuz-2017}. Therefore
justification is almost the same, as in \cite{Kuz-2017}, and we omit
it. \qed

Now, let us establish an important auxiliary inequality for the
difference $u_1-u_2$. The following lemma is somewhat similar to
 \cite[Corollary 1]{Kuz-2017}.

\begin{lemma}
\label{lem.7.1}%
For an arbitrarily fixed $t'\in(0,\tau]$ and for all nonnegative
test-functions\linebreak $\xi\in C^2(\mathbb{R}^d)$ the inequality
\begin{multline}
\label{e7.04}%
\int\limits_{\Xi^1}\left\vert
u_1(\boldsymbol{x},t',s)-u_2(\boldsymbol{x},t',s)\right\vert\xi(\boldsymbol{x})\,d\boldsymbol{x}ds
\\
-\int\limits_0^{t'}\int\limits_{\Xi^1}
\sum\limits_{i=1}^d\partial_{x_i}\xi(\boldsymbol{x})
\big(\varphi_i(u_1(\boldsymbol{x},t,s))
-\varphi_i(u_2(\boldsymbol{x},t,s))\big)\,{\rm
sgn}(u_1(\boldsymbol{x},t,s)-u_2(\boldsymbol{x},t,s))
\,d\boldsymbol{x}dsdt
\\
-\int\limits_0^{t'}\int\limits_{\Xi^1}\Delta_x\xi(\boldsymbol{x})
\big\vert
u_1(\boldsymbol{x},t,s)-u_2(\boldsymbol{x},t,s)\big\vert\,d\boldsymbol{x}dsdt
\\
-\int\limits_0^{t'}\int\limits_{\Xi^1}
(Z_\gamma(\boldsymbol{x},t,s,u_1(\boldsymbol{x},t,s))-Z_\gamma(\boldsymbol{x},t,s,u_2(\boldsymbol{x},t,s)))
{\rm
sgn}(u_1(\boldsymbol{x},t,s)-u_2(\boldsymbol{x},t,s))\xi(\boldsymbol{x})\,d\boldsymbol{x}dsdt
\leqslant
\\
\int\limits_{\Xi^1} \big\vert
u_{1,0}^{(1)}(\boldsymbol{x},s)-u_{2,0}^{(1)}(\boldsymbol{x},s)\big\vert\xi(\boldsymbol{x})
\,d\boldsymbol{x}ds
\\
+\max\limits_{\lambda\in\left[-M_1(t'),M_1(t')\right]}\vert
a'(\lambda)\vert\int\limits_0^{t'}\int\limits_\Omega\Big(\big\vert
u_{1,0}^{(2)}(\boldsymbol{x},t)-u_{2,0}^{(2)}(\boldsymbol{x},t)\big\vert+\big\vert
u_{1,S}^{(2)}(\boldsymbol{x},t)-u_{2,S}^{(2)}(\boldsymbol{x},t)\big\vert\Big)\xi(\boldsymbol{x})
\,d\boldsymbol{x}dt
  \end{multline}
holds true.
\end{lemma}

\proof Firstly, we apply assertion (iii) and, after this,
assertion  (ii) of Lemma \ref{lem.3.01} to all terms in equalities
\eqref{e7.01}, \eqref{e7.02} and \eqref{e7.03}. Secondly, we combine
these inequalities properly and take into account Remark
\ref{rem.7.1}.
Thus we arrive at inequality \eqref{e7.04}, which completes the
proof of the lemma.
\qed

Now we are in a position to establish estimate \eqref{e3.05}. Take
$\xi\equiv 1$ in \eqref{e7.04}, which is a legal choice of
test-function. Then estimate the last integral in the left-hand side
using the Lagrange mean value theorem:
\begin{multline}
\label{e7.04bis}%
\left\vert \int\limits_0^{t'} \int\limits_{\Xi^1}
\left(Z_\gamma(\boldsymbol{x},t,s,u_1(\boldsymbol{x},t,s))
-Z_\gamma(\boldsymbol{x},t,s,u_2(\boldsymbol{x},t,s))\right)\mathrm{sgn}(u_1(\boldsymbol{x},t,s)-u_2(\boldsymbol{x},t,s))\,d\boldsymbol{x}dsdt\right\vert
\leqslant
\\
\int\limits_0^{t'}
\max\limits_{(\boldsymbol{x},s,\lambda)\in\Xi^1\times[-M_1(t'),M_1(t')]}
|\partial_\lambda Z_\gamma(\boldsymbol{x},t,s,\lambda)|
\int\limits_{\Xi^1}
|u_1(\boldsymbol{x},t,s)-u_2(\boldsymbol{x},t,s)|\,d\boldsymbol{x}dsdt
\end{multline}

Combining this and \eqref{e7.04} (with $\xi\equiv 1$), we obtain the
estimate
\begin{multline}
\label{e7.05}%
\int\limits_{\Xi^1}\vert u_1(\boldsymbol{x},t',s)-
u_2(\boldsymbol{x},t',s)\vert\,d\boldsymbol{x}ds\leqslant
\int\limits_{\Xi^1}\vert
u_{1,0}^{(1)}(\boldsymbol{x},s)-u_{2,0}^{(1)}(\boldsymbol{x},s)\vert\,d\boldsymbol{x}ds+
\\
\int\limits_0^{t'}\Big\{
\max\limits_{(\boldsymbol{x},s,\lambda)\in\Xi^1\times[-M_1(t'),M_1(t')]}
|\partial_\lambda Z_\gamma(\boldsymbol{x},t,s,\lambda)|
\int\limits_{\Xi^1}|u_{1,0}^{(1)}-u_{2,0}^{(1)}|\,d\boldsymbol{x}ds+
\\
\max\limits_{\lambda\in[-M_1(t'),M_1(t')]}|a'(\lambda)| \int\limits_\Omega
\left(\vert
u_{1,0}^{(2)}(\boldsymbol{x},t)-u_{2,0}^{(2)}(\boldsymbol{x},t)\vert+
\vert
u_{1,S}^{(2)}(\boldsymbol{x},t)-u_{2,S}^{(2)}(\boldsymbol{x},t)\vert\right)\,d\boldsymbol{x}\Big\}dt,\\
\forall\, t'\in(0,T].
\end{multline}

Applying Gr\"onwall's lemma to \eqref{e7.05} and exchanging $t$ and
$t'$, we arrive exactly at the estimate \eqref{e3.05}. Clearly, the
uniqueness and stability of kinetic solutions to Problem
$\Pi_\gamma$ follow from this estimate.

Thus, the proof of assertion 1 of Theorem \ref{theo.3.1} is
complete.

\begin{remark}
\label{rem.7.2}%
Since the kinetic solution $u=u(\boldsymbol{x},t,s)$ of Problem
$\Pi_\gamma$ is unique, we conclude that {\bf the whole family}
$\{u_\varepsilon\}_{\varepsilon \in (0,1]}$ converges to $u$ as
$\varepsilon\to 0+$ strongly in $L^2(G_{T,S})$. Thus we have
$u=\mbox{\rm s-}\lim\limits_{\varepsilon\to 0} u_\varepsilon$, and the
limiting relation $u=\mbox{\rm s-}\lim\limits_{{\varepsilon''}\to 0}
u_{\varepsilon''}$ (see \eqref{e4.10}) is a particular case.
\end{remark}

\section{Proof of assertion 2 and 3 of Theorem \ref{theo.3.1}}
\label{sec.8}
The proof of equivalency of $L^{\infty}$-solutions of the kinetic
equation \eqref{e3.01a} to the entropy inequality \eqref{e3.04a} is
quite similar to justification of \cite[Remark 2.5]{Ch-P-2003}. It
is directly based on Lemma \ref{lem.3.01} and nonnegativity of the
kinetic defect measure $m$. The initial condition \eqref{e3.01b} is
equivalent to the initial condition \eqref{e3.04c} thanks to assertion
(ii) of Lemma \ref{lem.3.01}. The kinetic boundary conditions
\eqref{e3.01c} and \eqref{e3.01d} are equivalent to the entropy boundary
conditions \eqref{e3.04d} and \eqref{e3.04e}, respectively, due
to assertion (i) of Lemma \ref{lem.3.01} and nonnegativity of the
boundary kinetic defect measures $\mu_0^{(2)}$ and $\mu_S^{(2)}$. In
fact, we have proved equivalency of \eqref{e3.01c}--\eqref{e3.01d}
to \eqref{e3.04d}--\eqref{e3.04e} during justification of Lemma
\ref{lem.6.3} (see formulas \eqref{e6.10}--\eqref{e6.15}).

Thus, the proof of assertion 3 of Theorem \ref{theo.3.1} is
complete. Assertion 2 of Theorem \ref{theo.3.1} directly  follows
from assertions 1 and 3. Thus, Theorem \ref{theo.3.1} is proved. \qed

\section{Passage to the limit, as $\boldsymbol{\gamma\to0+}$:\\ formulation of the main results}
\label{SingLim}
Now suppose that the source term $Z_\gamma$ in \eqref{e1.01a} has
the form \eqref{u1.06} and Conditions on $K_\gamma\&\beta$ hold.
That is, \eqref{e1.01a} has the form
\begin{equation}
\label{e9.00}%
\partial_t
u+\partial_sa(u)+\mathrm{div}_x\boldsymbol{\varphi}(u)=\Delta_xu+K_\gamma(t,\tau)\beta(\boldsymbol{x},s,u).
\end{equation}

On the strength of \eqref{u1.08}, for $Z_\gamma=Z_\gamma(\boldsymbol{x},t,s,\lambda)$ we have
\begin{equation}
\label{e9.01}%
Z_\gamma\underset{\gamma\to0+}{\longrightarrow}\beta\delta_{(t=\tau-0)}\text{
weakly}^*\text{ in }
C^1(\Xi^1\times\mathbb{R}_\lambda)\times\mathcal{M}(0,T).
\end{equation}
Inserting directly $\beta(\boldsymbol{x},s,u)\delta_{(t=\tau-0)}$ on
the place of $K_\gamma(t,\tau)\beta(\boldsymbol{x},s,u)$ in equation
\eqref{e9.00}, we straightly set up the following formulation.\\[1ex]
\noindent \textbf{Problem} $\boldsymbol{\Pi_0.}$ {\bf (The Cauchy --- Dirichlet problem
for the impulsive Kolmogorov-type equation.)} \emph{For arbitrarily given
initial and final data satisfying Conditions on
$u_0^{(1)}\&u_0^{(2)}\&u_S^{(2)}$ and for arbitrarily given
impulsive perturbation $\beta$ satisfying the demands of item (ii) in
Conditions on $K_\gamma\&\beta$, it is necessary to find a function
$u$: $G_{T,S}\mapsto\mathbb{R}$ satisfying
the quasi-linear ultra-parabolic equation
\begin{subequations}
\begin{equation}
\label{e9.02a}%
\partial_tu+\partial_sa(u)+{\rm div}_x\boldsymbol{\varphi}(u)=
\Delta_xu,\quad(\boldsymbol{x},t,s)\in G_{T,S}\setminus \{t=\tau\},
\end{equation}
the impulsive condition
\begin{equation}
\label{e9.02b}%
u(\boldsymbol{x},\tau+0,s)=u(\boldsymbol{x},\tau-0,s)+\beta(\boldsymbol{x},s,u(\boldsymbol{x},\tau-0,s)),\quad
(\boldsymbol{x},s)\in\Xi^1,
\end{equation}
\end{subequations}
the initial and final conditions
\eqref{e1.01b} and \eqref{e1.01c},
and the homogeneous boundary condition \eqref{e1.01d}.}\\[1ex]
\indent In this formulation, the nonlinearities $a=a(\lambda)$ and
$\boldsymbol{\varphi}=\boldsymbol{\varphi}(\lambda)=(\varphi_1(\lambda),\ldots,\varphi_d(\lambda))$
satisfy the demands of items (i) and (ii) in Conditions on
$a\&\boldsymbol{\varphi}\& Z_\gamma$.

\begin{remark}
In the formulation of Problem $\Pi_0$, we have noticed that the
system of equations \eqref{e9.02a} and impulsive condition
\eqref{e9.02b} is equivalent to the equation
\begin{equation}
\label{e9.03}%
\partial_tu+\partial_sa(u)+{\rm
div}_x\boldsymbol{\varphi}(u)-\Delta_xu=\beta(\boldsymbol{x},s,u)\delta_{(t=\tau-0)},\quad(\boldsymbol{x},t,s)\in
G_{T,S},
\end{equation}
in the sense of distributions.
\end{remark}

We introduce the notions of kinetic and entropy solutions to Problem
$\Pi_0$  similarly to Definitions \ref{def.3.02} and \ref{def.3.03},
with natural changes.

\begin{definition}
\label{def.9.01}%
\begin{subequations}
Function $u\in L^\infty(G_{T,S})\cap
L^2((0,T)\times(0,S);\text{\it\r{W}}{}_2^{\,1}(\Omega))$ is called a
kinetic solution of Problem $\Pi_0$, if it satisfies the kinetic
equation
\begin{multline}
\label{e9.04a}%
\partial_t\chi(\lambda;u(\boldsymbol{x},t,s))+a'(\lambda)\partial_s\chi(\lambda;u(\boldsymbol{x},t,s))+
\boldsymbol{\varphi}'(\lambda)\cdot\nabla_x\chi(\lambda;u(\boldsymbol{x},t,s))=\\ \Delta_x
\chi(\lambda;u(\boldsymbol{x},t,s))
+\partial_\lambda(m(\boldsymbol{x},t,s,\lambda)+n(\boldsymbol{x},t,s,\lambda)),\\
(\boldsymbol{x},t,s,\lambda)\in (G_{T,S}\times \RR_\lambda)\setminus \{t=\tau\},
\end{multline}
with some measure
$m\in\mathcal{M}^+(G_{T,S}\times\mathbb{R}_\lambda)$ such that
$\mathrm{supp\,}m\subset G_{T,S}\times[-M_3,M_3]$, and with
$n=\delta_{(\lambda=u)}|\nabla_x u|^2$; the kinetic initial
condition
\begin{equation}
\label{e9.04b}%
\underset{t\to0+}{\mathrm{esslim}}\int\limits_{-M_2}^{M_2}\int\limits_{\Xi^1}
\left\vert\chi(\lambda;u(\boldsymbol{x},s,t))-
\chi(\lambda;u_0^{(1)}(\boldsymbol{x},s))\right\vert\,d\boldsymbol{x}dsd\lambda=0,
\end{equation}
the kinetic boundary conditions
\begin{multline}
\label{e9.04c}%
a^{\prime}(\lambda)\big(\chi(\lambda;u_0^{{\rm
tr},(2)}(\boldsymbol{x},t))-
\chi(\lambda;u_0^{(2)}(\boldsymbol{x},t))\big)-
\\
\delta_{(\lambda=u_0^{(2)}(\boldsymbol{x},t))} \big(a(u_0^{{\rm
tr},(2)}(\boldsymbol{x},t)) -a(u_0^{(2)}(\boldsymbol{x},t))\big)=
\partial_\lambda\mu_0^{(2)}(\boldsymbol{x},t,\lambda),\\
(\mathbf{x},t,\lambda) \in \Xi^2 \times [-M_3,M_3],
  \end{multline}
\begin{multline}
\label{e9.04d}%
a^{\prime}(\lambda)\big(\chi(\lambda;u_{S}^{{\rm
tr},(2)}(\boldsymbol{x},t))
-\chi(\lambda;u_{S}^{(2)}(\boldsymbol{x},t))\big)-
\\
\delta_{(\lambda=u_{S}^{(2)}(\boldsymbol{x},t))} \big(a(u_{S}^{{\rm
tr},(2)}(\boldsymbol{x},t)) -a(u_{S}^{(2)}(\boldsymbol{x},t))\big)=
-\partial_\lambda\mu_{S}^{(2)}(\boldsymbol{x},t,\lambda),\\
(\mathbf{x},t,\lambda) \in \Xi^2 \times [-M_3,M_3],
  \end{multline}
with some measures $\mu_0^{(2)}$,
$\mu_S^{(2)}\in\mathcal{M}^+(\Xi^2\times\mathbb{R}_\lambda)$  such
that $\mathrm{supp\,}\mu_0^{(2)}$,
$\mathrm{supp\,}\mu_S^{(2)}\subset\Xi^2\times[-M_3,M_3]$; and the
kinetic impulsive condition
\begin{equation}
\label{e9.04e}%
\int\limits_{-M_3}^{M_3}\chi(\lambda;u(\boldsymbol{x},\tau+0,s))\,d\lambda=
\int\limits_{-M_3}^{M_3}(1+\partial_\lambda\beta(\boldsymbol{x},s,\lambda))\chi(\lambda;u(\boldsymbol{x},\tau-0,s))\,d\lambda
+\beta(\boldsymbol{x},s,0),\quad(\boldsymbol{x},s)\in\Xi^1.
\end{equation}
\end{subequations}
  \end{definition}
Constants $M_2$ and $M_3$ arise from the maximum principle. They are
defined by the formulas
\begin{equation}
\label{e9.05}%
M_2:=\max\left\{\|u_0^{(1)}\|_{L^\infty(\Xi^1)},
\|u_0^{(2)}\|_{L^\infty(\Omega\times(0,\tau))},
\|u_S^{(2)}\|_{L^\infty(\Omega\times(0,\tau))}\right\}
\end{equation}
and
\begin{equation}
\label{e9.06}%
M_3:=\max\left\{M_2+\|\beta\|_{C(\Xi^1\times[-M_2,M_2])},
\|u_0^{(2)}\|_{L^\infty(\Omega\times(\tau,T))},
\|u_S^{(2)}\|_{L^\infty(\Omega\times(\tau,T))}\right\}.
\end{equation}
Functions $u_0^{\mathrm{tr},(2)}$ and $u_S^{\mathrm{tr},(2)}$ are
strong traces of a solution $u=u(\boldsymbol{x},t,s)$ (if any) of
the kinetic equation \eqref{e9.04a} on the planes $\{s=0\}$ and
$\{s=S\}$, respectively. They are understood in the sense of
limiting relations \eqref{e3.02} and \eqref{e3.03}.

The kinetic equation \eqref{e9.04a} and the kinetic boundary
conditions \eqref{e9.04c} and \eqref{e9.04d} are understood in the
sense of distributions.

\begin{remark}
\label{rem.9.01-bis}%
The system consisting of the kinetic equation \eqref{e9.04a} and the kinetic impulsive
condition \eqref{e9.04e} can be equivalently written in the form of
the kinetic equation with the singular term, incorporating
delta-measure $\delta_{(t=\tau-0)}$, as follows:
\begin{equation}
\label{e9.04a-bis}%
\partial_t\chi(\lambda;u)+a'(\lambda)\partial_s\chi(\lambda;u)+\boldsymbol{\varphi}'(\lambda)\cdot
\nabla_x\chi(\lambda;u)-\Delta_x\chi(\lambda;u)=\partial_\lambda
\left(m+n+\delta_{t=\tau-0}\mathbf{1}_{\lambda\geqslant
u}\beta(x,s,u)\right).
\end{equation}
Justification of this claim is given further in detail in Section \ref{ImpulsiveKC}.
\end{remark}

\begin{definition}
\label{def.9.02}
\begin{subequations}
Function $u\in L^\infty(G_{T,S})\cap
L^2((0,T)\times(0,S);\text{\it\r{W}}{}_2^{\,1}(\Omega))$ is called
an entropy solution of Problem $\Pi_0$, if it satisfies the entropy
inequality
\begin{equation}
\label{e9.07a}%
\partial_{t}\eta(u)+\partial_{s}q_a(u)+{\rm
div}_x\boldsymbol{q}_\varphi(u)-\Delta_x\eta(u)\leqslant
-\eta^{\prime\prime}(u)|\nabla_x u|^2, \quad (\boldsymbol{x},t,s)\in G_{T,S}\setminus \{t=\tau\},
\end{equation}
the maximum principle
\begin{equation}
\label{e9.07b}%
\left\Vert u\right\Vert_{L^\infty(\Xi^1\times(0,\tau))}\leqslant
M_2,\quad \left\Vert
u\right\Vert_{L^\infty(\Xi^1\times(\tau,T))}\leqslant M_3,
   \end{equation}
the initial condition
\begin{equation}
\label{e9.07c}%
\underset{t\to0+}{\mathrm{ess\,lim}}\int\limits_{\Xi^1} \left\vert
u(\boldsymbol{x},t,s)-u_0^{(1)}(\boldsymbol{x},s)\right\vert\,d\boldsymbol{x}ds=0,
\end{equation}
the entropy boundary conditions
\begin{equation}
\label{e9.07d}%
q_{a}(u_0^{{\rm tr},(2)}(\boldsymbol{x},t))
-q_{a}(u_0^{(2)}(\boldsymbol{x},t))
-\eta^{\prime}(u_0^{(2)}(\boldsymbol{x},t)) (a(u_0^{{\rm
tr},(2)}(\boldsymbol{x},t))
-a(u_0^{(2)}(\boldsymbol{x},t)))\leqslant0,
\quad(\boldsymbol{x},t)\in \Xi^2,
  \end{equation}
\begin{equation}
\label{e9.07e}%
q_a(u_{S}^{{\rm tr},(2)}(\boldsymbol{x},t))-
q_a(u_{S}^{(2)}(\boldsymbol{x},t))
-\eta^{\prime}(u_{S}^{(2)}(\boldsymbol{x},t)) (a(u_{S}^{{\rm
tr},(2)}(\boldsymbol{x},t))
-a(u_{S}^{(2)}(\boldsymbol{x},t)))\geqslant0,
\quad(\boldsymbol{x},t)\in \Xi^2.
  \end{equation}
and the impulsive condition \eqref{e9.02b}. In \eqref{e9.07a},
\eqref{e9.07d} and \eqref{e9.07e}, $\eta\in C^2(\mathbb{R})$ is an
arbitrary convex  test-function, i.e., $\eta''(\lambda)\geqslant 0$ $\forall\,
\lambda\in\mathbb{R}$, and $(\eta,q_a,\boldsymbol{q}_\varphi)$ is a convex
entropy flux triple:
\begin{equation*}
q_a^{\prime}(\lambda)=a^{\prime}(\lambda)\eta^{\prime}(\lambda),\quad
\boldsymbol{q}_\varphi^{\prime}(\lambda)=\boldsymbol{\varphi}^{\prime}(\lambda)\eta^{\prime}(\lambda),\quad
\lambda\in\mathbb{R}.
  \end{equation*}
Entropy inequality \eqref{e9.07a} is understood in the sense of
distributions. Entropy boundary conditions \eqref{e9.07d} and \eqref{e9.07e} and impulsive condition \eqref{e9.02b} are
understood almost everywhere in $\Xi^2$ and $\Xi^1$, respectively.
\end{subequations}
\end{definition}

We are going to fulfill and rigorously justify the limiting passage
from Problem $\Pi_\gamma$ (with $Z_\gamma=K_\gamma\beta$) to Problem
$\Pi_0$ as $\gamma\to0+$, and to establish the well-posedness of Problem
$\Pi_0$. More precisely, in Sections \ref{Equiv}--\ref{Limited} further
we prove the following theorem, which is the second main result of
the article.

\begin{theorem}
\label{theo.9.1}
{\; \bf 1.\;(Convergence result.)} Let the source term $Z_\gamma$ in
Problem $\Pi_\gamma$ have the form \eqref{u1.06}, where functions
$K_\gamma$ and $\beta$ satisfy Conditions on $K_\gamma\&\beta$. Let
the nonlinearities $a=a(\lambda)$ and
$\boldsymbol{\varphi}=\boldsymbol{\varphi}(\lambda)$ satisfy the demands
of items (i) and (ii) in Conditions on
$a\&\boldsymbol{\varphi}\&Z_\gamma$ and the additional demands
\begin{equation}
\label{e9.08i}
\max\limits_{\lambda\in\mathbb{R}}|a'(\lambda)|\leqslant C_5,\quad
\max\limits_{\lambda\in\mathbb{R}}|\boldsymbol{\varphi}'(\lambda)|\leqslant
C_5,\quad C_5=\mathrm{const}<+\infty.
\end{equation}
Then there exists the unique limiting function $u_*\in
L^\infty(G_{T,S})\cap
L^2((0,T)\times(0,S);\text{\it\r{W}}{}_2^{\,1}(\Omega))$ of the
family $\{u_\gamma\}_{\gamma>0}$ of kinetic and entropy solutions of
Problem $\Pi_\gamma$ as $\gamma\to0+$, such that
\begin{eqnarray}
\label{e9.08} u_\gamma\underset{\gamma\to0+}{\longrightarrow}
u_* & & \text{strongly in }L^1(G_{T,S}),\\
\label{e9.09} & & \text{weakly in
}L^2((0,T)\times(0,S);\text{\it\r{W}}{}_2^{\,1}(\Omega)).
\end{eqnarray}
Furthermore, $u_*$ is a kinetic and entropy solution of Problem
$\Pi_0$ in the sense of Definitions \ref{def.9.01} and
\ref{def.9.02}.\\[1ex]
{\bf 2.\;(Existence, uniqueness and stability of kinetic solutions to Problem
$\boldsymbol{\Pi_0}$.)} Let the nonlinearities $a=a(\lambda)$ and
$\boldsymbol{\varphi}=\boldsymbol{\varphi}(\lambda)$ satisfy demands
of items (i) and (ii) in Conditions on
$a\&\boldsymbol{\varphi}\&Z_\gamma$. Whenever the impulsive perturbation
$\beta=\beta(\boldsymbol{x},s,\lambda)$ belongs to
$C_{\mathrm{loc}}^1(\Xi^1\times\mathbb{R}_\lambda)$ and initial and
final data meet Conditions on $u_0^{(1)}\&u_0^{(2)}\&u_S^{(2)}$,
Problem $\Pi_0$ has the unique kinetic solution
$u_*=u_*(\boldsymbol{x},t,s)$ in the sense of Definition
\ref{def.9.01}.

Moreover, let $u_{*1}$ and $u_{*2}$ be two kinetic solutions of
Problem $\Pi_0$ corresponding to two given sets of data
$(\beta_1,u_{1,0}^{(1)},u_{1,0}^{(2)},u_{1,S}^{(2)})$ and
$(\beta_2,u_{2,0}^{(1)},u_{2,0}^{(2)},u_{2,S}^{(2)})$, respectively,
then the estimate
\begin{multline}
\label{e9.10}%
\|u_{*1}(\cdot,t,\cdot)-u_{*2}(\cdot,t,\cdot)\|_{L^1(\Xi^1)}\leqslant
\\
\|u_{1,0}^{(1)}-u_{2,0}^{(1)}\|_{L^1(\Xi^1)}+\|a'\|_{C[-M_4(T),M_4(T)]}
\left(\|u_{1,0}^{(2)}-u_{2,0}^{(2)}\|_{L^1(\Omega\times(0,t))}+
\|u_{1,S}^{(2)}-u_{2,S}^{(2)}\|_{L^1(\Omega\times(0,t))}\right)+
\\
\mathbf{1}_{(t\geqslant\tau)}\Big[(S\,\mathrm{meas\,}\Omega)
\|\beta_1-\beta_2\|_{C(\Xi^1\times[-M_5(\tau),M_5(\tau)])}
+\\ \max\limits_{\Xi^1\times[-M_5(\tau),M_5(\tau)]}|\partial_\lambda\beta_1(\boldsymbol{x},s,\lambda)|
\big(\|u_{1,0}^{(1)}-u_{2,0}^{(1)}\|_{L^1(\Xi^1)}
\\
+\|a'\|_{C[-M_5(\tau),M_5(\tau)]}
\left(\|u_{1,0}^{(2)}-u_{2,0}^{(2)}\|_{L^1(\Omega\times(0,\tau))}+
\|u_{1,S}^{(2)}-u_{2,S}^{(2)}\|_{L^1(\Omega\times(0,\tau))}\right)\big)\Big],\quad
t\in(0,T],
  \end{multline}
holds true.

Here,
\begin{multline}
\label{e9.11}%
M_4(t):=\max\Bigl\{M_5(\tau)+\|\beta_1\|_{C(\Xi^1\times[-M_5(\tau),M_5(\tau)])},
\|u_{1,0}^{(2)}\|_{L^\infty(\Omega\times(\tau,t))},\|u_{1,S}^{(2)}\|_{L^\infty(\Omega\times(\tau,t))},
\\
M_5(\tau)+\|\beta_2\|_{C(\Xi^1\times[-M_5(\tau),M_5(\tau)])},
\|u_{2,0}^{(2)}\|_{L^\infty(\Omega\times(\tau,t))},\|u_{2,S}^{(2)}\|_{L^\infty(\Omega\times(\tau,t))}
\Bigr\}\quad \text{for }t\in(\tau,T]
  \end{multline}
and
\begin{multline}
\label{e9.12}%
M_5(t):=\max\Bigl\{\|u_{1,0}^{(1)}\|_{L^\infty(\Xi^1)},\|u_{1,0}^{(2)}\|_{L^\infty(\Omega\times(0,t))},
\|u_{1,S}^{(2)}\|_{L^\infty(\Omega\times(0,t))},\\ \|u_{2,0}^{(1)}\|_{L^\infty(\Xi^1)},
\|u_{2,0}^{(2)}\|_{L^\infty(\Omega\times(0,t))},
\|u_{2,S}^{(2)}\|_{L^\infty(\Omega\times(0,t))}\Bigr\}\quad \text{for }
t\in(0,\tau].
  \end{multline}
{\bf 3.\;(Equivalency of the notions of kinetic and entropy
solutions.)} Function $u_*$ is an entropy
solution of Problem $\Pi_0$ in the sense of Definition
\ref{def.9.02}, if and only if it is a kinetic solution in the
sense of Definition \ref{def.9.01}.
\end{theorem}

We are going to fulfill justification of Theorem \ref{theo.9.1} in
the reverse order. That is, firstly, we prove assertion 3. Secondly,
we prove assertion 2. Finally, we establish assertion 1 (convergence
result), which is the most difficult part in the proof of Theorem
\ref{theo.9.1}.

\begin{remark}
\label{rem.9.02} From the arguments of Section \ref{Equiv} below, it becomes clear that the impulsive perturbation
$\beta=\beta(\boldsymbol{x},s,\lambda)$ and the nonlinearities
$a=a(\lambda)$ and
$\boldsymbol{\varphi}=\boldsymbol{\varphi}(\lambda)$ may not
necessarily meet restrictions \eqref{u1.07iv} and \eqref{e9.08i} in
assertions 2 and 3 of Theorem \ref{theo.9.1}.
\end{remark}
\section{Proof of assertions 2 and 3 of Theorem \ref{theo.9.1}} \label{Equiv} 
The proof of assertion 3 of Theorem \ref{theo.9.1} is quite similar
to the proof of assertion 3 of Theorem \ref{theo.3.1} in Section
\ref{sec.8}. Now let us show that assertion 2 of Theorem \ref{theo.9.1} directly
follows from assertion 1 of Theorem \ref{theo.3.1}.

We clearly notice that, in order to find a kinetic solution of
Problem $\Pi_0$, it is necessary and sufficient to fulfill the
following two steps. Firstly, in the subdomain
$\Omega\times(0,\tau)\times(0,S)\subset G_{T,S}$, we find the
kinetic solution $u_*$ of Problem $\Pi_\gamma$ provided with the
homogeneous right-hand side $Z_\gamma\equiv0$ and the initial data
\eqref{e1.01b}. Secondly, in the subdomain
$\Omega\times(\tau,T)\times(0,S)\subset G_{T,S}$ we find the kinetic
solution $u_*$ of Problem $\Pi_\gamma$ provided with
$Z_\gamma\equiv0$ and the initial data
\begin{equation}
\label{e10.01}%
u_*(\boldsymbol{x},t,s)\vert_{t=\tau}=u_*(\boldsymbol{x},\tau-0,s)+\beta(\boldsymbol{x},s,u_*(\boldsymbol{x},\tau-0,s)),\quad
(\boldsymbol{x},s)\in \Xi^1.
\end{equation}
In the right-hand side of \eqref{e10.01},
$u_*(\boldsymbol{x},\tau-0,s)$ is the trace on $\{t=\tau-0\}$ of the
kinetic solution $u_*$ obtained on the first step of the above
described procedure.

On the strength of assertion 1 of Theorem \ref{theo.3.1}, each of
the problems considered on the first and second steps has the unique
kinetic solution. Thus, Problem $\Pi_0$ has the unique kinetic
solution in the sense of Definition \ref{def.9.01}. The maximum
principle \eqref{e9.07b} directly follows from \eqref{e2.03}, since
we have $b_\gamma^{(1)}=b_\gamma^{(2)}=0$ in \eqref{e2.03} in the case
when $Z_\gamma\equiv0$.

Analogously, estimate \eqref{e9.10} follows from estimate
\eqref{e3.05}. More precisely, estimate \eqref{e3.05} with
$Z_\gamma\equiv0$ reads as follows:
\begin{multline}
\label{e10.02a}%
\|
u_{*1}(\cdot,t,\cdot)-u_{*2}(\cdot,t,\cdot)\|_{L^1(\Xi^1)}\leqslant
\| u_{1,0}^{(1)}-u_{2,0}^{(1)}\|_{L^1(\Xi^1)}
\\
+\|a'\|_{C[-M_5(t),M_5(t)]}\left(\|
u_{1,0}^{(2)}-u_{2,0}^{(2)}\|_{L^1(\Omega\times(0,t))}+ \|
u_{1,S}^{(2)}-u_{2,S}^{(2)}\|_{L^1(\Omega\times(0,t))}\right)\quad \text{for }t\in(0,\tau)
\end{multline}
and
\begin{multline}
\label{e10.02b}%
\|
u_{*1}(\cdot,t,\cdot)-u_{*2}(\cdot,t,\cdot)\|_{L^1(\Xi^1)}\leqslant
\\
\|
u_{*1}(\cdot,\tau-0,\cdot)+\beta_1(\cdot,\cdot,u_{*1}(\cdot,\tau-0,\cdot))
-u_{*2}(\cdot,\tau-0,\cdot)-\beta_2(\cdot,\cdot,u_{*2}(\cdot,\tau-0,\cdot))\|_{L^1(\Xi^1)}
\\
+\| a'\|_{C[-M_4(t),M_4(t)]}\left(\|
u_{1,0}^{(2)}-u_{2,0}^{(2)}\|_{L^1(\Omega\times(\tau,t))}+\|
u_{1,S}^{(2)}-u_{2,S}^{(2)}\|_{L^1(\Omega\times(\tau,t))}\right)\quad \text{for }t\in[\tau,T].
\end{multline}

We evaluate in the right-hand side of \eqref{e10.02b} to get
\begin{multline}
\label{e10.03}%
\left\Vert
u_{*1}(\cdot,\tau-0,\cdot)+\beta_1(\cdot,\cdot,u_{*1}(\cdot,\tau-0,\cdot))
-u_{*2}(\cdot,\tau-0,\cdot)-\beta_2(\cdot,\cdot,u_{*2}(\cdot,\tau-0,\cdot))\right\Vert_{L^1(\Xi^1)}\leqslant
\\
\left\Vert
u_{*1}(\cdot,\tau-0,\cdot)-u_{*2}(\cdot,\tau-0,\cdot)\right\Vert_{L^1(\Xi^1)}
+\\
\left\Vert\beta_1(\cdot,\cdot,u_{*1}(\cdot,\tau-0,\cdot))-
\beta_1(\cdot,\cdot,u_{*2}(\cdot,\tau-0,\cdot))\right\Vert_{L^1(\Xi^1)}+
\\
\left\Vert\beta_1(\cdot,\cdot,u_{*2}(\cdot,\tau-0,\cdot))-
\beta_2(\cdot,\cdot,u_{*2}(\cdot,\tau-0,\cdot))\right\Vert_{L^1(\Xi^1)}\leqslant
\\
\| u_{1,0}^{(1)}-u_{2,0}^{(1)}\|_{L^1(\Xi^1)}+\|
a'\|_{C[-M_5(\tau),M_5(\tau)]}\bigl(\|
u_{1,0}^{(2)}-u_{2,0}^{(2)}\|_{L^1(\Omega\times(0,\tau))}+
\|
u_{1,S}^{(2)}-u_{2,S}^{(2)}\|_{L^1(\Omega\times(0,\tau))}
\bigr)+
\\
\max\limits_{\Xi^1\times[-M_5(\tau),M_5(\tau)]}|
\partial_\lambda \beta_1(\boldsymbol{x},s,\lambda)|\Bigl[
\| u_{1,0}^{(1)}-u_{2,0}^{(1)}\|_{L^1(\Xi^1)}+\|
a'\|_{C[-M_5(\tau),M_5(\tau)]}\bigl(\|
u_{1,0}^{(2)}-u_{2,0}^{(2)}\|_{L^1(\Omega\times(0,\tau))}+
\\
\|u_{1,S}^{(2)}-u_{2,S}^{(2)}\|_{L^1(\Omega\times(0,\tau))}
\bigr)\Bigr]+(\mathrm{meas\,}\Xi^1)\left\Vert\beta_1-\beta_2\right\Vert_{C(\Xi^1\times[-M_5(\tau),M_5(\tau)])}.
\end{multline}

Remark that $\mathrm{meas\,}\Xi^1=S\,\mathrm{meas\,}\Omega$,
$[-M_5(t),M_5(t)]\subset[-M_5(\tau),M_5(\tau)]$ $\forall\,
t\in(0,\tau]$, and $[-M_5(\tau),M_5(\tau)]\subset[-M_4(t),M_4(t)]
\subset[-M_4(T),M_4(T)]$ $\forall\, t\in[\tau,T)$. Taking this into
account and combining \eqref{e10.02a}, \eqref{e10.02b} and
\eqref{e10.03}, we arrive at \eqref{e9.10}, which completes the
proof of assertion 2 of Theorem \ref{theo.9.1}.

\section{Uniform in $\gamma$ estimates of the family $\boldsymbol{\{u_\gamma\}_{\gamma>0}}$} \label{Unif} 
In order to pass to the limit as $\gamma\to0+$ in the
ultra-parabolic equation \eqref{e9.00} on the rigorous mathematical
level, we need to establish some appropriate uniform in $\gamma$
estimates of the family of kinetic solutions $u_\gamma$ of Problem
$\Pi_\gamma$, $\gamma>0$.
To this end, at first, we revisit the strictly regularized formulation, i.e., Problem
$\Pi_{\gamma\varepsilon}$ incorporating the source term
\eqref{u1.06}, and build the refined energy estimate.

\begin{lemma}
\label{lem.11.2}%
Let the source term $Z_\gamma$ in Problem $\Pi_{\gamma\varepsilon}$
have the form \eqref{u1.06}, where functions $K_\gamma$ and $\beta$
satisfy Conditions on $K_\gamma\&\beta$. Let the nonlinearities
$a=a(\lambda)$ and
$\boldsymbol{\varphi}=\boldsymbol{\varphi}(\lambda)$ satisfy demands
of items (i) and (ii) in Conditions on
$a\&\boldsymbol{\varphi}\&Z_\gamma$ and the additional bounds
\eqref{e9.08i}.

Set
\begin{equation}
\label{e11.00}%
\gamma_0:=\min\{\tau,T-\tau\}.
\end{equation}

Then the family of classical solutions
$\{u_{\gamma\varepsilon}\}_{{}_{\varepsilon\in(0,1]}^{\gamma\in(0,\gamma_0]}}$
of Problem $\Pi_{\gamma\varepsilon}$ satisfies the energy estimate
\begin{equation}
\label{e11.10}%
\Vert u_{\gamma\varepsilon}(t')\Vert_{L^2(\Xi^1)}+ \Vert\nabla_x
u_{\gamma\varepsilon}\Vert_{L^2(\Xi^1\times(0,t'))}+\sqrt{\varepsilon}
\Vert\partial_s
u_{\gamma\varepsilon}\Vert_{L^2(\Xi^1\times(0,t'))}\leqslant
C_6\quad\forall\, t'\in(0,T],
  \end{equation}
where the positive constant $C_6$ does not depend on $\gamma$ and
$\varepsilon$. Thus, \eqref{e11.10} is the uniform bound on the
family
$\{u_{\gamma\varepsilon}\}_{{}_{\varepsilon\in(0,1]}^{\gamma\in(0,\gamma_0]}}$.
  \end{lemma}

The constant $C_6$ will be defined explicitly in the proof, see formula \eqref{e11.25} below.

\proof Justification of Lemma \ref{lem.11.2} is fulfilled by means
of the standard techniques for derivation of energy estimates.

Notice that due to Conditions on $u_0^{(1)}\& u_0^{(2)}\& u_S^{(2)}$
we can take $\hat{u}$ belonging to $C^{2+\alpha}(G_{T,S})$. For
example $\hat{u}$ can be constructed explicitly in terms of
$u_0^{(1)}$, $u_0^{(2)}$ and $u_S^{(2)}$ as the solution of the
nonhomogeneous Dirichlet problem for Laplace's equation \cite[Section
2.2.4]{E-2010}:
\[
\Delta_{x,s,t}\hat{u}=0\text{ in
}G_{T,S},\quad \hat{u}\vert_{\Gamma_l}=0,\quad \hat{u}\vert_{\Gamma_0^1}=u_0^{(1)},\quad\hat{u}\vert_{\Gamma_T^1}=0,
\quad\hat{u}\vert_{\Gamma_0^2}=u_0^{(2)},\quad\hat{u}\vert_{\Gamma_S^2}=u_S^{(2)}.
\]

Rewrite \eqref{e9.00} in the form
\begin{multline}
\label{e11.11}%
\partial_t(u_{\gamma\varepsilon}-\hat{u})+\partial_s
a(u_{\gamma\varepsilon})+
\mathrm{div}_x\boldsymbol{\varphi}(u_{\gamma\varepsilon})-\Delta_x(u_{\gamma\varepsilon}-\hat{u})-\\
\varepsilon\partial_s^2(u_{\gamma\varepsilon}-\hat{u})-K_\gamma(t,\tau)(\beta(\boldsymbol{x},s,u_{\gamma\varepsilon})-
\beta(\boldsymbol{x},s,\hat{u}))=
\\
-\partial_t\hat{u}+\Delta_x\hat{u}+\varepsilon\partial_s^2\hat{u}+
K_\gamma(t,\tau)\beta(\boldsymbol{x},s,\hat{u}).
\end{multline}

Multiply the both sides of \eqref{e11.11} by
$u_{\gamma\varepsilon}-\hat{u}$, integrate over $\Xi^1\times(0,t')$,
where value $t'\in(0,T]$ is taken arbitrarily, and consider
 terms in the resulting equality one by one. Applying
integration by parts and Green's formula, we derive
\begin{multline*}
\int\limits_0^{t'}\int\limits_{\Xi^1}
(\partial_t(u_{\gamma\varepsilon}-\hat{u}))(u_{\gamma\varepsilon}-\hat{u})\,d\boldsymbol{x}dsdt=\\
\frac12\int\limits_{\Xi^1}
|u_{\gamma\varepsilon}(\boldsymbol{x},t',s)
-\hat{u}(\boldsymbol{x},t',s)|^2\,d\boldsymbol{x}ds-\frac12\int\limits_{\Xi^1}
|u_0^{(1)}-u_0^{(1)}|^2\,d\boldsymbol{x}ds=
\\
\frac12\int\limits_{\Xi^1}
|u_{\gamma\varepsilon}(\boldsymbol{x},t',s)
-\hat{u}(\boldsymbol{x},t',s)|^2\,d\boldsymbol{x}ds;
\end{multline*}
\begin{multline*}
\int\limits_0^{t'}\int\limits_{\Xi^1}
\partial_sa(u_{\gamma\varepsilon})(u_{\gamma\varepsilon}-\hat{u})\,d\boldsymbol{x}dsdt=
\int\limits_0^{t'}\int\limits_{\Xi^1}
\partial_s\tilde{q}_a(u_{\gamma\varepsilon})\,d\boldsymbol{x}dsdt-
\int\limits_0^{t'}\int\limits_{\Xi^1}
\partial_sa(u_{\gamma\varepsilon})\hat{u}\,d\boldsymbol{x}dsdt=
\\
\int\limits_0^{t'}\int\limits_\Omega
\left(\tilde{q}_a(u_S^{(2)})-\tilde{q}_a(u_0^{(2)})-u_S^{(2)}
a(u_S^{(2)})+u_0^{(2)} a(u_0^{(2)})\right)\,d\boldsymbol{x}dt+
\int\limits_0^{t'}\int\limits_{\Xi^1}
a(u_{\gamma\varepsilon})\partial_s\hat{u}\,d\boldsymbol{x}dsdt=
\\
\int\limits_0^{t'}\int\limits_\Omega
\left(\tilde{q}_a(u_S^{(2)})-\tilde{q}_a(u_0^{(2)})
+A(u_S^{(2)})-A(u_0^{(2)})-u_S^{(2)} a(u_S^{(2)})+u_0^{(2)}
a(u_0^{(2)})\right)\,d\boldsymbol{x}dt
\\
+\int\limits_0^{t'}\int\limits_{\Xi^1}
(a(u_{\gamma\varepsilon})-a(\hat{u}))\partial_s\hat{u}\,d\boldsymbol{x}dsdt=
\int\limits_0^{t'}\int\limits_{\Xi^1}
(a(u_{\gamma\varepsilon})-a(\hat{u}))\partial_s\hat{u}\,d\boldsymbol{x}dsdt;
\end{multline*}
where $\displaystyle \tilde{q}_a(\lambda)=\int\limits_0^\lambda
a'(\lambda)\lambda\,d\lambda$, $\displaystyle A(\lambda)=\int\limits_0^\lambda
a(\lambda)\,d\lambda$, and we have noticed that
\[
\tilde{q}_a(\lambda)+A(\lambda)-a(\lambda)\lambda=0\quad\forall\,\lambda\in\mathbb{R};
\]
\begin{multline*}
\int\limits_0^{t'}\int\limits_{\Xi^1}
(\mathrm{div}_x\boldsymbol{\varphi}(u_{\gamma\varepsilon}))(u_{\gamma\varepsilon}-\hat{u})\,d\boldsymbol{x}dsdt=\\
\int\limits_0^{t'}\int\limits_0^S\int\limits_{\partial\Omega}
\widetilde{\boldsymbol{q}}_\varphi(u_{\gamma\varepsilon}(\boldsymbol{\sigma},t,s))\cdot
\boldsymbol{n}(\boldsymbol{\sigma})\,d\boldsymbol{\sigma}dsdt
- \int\limits_0^{t'}\int\limits_0^S\int\limits_{\partial\Omega}
\boldsymbol{\varphi}(u_{\gamma\varepsilon}(\boldsymbol{\sigma},t,s))
\hat{u}(\boldsymbol{\sigma},t,s)\cdot
\boldsymbol{n}(\boldsymbol{\sigma})\,d\boldsymbol{\sigma}dsdt+\\
\int\limits_0^{t'}\int\limits_{\Xi^1}
\boldsymbol{\varphi}(u_{\gamma\varepsilon})
\cdot\nabla_x\hat{u}\,d\boldsymbol{x}dsdt\pm\int\limits_0^{t'}\int\limits_{\Xi^1}
\boldsymbol{\varphi}(\hat{u})\cdot\nabla_x\hat{u}\,d\boldsymbol{x}dsdt=
\\
\int\limits_0^{t'}\int\limits_{\Xi^1}
(\boldsymbol{\varphi}(u_{\gamma\varepsilon})-
\boldsymbol{\varphi}(\hat{u}))\cdot\nabla_x\hat{u}\,d\boldsymbol{x}dsdt+
\int\limits_0^{t'}\int\limits_0^S\int\limits_{\partial\Omega}
\widetilde{\boldsymbol{q}}_\varphi(\hat{u}(\boldsymbol{\sigma},t,s))\cdot
\boldsymbol{n}(\boldsymbol{\sigma})\,d\boldsymbol{\sigma}dsdt=\\
\int\limits_0^{t'}\int\limits_{\Xi^1}
(\boldsymbol{\varphi}(u_{\gamma\varepsilon})-
\boldsymbol{\varphi}(\hat{u}))\cdot\nabla_x\hat{u}\,d\boldsymbol{x}dsdt,
\end{multline*}
where
$\displaystyle \widetilde{\boldsymbol{q}}_\varphi(\lambda)=\int\limits_0^\lambda
\boldsymbol{\varphi}'(\lambda)\lambda\,d\lambda,$ $\boldsymbol{n}=
\boldsymbol{n}(\sigma)$ is the unit outward normal to
$\partial\Omega$, and we have noticed that
$\widetilde{\boldsymbol{q}}_\varphi(u_{\gamma\varepsilon})$,
$\widetilde{\boldsymbol{q}}_\varphi(\hat{u})$ and $\hat{u}$ vanish
on $\partial\Omega$;
\begin{equation*}
-\int\limits_0^{t'}\int\limits_{\Xi^1}
(\Delta_x(u_{\gamma\varepsilon}-\hat{u}))(u_{\gamma\varepsilon}-\hat{u})\,d\boldsymbol{x}dsdt=
\int\limits_0^{t'}\int\limits_{\Xi^1}
|\nabla_x(u_{\gamma\varepsilon}-\hat{u})|^2\,d\boldsymbol{x}dsdt;
\end{equation*}
\begin{equation*}
-\int\limits_0^{t'}\int\limits_{\Xi^1}
(\partial_s^2(u_{\gamma\varepsilon}-\hat{u}))\varepsilon(u_{\gamma\varepsilon}-\hat{u})\,d\boldsymbol{x}dsdt=
\int\limits_0^{t'}\int\limits_{\Xi^1}
\varepsilon\vert\partial_s(u_{\gamma\varepsilon}-\hat{u})\vert^2\,d\boldsymbol{x}dsdt.
  \end{equation*}
Using these representations, integrating by parts in $s$, using
Green's formula with respect to $\boldsymbol{x}$ in the right-hand side, and
properly arranging summands, from \eqref{e11.11} we derive the first
energy identity as follows:
\begin{multline*}
\frac12\int\limits_{\Xi^1}\vert
u_{\gamma\varepsilon}(\boldsymbol{x},t',s)-\hat{u}(\boldsymbol{x},t',s)\vert^2\,d\boldsymbol{x}ds+\\
\int\limits_0^{t'}\int\limits_{\Xi^1}
\vert\nabla_x(u_{\gamma\varepsilon}-\hat{u})\vert^2\,d\boldsymbol{x}dsdt
+\varepsilon\int\limits_0^{t'}\int\limits_{\Xi^1}
\vert\partial_s(u_{\gamma\varepsilon}-\hat{u})\vert^2\,d\boldsymbol{x}dsdt=
\\
-\int\limits_0^{t'}\int\limits_{\Xi^1}
(a(u_{\gamma\varepsilon})-a(\hat{u}))\partial_s\hat{u}\,d\boldsymbol{x}dsdt
-\int\limits_0^{t'}\int\limits_{\Xi^1}
\left(\boldsymbol{\varphi}(u_{\gamma\varepsilon})-\boldsymbol{\varphi}(\hat{u})\right)\cdot\nabla_x\hat{u}
\,d\boldsymbol{x}dsdt+
\\
\int\limits_0^{t'}K_\gamma(t,\tau)\int\limits_{\Xi^1}
\left(\beta(\boldsymbol{x},s,u_{\gamma\varepsilon})-\beta(\boldsymbol{x},s,\hat{u})\right)(u_{\gamma\varepsilon}-\hat{u})\,d\boldsymbol{x}dsdt
- \int\limits_0^{t'}\int\limits_{\Xi^1}
\partial_t\hat{u}(u_{\gamma\varepsilon}-\hat{u})\,d\boldsymbol{x}dsdt-
\end{multline*}
\begin{multline}
\int\limits_0^{t'}\int\limits_{\Xi^1}
\nabla_x\hat{u}\cdot\nabla(u_{\gamma\varepsilon}-\hat{u})\,d\boldsymbol{x}dsdt
-\varepsilon\int\limits_0^{t'}\int\limits_{\Xi^1}
\partial_s\hat{u}\,\partial_s(u_{\gamma\varepsilon}-\hat{u})\,d\boldsymbol{x}dsdt+
\\
\int\limits_0^{t'}K_\gamma(t,\tau)\int\limits_{\Xi^1}
\beta(\boldsymbol{x},s,\hat{u})(u_{\gamma\varepsilon}-\hat{u})\,d\boldsymbol{x}dsdt,\quad
t'\in(0,T].
\label{e11.12}%
\end{multline}

From the Lagrange mean value theorem it follows that
\begin{equation}
\label{e11.13}%
\left\vert
a(u_{\gamma\varepsilon})-a(\hat{u})\right\vert\leqslant\max\limits_{\lambda\in\mathbb{R}}
\vert a'(\lambda)\vert \vert
u_{\gamma\varepsilon}-\hat{u}\vert\leqslant C_5 \vert
u_{\gamma\varepsilon}-\hat{u}\vert,
\end{equation}
\begin{equation}
\label{e11.14}%
\left\vert \boldsymbol{\varphi}(u_{\gamma\varepsilon})-
\boldsymbol{\varphi}(\hat{u})\right\vert\leqslant\max\limits_{\lambda\in\mathbb{R}}
\vert \boldsymbol{\varphi}'(\lambda)\vert \vert
u_{\gamma\varepsilon}-\hat{u}\vert\leqslant C_5\vert
u_{\gamma\varepsilon}-\hat{u}\vert,
\end{equation}
and
\begin{equation}
\label{e11.15}%
\left\vert \beta(\boldsymbol{x},s,u_{\gamma\varepsilon})-
\beta(\boldsymbol{x},s,\hat{u})\right\vert\leqslant\left(
\max\limits_{(\boldsymbol{x},s,\lambda)\in\Xi^1\times\mathbb{R}}
\vert
\partial_\lambda\beta(\boldsymbol{x},s,\lambda)\vert\right)\,\vert
u_{\gamma\varepsilon}-\hat{u}\vert\leqslant b_0\vert
u_{\gamma\varepsilon}-\hat{u}\vert.
\end{equation}

Using these estimates and Young's inequality, from \eqref{e11.12} we
derive the following inequality:
\begin{multline}
\label{e11.16}%
\frac12\|u_{\gamma\varepsilon}(t')-\hat{u}(t')\|_{L^2(\Xi^1)}^2+\left(1-\frac{\delta_4}2\right)
\int\limits_0^{t'}\|\nabla_x(u_{\gamma\varepsilon}(t)-\hat{u}(t))\|_{L^2(\Xi^1)}^2\,dt+\\
\left(1-\frac{\delta_5}2\right)
\varepsilon\int\limits_0^{t'}\|\partial_s(u_{\gamma\varepsilon}(t)-\hat{u}(t))\|_{L^2(\Xi^1)}^2\,dt\leqslant
\\
\int\limits_0^{t'} \left( \frac{\delta_1}2C_5+\frac{\delta_2}2C_5+
\frac{\delta_3}2+\frac{\delta_6}2K_\gamma(t,\tau)+b_0K_\gamma(t,\tau)\right)
\Vert u_{\gamma\varepsilon}(t)-\hat{u}(t)\Vert_{L^2(\Xi^1)}^2\,dt+
\\
\int\limits_0^{t'}\Biggl[
\left(\frac{C_5}{2\delta_2}+\frac{1}{2\delta_4}\right)\Vert\nabla_x\hat{u}(t)\Vert_{L^2(\Xi^1)}^2+
\left(\frac{C_5}{2\delta_1}+\frac{\varepsilon}{2\delta_5}\right)\|\partial_s\hat{u}(t)\|_{L^2(\Xi^1)}^2+
\\
\frac1{2\delta_3} \Vert\partial_t\hat{u}(t)\Vert_{L^2(\Xi^1)}^2+
\frac1{2\delta_6}\Vert\beta\Vert_{C(\Xi^1\times[-M_6,M_6])}^2(\mathrm{meas\,}\Omega)
S\,K_\gamma(t,\tau)\Biggr]\,dt,\quad t'\in(0,T],
\end{multline}
where $M_6=\|\hat{u}\|_{C(\overline{G}_{T,S})}$ and
$\delta_k\in\mathbb{R}$ $(k=1,\ldots,6)$ are arbitrary. Take
$\delta_1=\delta_2=\ldots=\delta_6=1$ and denote
\begin{equation}
\label{e11.17}%
C_7(\gamma,t):=2C_5+1+(1+2b_0)K_\gamma(t,\tau)
  \end{equation}
and
\begin{multline}
\label{e11.18}%
C_8(\gamma,\varepsilon,t):=(C_5+1)\Vert\nabla_x\hat{u}(t)\Vert_{L^2(\Xi^1)}^2+
\left(C_5+\varepsilon\right)\Vert\partial_s\hat{u}(t)\Vert_{L^2(\Xi^1)}^2
+\\
\Vert\partial_t\hat{u}(t)\Vert_{L^2(\Xi^1)}^2+
\Vert\beta\Vert_{C(\Xi^1\times[-M_7,M_7])}^2(\mathrm{meas\,}\Omega)
S\,K_\gamma(t,\tau)
  \end{multline}
for the sake of conciseness. With this choice of $\delta_k$ and
notation, inequality \eqref{e11.16} reduces to
\begin{multline}
\label{e11.19}%
\Vert
u_{\gamma\varepsilon}(t')-\hat{u}(t')\Vert_{L^2(\Xi^1)}^2+\int\limits_0^{t'}
\Vert\nabla_x(u_{\gamma\varepsilon}(t)-\hat{u}(t))\Vert_{L^2(\Xi^1)}^2\,dt+
\varepsilon\int\limits_0^{t'}
\Vert\partial_s(u_{\gamma\varepsilon}(t)-\hat{u}(t))\Vert_{L^2(\Xi^1)}^2\,dt\leqslant
\\
\int\limits_0^{t'}\left(C_7(\gamma,t)\Vert
u_{\gamma\varepsilon}(t)-\hat{u}(t)\Vert_{L^2(\Xi^1)}^2+C_8(\gamma,\varepsilon,t)\right)\,dt\quad
\forall t'\in(0,T].
\end{multline}
Here, discarding the second and the third terms in the left-hand side
and applying Gr\"onwall's lemma, we establish the bound
\begin{equation}
\label{e11.20}%
\Vert
u_{\gamma\varepsilon}(t')-\hat{u}(t')\Vert_{L^2(\Xi^1)}^2\leqslant
\exp\left\{\int\limits_0^{t'}C_7(\gamma,t)\,dt\right\}\left(
\int\limits_0^{t'}
C_8(\gamma,\varepsilon,t)\exp\left\{-\int\limits_0^t
C_7(\gamma,\xi)\,d\xi\right\}dt\right).
\end{equation}
Combining \eqref{e11.19} and \eqref{e11.20} we obtain the estimate
\begin{multline}
\label{e11.21}%
\Vert
u_{\gamma\varepsilon}(t')-\hat{u}(t')\Vert_{L^2(\Xi^1)}^2+\int\limits_0^{t'}
\Vert\nabla_x(u_{\gamma\varepsilon}(t)-\hat{u}(t))\Vert_{L^2(\Xi^1)}^2\,dt+
\varepsilon\int\limits_0^{t'}
\Vert\partial_s(u_{\gamma\varepsilon}(t)-\hat{u}(t))\Vert_{L^2(\Xi^1)}^2\,dt\leqslant
\\
\int\limits_0^{t'}\Bigl(C_7(\gamma,t)\exp\Bigl\{\int\limits_0^{t}C_7(\gamma,t'')\,dt''\Bigr\}\Bigl(
\int\limits_0^t
C_8(\gamma,\varepsilon,t'')\exp\Bigl\{-\int\limits_0^{t''}
C_7(\gamma,\xi)\,d\xi\Bigr\}dt''\Bigr)
+C_8(\gamma,\varepsilon,t)\Bigr)\,dt\\ \forall t'\in(0,T].
\end{multline}
In the right-hand side of \eqref{e11.21} we easily see that
\begin{equation}
\label{e11.22}%
0\leqslant\int\limits_0^t C_7(\gamma,t'')\,dt''\leqslant
2C_5 T+1+2b_0\equiv C_9\quad \forall\,\gamma\in(0,\gamma_0],\quad\forall\,t\in(0,T],
\end{equation}
\begin{multline}
\label{e11.23}%
0\leqslant\int\limits_0^t C_8(\gamma,\varepsilon,t'')\,dt''\leqslant
(C_5+1)\left(\Vert\nabla_x\hat{u}\Vert_{L^2(0,T;L^2(\Xi^1))}^2+
\Vert\partial_s\hat{u}\Vert_{L^2(0,T;L^2(\Xi^1))}^2\right)
+
\\
\Vert\partial_t\hat{u}\Vert_{L^2(0,T;L^2(\Xi^1))}^2+S\,\mathrm{meas\,}\Omega\,\Vert\beta\Vert_{C(\Xi^1\times[-M_6,M_6])}^2\equiv
C_{10}\\ \forall\,\gamma\in(0,\gamma_0],\; \forall\,\varepsilon\in(0,1],\;\forall\,t\in(0,T]
\end{multline}
due to \eqref{e11.17}, \eqref{e11.18} and due to the fact that $\mathrm{supp\,}K_\gamma(\cdot,\tau)\subset[0,T]$ for
$\gamma\in(0,\gamma_0]$, and therefore $\int\limits_0^TK_\gamma(t'',\tau)\,dt''=1$. Using \eqref{e11.22} and \eqref{e11.23}, from \eqref{e11.21} we
derive the bound
\begin{multline}
\label{e11.24}%
\Vert u_{\gamma\varepsilon}(t')-\hat{u}(t')\Vert_{L^2(\Xi^1)}^2+
\Vert\nabla_x(u_{\gamma\varepsilon}-\hat{u})\Vert_{L^2(\Xi^1\times(0,t'))}^2+
\varepsilon
\Vert\partial_s(u_{\gamma\varepsilon}-\hat{u})\Vert_{L^2(\Xi^1\times(0,t'))}^2\leqslant
(C_9 e^{C_9}+1)C_{10}\\
\forall\,\gamma\in(0,\gamma_0],\quad \forall\,\varepsilon\in (0,1],\quad \forall\,
t'\in(0,T].
\end{multline}
Using the elementary inequality
\[
\frac{A_1+A_2+A_3}3\leqslant\sqrt{\frac{A_1^2+A_2^2+A_3^2}{3}}\quad
\forall\, A_1,A_2,A_3\in\mathbb{R}
\]
and the triangle inequality, from \eqref{e11.24} we finally deduce
\eqref{e11.10} with
\begin{equation}
\label{e11.25}%
C_6=\sqrt{3C_{10}(C_9 e^{C_9}+1)}+
\Vert\hat{u}\Vert_{L^\infty(0,T;L^2(\Xi^1))}+
\Vert\nabla_x\hat{u}\Vert_{L^2(G_{T,S})}+
\Vert\partial_s\hat{u}\Vert_{L^2(G_{T,S})}.
  \end{equation}
Lemma \ref{lem.11.2} is proved.\qed\\[1ex]
\indent Recall that the kinetic equation \eqref{e3.01a} with
$Z_\gamma=K_\gamma\beta$ can be written in the form
\begin{multline}
\label{e11.26}%
\partial_t\chi(\lambda;u_\gamma)+a'(\lambda)
\partial_s\chi(\lambda;u_\gamma)+\boldsymbol{\varphi}'(\lambda)\cdot
\nabla_x\chi(\lambda;u_\gamma)-\Delta_x\chi(\lambda;u_\gamma)=\partial_\lambda
(m_\gamma+n_\gamma+\mathcal{Z}_\gamma),
\\
(\boldsymbol{x},t,s,\lambda)\in G_{T,S}\times\mathbb{R}_\lambda,
  \end{multline}
where
\begin{equation}
\label{e11.26ii}%
\mathcal{Z}_\gamma\overset{\mathrm{def}}{=}\mathbf{1}_{(\lambda\geqslant
u_\gamma)}K_\gamma(t,\tau)\beta(\boldsymbol{x},s,\lambda)-\int\limits_0^\lambda
\mathbf{1}_{(\lambda'\geqslant
u_\gamma)}K_\gamma(t,\tau)\partial_{\lambda'}\beta(\boldsymbol{x},s,\lambda')\,d\lambda',
  \end{equation}
  see equation \eqref{e3.01bis}.

Recall from Lemma \ref{lem.4.2} and Remark \ref{rem.7.2} that
$u_\gamma$ appears as the strong limiting point of the sequence
$\{u_{\gamma\varepsilon}\}_{\varepsilon\to0}$. This fact  along with
Lemma \ref{lem.11.2} has the following implications.

\begin{lemma} \label{lem.11.3}
{\bf (i)} The family $\{u_\gamma\}_{\gamma\in(0,\gamma_0]}$ satisfies the
energy estimate
\begin{equation}
\label{e11.27}%
\Vert u_\gamma(t')\Vert_{L^2(\Xi^1)}+\Vert\nabla_x
u_\gamma\Vert_{L^2(\Xi^1\times(0,t'))}\leqslant C_6\quad \forall\,
t'\in(0,T].
\end{equation}
{\bf (ii)} The families of measures $m_\gamma\in\mathcal{M}^+(G_{T,S}\times\mathbb{R}_\lambda)$
and
$n_\gamma=\delta_{(\lambda=u_\gamma)}|\nabla_xu_\gamma|^2\in\mathcal{M}^+(G_{T,S}\times\mathbb{R}_\lambda)$
are bounded uniformly in $\gamma$. More precisely,
\begin{equation}
\label{e11.28}%
\Vert
m_\gamma\Vert_{\mathcal{M}(G_{T,S}\times\mathbb{R}_\lambda)}\leqslant
3C_6^2,\quad \Vert
n_\gamma\Vert_{\mathcal{M}(G_{T,S}\times\mathbb{R}_\lambda)}\leqslant
C_6^2\quad\forall\,\gamma\in(0,\gamma_0].
\end{equation}
{\bf (iii)} There exists a constant $C_{11}>0$, independent of
$\gamma$, such that
\begin{equation}
\label{e11.29}%
\Vert\mathcal{Z}_\gamma\Vert_{L^1(G_{T,S}\times\mathbb{R}_\lambda)}\leqslant
C_{11}\quad\forall\,\gamma\in(0,\gamma_0].
\end{equation}
\end{lemma}
\proof {\bf (i)} Energy estimate \eqref{e11.27} immediately follows from
\eqref{e11.10}, Remark \ref{rem.7.2}, and the lower semicontinuity
property: $\displaystyle \Vert\nabla_xu_\gamma\Vert_{L^2(\Xi^1\times(0,t'))} \leqslant
\liminf\limits_{\varepsilon\to0+} \Vert
\nabla_xu_{\gamma\varepsilon}\Vert_{L^2(\Xi^1\times(0,t'))}$.\\
{\bf (ii)} For any $\phi\in L^{\infty}(G_{T,S}\times\mathbb{R}_\lambda)$
we have
\begin{multline*}
\vert\langle n_\gamma,\phi\rangle\vert=\left\vert
\,\int\limits_{G_{T,S}}\left\vert\nabla_xu_\gamma\right\vert^2\phi(\boldsymbol{x},t,s,u_\gamma)\,
d\boldsymbol{x}dtds\right\vert\leqslant\\
\left\Vert\nabla_xu_\gamma\right\Vert_{L^2(G_{T,S})}^2
\left\Vert\phi\right\Vert_{L^\infty(G_{T,S}\times\mathbb{R}_\lambda)}
\leqslant
C_6^2\left\Vert\phi\right\Vert_{L^\infty(G_{T,S}\times\mathbb{R}_\lambda)},
  \end{multline*}
and
\begin{multline*}
\vert\langle m_\gamma,\phi\rangle\vert= \left\vert \left\langle
\lim\limits_{\varepsilon\to 0+}
\delta_{(\lambda=u_{\gamma\varepsilon})}\left(|\nabla_x
u_{\gamma\varepsilon}|^2+\varepsilon|\partial_s
u_{\gamma\varepsilon}|^2\right)-\delta_{(\lambda=u_\gamma)}|\nabla_x
u_\gamma|^2,\phi\right\rangle\right\vert=
\\
\left\vert\,\int\limits_{G_{T,S}}\left(
\lim\limits_{\varepsilon\to0+} \left(|\nabla_x
u_{\gamma\varepsilon}|^2+
\varepsilon|\partial_s u_{\gamma\varepsilon}|^2\right)
\phi(\boldsymbol{x},t,s,u_{\gamma\varepsilon})-
|\nabla_xu_{\gamma}|^2\phi(\boldsymbol{x},t,s,u_\gamma)\right)\,
d\boldsymbol{x}dtds\right\vert\leqslant\\
3C_6^2\left\Vert\phi\right\Vert_{L^\infty(G_{T,S}\times\mathbb{R}_\lambda)}
\end{multline*}
due to \eqref{e11.27}, \eqref{e4.12iii} and \eqref{e4.12iv}. This
immediately yields the bounds \eqref{e11.28}.
\\
{\bf (iii)} Simple evaluation gives
\begin{multline*}
\left\Vert\mathcal{Z}_\gamma\right\Vert_{L^1(G_{T,S}\times\mathbb{R}_\gamma)}=\\
\int\limits_{G_{T,S}\times\mathbb{R}_\lambda}
\Bigl\vert\mathbf{1}_{(\lambda\geqslant
u_\gamma)}K_\gamma(t,\tau)\beta(\boldsymbol{x},s,\lambda)-\int\limits_0^\lambda
\mathbf{1}_{(\lambda'\geqslant u_\gamma)}
K_\gamma(t,\tau)\partial_{\lambda'}\beta(\boldsymbol{x},s,\lambda')\,d\lambda'\Bigr\vert
d\boldsymbol{x}dtdsd\lambda\leqslant
\\
\int\limits_0^T\left(2b_1S(\mathrm{meas\,}\Omega)\|\beta\|_{C(\Xi^1\times[-b_1,b_1])}+
2b_1^2S\,(\mathrm{meas\,}\Omega)\|\partial_\lambda\beta\|_{C(\Xi^1\times[-b_1,b_1])}
\right)K_\gamma(t,\tau)\,dt=
\\
2b_1 S\,(\mathrm{meas\,}\Omega)\left(\Vert\beta\Vert_{C^1(\Xi^1\times[-b_1,b_1])}
+\Vert\partial_\lambda\beta\Vert_{C(\Xi^1\times[-b_1,b_1])}\right)\stackrel{def}{=}
C_{12}.
\end{multline*}
This estimate completes the proof of Lemma \ref{lem.11.3}.\qed

\begin{lemma}
\label{lem.11.4}%
Let the source term $Z_\gamma$ in Problem
$\Pi_{\gamma\varepsilon}$ have the form \eqref{u1.06}, where functions
$K_\gamma$ and $\beta$ satisfy Conditions on $K_\gamma\&\beta$. Let
the nonlinearities $a=a(\lambda)$ and
$\boldsymbol{\varphi}=\boldsymbol{\varphi}(\lambda)$ satisfy the demands
of items (i) and (ii) in Conditions on $a\&\boldsymbol{\varphi}\&
Z_\gamma$. Set $\gamma_0:=\min\{\tau,T-\tau\}$. Then the family of
classical solutions
$\{u_{\gamma\varepsilon}\}_{{}_{\varepsilon\in(0,1]}^{\gamma\in(0,{\gamma_0}/2]}}$
of Problem $\Pi_{\gamma\varepsilon}$ satisfies the maximum principle
\begin{multline}
\label{e11.30}%
\Vert
u_{\gamma\varepsilon}(\cdot,t',\cdot)\Vert_{L^\infty(\Xi^1)}\leqslant
e^{\xi_*
t'}\max\{\|u_0^{(1)}\|_{L^\infty(\Xi^1)},\|u_0^{(2)}\|_{L^\infty(\Xi^2)},
\|u_S^{(2)}\|_{L^\infty(\Xi^2)}\}\stackrel{def}{=}M_7(t')\\ \forall\,
t'\in(0,T],\;
\forall\,\gamma\in(0,\gamma_0/2],\; \forall\,\varepsilon\in(0,1],
  \end{multline}
where
\begin{equation}
\label{e11.31}%
\xi_*=\left\{
\begin{array}{lll}
0 &\mbox{ for }& \displaystyle b_1\leqslant\Vert u_0^{(1)}
\Vert_{L^\infty(\Xi^1)}-1,
\\[1ex] \displaystyle
\frac{2}{2\tau-\gamma_0}\ln\frac{b_1+1}{\Vert
u_0^{(1)}\Vert_{L^\infty(\Xi^1)}} &\mbox{ for }& \displaystyle b_1>\Vert u_0^{(1)}
\Vert_{L^\infty(\Xi^1)}-1.
\end{array}
\right.\end{equation} In particular,
$\{u_{\gamma\varepsilon}\}_{{}_{\varepsilon\in(0,1]}^{\gamma\in(0,{\gamma_0}/2]}}$
is uniformly in $\gamma$ and $\varepsilon$ bounded in
$L^\infty(G_{T,S})$.
\end{lemma}
\proof Substitute
$u_{\gamma\varepsilon}(\boldsymbol{x},t,s)=e^{\xi
t}\omega_{\gamma\varepsilon}(\boldsymbol{x},t,s)$ into
\eqref{e9.00}. Here
$\omega_{\gamma\varepsilon}=\omega_{\gamma\varepsilon}(\boldsymbol{x},t,s)$
is a new dependent variable and $\xi$ is an
arbitrarily fixed positive parameter. Upon division by $e^{\xi t}$, we get
\begin{multline}
\label{e11.32}%
\partial_t\omega_{\gamma\varepsilon}+\partial_s\left(
e^{-\xi t}a(\omega_{\gamma\varepsilon}e^{\xi t})\right)+
\mathrm{div}_x\left( e^{-\xi
t}\boldsymbol{\varphi}(\omega_{\gamma\varepsilon}e^{\xi
t})\right)=\\
\Delta_x\omega_{\gamma\varepsilon}+\varepsilon\partial_s^2\omega_{\gamma\varepsilon}+
K_\gamma(t,\tau)e^{-\xi
t}\beta(\boldsymbol{x},s,\omega_{\gamma\varepsilon}e^{\xi
t})-\xi\omega_{\gamma\varepsilon}.
\end{multline}
Also we have
\begin{equation*}
\omega_{\gamma\varepsilon}\vert_{t=0}=u_0^{(1)},\quad
\omega_{\gamma\varepsilon}\vert_{s=0}=u_0^{(2)}e^{-\xi t},\quad
\omega_{\gamma\varepsilon}\vert_{s=S}=u_S^{(2)}e^{-\xi t}.
\end{equation*}
Set
$M_*=\Vert\omega_{\gamma\varepsilon}\Vert_{L^\infty(\Gamma_0^1\cup\Gamma_0^2\cup\Gamma_S^2)}$
and
\begin{equation}
\label{e11.33}%
\omega_{\gamma\varepsilon}^{M_*}=(\omega_{\gamma\varepsilon}-M_*)^+\overset{\mathrm{def}}{=}\max\{\omega_{\gamma\varepsilon}-M_*,0\}.
\end{equation}
Remark that
\begin{equation}
\label{e11.34}%
\omega_{\gamma\varepsilon}^{M_*}\vert_{\partial
G_{T,S}\setminus\{t=T\}}=0,
\end{equation}
\begin{equation}
\label{e11.35}%
\partial_t\omega_{\gamma\varepsilon}^{M_*}=
\left\{\begin{array}{lll}
\partial_t\omega_{\gamma\varepsilon}
&\text{ for }&\omega_{\gamma\varepsilon}>M_*,
\\
0&\text{ for }&\omega_{\gamma\varepsilon}\leqslant M_*,
\end{array}\right.
\end{equation}
\begin{equation}
\label{e11.36}%
\partial_s\omega_{\gamma\varepsilon}^{M_*}=
\left\{\begin{array}{lll}
\partial_s\omega_{\gamma\varepsilon}
&\text{ for }&\omega_{\gamma\varepsilon}>M_*,
\\
0&\text{ for }&\omega_{\gamma\varepsilon}\leqslant M_*,
\end{array}\right.
\end{equation}
\begin{equation}
\label{e11.37}%
\nabla_x\omega_{\gamma\varepsilon}^{M_*}= \left\{\begin{array}{lll}
\nabla_x\omega_{\gamma\varepsilon}&\text{ for
}&\omega_{\gamma\varepsilon}>M_*,
\\
0&\text{ for }&\omega_{\gamma\varepsilon}\leqslant M_*,
\end{array}\right.
\end{equation}
and that $\omega_{\gamma\varepsilon}^{M_*}\in W^1_p(G_{T,S})$
$\forall\, p\in[1,+\infty)$ due to \cite[Appendix A.1, Lemma
1.7]{MNRR-1996}.

Multiply \eqref{e11.32} by $\omega_{\gamma\varepsilon}^{M_*}$ and
integrate on $\Xi^1\times(0,t')$:
\begin{multline*}
\int\limits_0^{t'}\int\limits_{\Xi^1}\omega_{\gamma\varepsilon}^{M_*}
\partial_t\omega_{\gamma\varepsilon}\,d\boldsymbol{x}dsdt
+\\
\int\limits_0^{t'}\int\limits_{\Xi^1}\omega_{\gamma\varepsilon}^{M_*}
\partial_s(e^{-\xi t}a(\omega_{\gamma\varepsilon}e^{\xi
t}))\,d\boldsymbol{x}dsdt
+\int\limits_0^{t'}\int\limits_{\Xi^1}\omega_{\gamma\varepsilon}^{M_*}
\mathrm{div}_x(e^{-\xi
t}\boldsymbol{\varphi}(\omega_{\gamma\varepsilon}e^{\xi
t}))\,d\boldsymbol{x}dsdt=
\end{multline*}
\begin{multline}
\int\limits_0^{t'}\int\limits_{\Xi^1}\omega_{\gamma\varepsilon}^{M_*}
\Delta_x\omega_{\gamma\varepsilon}\,d\boldsymbol{x}dsdt+
\varepsilon\int\limits_0^{t'}\int\limits_{\Xi^1}\omega_{\gamma\varepsilon}^{M_*}
\partial_s^2\omega_{\gamma\varepsilon}\,d\boldsymbol{x}dsdt+\\
\int\limits_0^{t'}\int\limits_{\Xi^1}\omega_{\gamma\varepsilon}^{M_*}
K_\gamma(t,\tau)e^{-\xi
t}\beta(\boldsymbol{x},s,\omega_{\gamma\varepsilon}e^{\xi
t})\,d\boldsymbol{x}dsdt-
\xi\int\limits_0^{t'}\int\limits_{\Xi^1}\omega_{\gamma\varepsilon}^{M_*}
\omega_{\gamma\varepsilon}\,d\boldsymbol{x}dsdt,\quad t'\in(0,T].
\label{e11.38}%
\end{multline}

Let us consider each term in \eqref{e11.38} separately. We have
\begin{equation}
\label{e11.39i}%
\int\limits_0^{t'}\int\limits_{\Xi^1}\omega_{\gamma\varepsilon}^{M_*}
\partial_t\omega_{\gamma\varepsilon}\,d\boldsymbol{x}dsdt\overset{\eqref{e11.35}}{=}
\int\limits_0^{t'}\int\limits_{\Xi^1}\omega_{\gamma\varepsilon}^{M_*}
\partial_t\omega_{\gamma\varepsilon}^{M_*}\,d\boldsymbol{x}dsdt\overset{\eqref{e11.34}}{=}
\frac12\Vert\omega_{\gamma\varepsilon}^{M_*}(t')\Vert_{L^2(\Xi^1)};
\end{equation}
\begin{multline}
\label{e11.39ii}%
\int\limits_0^{t'}\int\limits_{\Xi^1}\omega_{\gamma\varepsilon}^{M_*}
\partial_s(e^{-\xi t}a(\omega_{\gamma\varepsilon}e^{\xi
t}))\,d\boldsymbol{x}dsdt=\int\limits_0^{t'}\int\limits_{\Xi^1}\omega_{\gamma\varepsilon}^{M_*}
a'(\omega_{\gamma\varepsilon}e^{\xi
t})\partial_s\omega_{\gamma\varepsilon}\,d\boldsymbol{x}dsdt\overset{\eqref{e11.33},\eqref{e11.36}}{=}
\\
\int\limits_0^{t'}\int\limits_{\Xi^1}\omega_{\gamma\varepsilon}^{M_*}
a'((\omega_{\gamma\varepsilon}^{M_*}+M_*)e^{\xi
t})\partial_s\omega_{\gamma\varepsilon}^{M_*}\,d\boldsymbol{x}dsdt=
\int\limits_0^{t'}\int\limits_{\Xi^1}
\partial_s
W(t,\omega_{\gamma\varepsilon}^{M_*})\,d\boldsymbol{x}dsdt=
\\
\int\limits_0^{t'}\int\limits_{\Omega}\left(W(t,\omega_{\gamma\varepsilon}^{M_*}(\boldsymbol{x},t,S))-
W(t,\omega_{\gamma\varepsilon}^{M_*}(\boldsymbol{x},t,0))\right)\,d\boldsymbol{x}dt,
\end{multline}
where $\displaystyle W(t,\lambda)=\int\limits_0^\lambda \lambda
a'((\lambda+M_*)e^{\xi t})\,d\lambda$;
\begin{multline}
\label{e11.39iii}%
\int\limits_0^{t'}\int\limits_{\Xi^1}\omega_{\gamma\varepsilon}^{M_*}
\mathrm{div}_x(e^{-\xi
t}\boldsymbol{\varphi}(\omega_{\gamma\varepsilon}e^{\xi
t}))\,d\boldsymbol{x}dsdt=
\int\limits_0^{t'}\int\limits_{\Xi^1}\omega_{\gamma\varepsilon}^{M_*}
\boldsymbol{\varphi}'(\omega_{\gamma\varepsilon}e^{\xi t}))
\cdot\nabla_x\omega_{\gamma\varepsilon}\,d\boldsymbol{x}dsdt\overset{\eqref{e11.33},\eqref{e11.37}}{=}
\\
\int\limits_0^{t'}\int\limits_{\Xi^1}\omega_{\gamma\varepsilon}^{M_*}
\boldsymbol{\varphi}'((\omega_{\gamma\varepsilon}^{M_*}+M_*)e^{\xi
t})\cdot\nabla_x\omega_{\gamma\varepsilon}^{M_*}\,d\boldsymbol{x}dsdt=
\int\limits_0^{t'}\int\limits_0^{S}\int\limits_{\partial\Omega}
\boldsymbol{\Phi}(t,\omega_{\gamma\varepsilon}^{M_*}(\boldsymbol{\sigma},t,s))
\cdot\boldsymbol{n}(\boldsymbol{\sigma})\,d\boldsymbol{\sigma}dsdt\overset{\eqref{e11.34}}{=}0,
\end{multline}
where $\displaystyle \boldsymbol{\Phi}(t,\lambda)=\int\limits_0^\lambda \lambda
\boldsymbol{\varphi}'((\lambda+M_*)e^{\xi t})\,d\lambda$;
\begin{multline}
\label{e11.39iv}%
\int\limits_0^{t'}\int\limits_{\Xi^1}\omega_{\gamma\varepsilon}^{M_*}
\Delta_x\omega_{\gamma\varepsilon}\,d\boldsymbol{x}dsdt\overset{\eqref{e11.34}}{=}
-\int\limits_0^{t'}\int\limits_{\Xi^1}
\nabla_x\omega_{\gamma\varepsilon}^{M_*}\cdot
\nabla_x\omega_{\gamma\varepsilon}\,d\boldsymbol{x}dsdt\overset{\eqref{e11.37}}{=}\\
-\int\limits_0^{t'}\int\limits_{\Xi^1}
\vert\nabla_x\omega_{\gamma\varepsilon}^{M_*}\vert^2\,d\boldsymbol{x}dsdt=
-\Vert
\nabla_x\omega_{\gamma\varepsilon}^{M_*}\Vert_{L^2(\Xi^1\times(0,t'))}^2;
 \end{multline}
\begin{multline}
\label{e11.39v}%
\varepsilon\int\limits_0^{t'}\int\limits_{\Xi^1}\omega_{\gamma\varepsilon}^{M_*}
\partial_s^2\omega_{\gamma\varepsilon}\,d\boldsymbol{x}dsdt\overset{\eqref{e11.34}}{=}
-\varepsilon\int\limits_0^{t'}\int\limits_{\Xi^1}
\partial_s\omega_{\gamma\varepsilon}^{M_*}
\partial_s\omega_{\gamma\varepsilon}\,d\boldsymbol{x}dsdt
\overset{\eqref{e11.37}}{=}\\
-\varepsilon\int\limits_0^{t'}\int\limits_{\Xi^1}
\vert\partial_s\omega_{\gamma\varepsilon}^{M_*}\vert^2\,d\boldsymbol{x}dsdt=
-\varepsilon\Vert\partial_s\omega_{\gamma\varepsilon}^{M_*}\Vert_{L^2(\Xi^1\times(0,t'))}^2;
 \end{multline}
\begin{multline}
\label{e11.39vi}%
\int\limits_0^{t'}\int\limits_{\Xi^1}\omega_{\gamma\varepsilon}^{M_*}
K_\gamma(t,\tau)e^{-\xi
t}\beta(\boldsymbol{x},s,\omega_{\gamma\varepsilon}e^{\xi
t})\,d\boldsymbol{x}dsdt\overset{\eqref{e11.33}}{=}
\\
\int\limits_0^{t'}\int\limits_{\Xi^1}\omega_{\gamma\varepsilon}^{M_*}
K_\gamma(t,\tau)e^{-\xi
t}\beta(\boldsymbol{x},s,(\omega_{\gamma\varepsilon}^{M_*}+M_*)e^{\xi
t})\,d\boldsymbol{x}dsdt;
 \end{multline}
and, finally,
 \begin{equation}
\label{e11.39vii}%
\xi\int\limits_0^{t'}\int\limits_{\Xi^1}\omega_{\gamma\varepsilon}^{M_*}
\omega_{\gamma\varepsilon}\,d\boldsymbol{x}dsdt\overset{\eqref{e11.33}}{=}
\xi\int\limits_0^{t'}\int\limits_{\Xi^1}\omega_{\gamma\varepsilon}^{M_*}
(\omega_{\gamma\varepsilon}^{M_*}+M_*)\,d\boldsymbol{x}dsdt.
 \end{equation}

Aggregating \eqref{e11.38}--\eqref{e11.39vii}, we arrive at the
integral equality
\begin{multline}
\label{e11.39}%
\frac12\Vert\omega_{\gamma\varepsilon}^{M_*}(t')\Vert_{L^2(\Xi^1)}^2+
\Vert\nabla_x\omega_{\gamma\varepsilon}^{M_*}\Vert_{L^2(\Xi^1\times(0,t'))}^2
+\\
\varepsilon\Vert\partial_s\omega_{\gamma\varepsilon}^{M_*}\Vert_{L^2(\Xi^1\times(0,t'))}^2+
\int\limits_0^{t'}\int\limits_{\Xi^1}(\omega_{\gamma\varepsilon}^{M_*}+M_*)
\omega_{\gamma\varepsilon}^{M_*}\xi\,d\boldsymbol{x}dtds=
\\
\int\limits_0^{t'}\int\limits_{\Xi^1}\omega_{\gamma\varepsilon}^{M_*}
K_\gamma(t,\tau)e^{-\xi
t}\beta(\boldsymbol{x},s,(\omega_{\gamma\varepsilon}^{M_*}+M_*)e^{\xi
t})\,d\boldsymbol{x}dtds.
\end{multline}

Estimating the right-hand side from above, using the Lagrange
mean-value theorem and the growth condition \eqref{u1.07iv}, we get
\begin{multline}
\label{e11.40}%
\int\limits_0^{t'}\int\limits_{\Xi^1}\omega_{\gamma\varepsilon}^{M_*}
K_\gamma(t,\tau)e^{-\xi
t}\beta(\boldsymbol{x},s,(\omega_{\gamma\varepsilon}^{M_*}+M_*)e^{\xi
t})\,d\boldsymbol{x}dtds=
\\
\int\limits_0^{t'}\int\limits_{\Xi^1}\omega_{\gamma\varepsilon}^{M_*}
K_\gamma(t,\tau)e^{-\xi
t}\left(\beta(\boldsymbol{x},s,(\omega_{\gamma\varepsilon}^{M_*}+M_*)e^{\xi
t})-\beta(\boldsymbol{x},s,M_*e^{\xi
t})\right)\,d\boldsymbol{x}dtds+
\\
\int\limits_0^{t'}\int\limits_{\Xi^1}\omega_{\gamma\varepsilon}^{M_*}
K_\gamma(t,\tau)e^{-\xi t}\beta(\boldsymbol{x},s,M_*e^{\xi
t})\,d\boldsymbol{x}dtds\leqslant \\ \int\limits_0^{t'}b_0
K_\gamma(t,\tau)\Vert\omega_{\gamma\varepsilon}^{M_*}(t)\Vert_{L^2(\Xi^1)}^2dt+
\int\limits_0^{t'}\int\limits_{\Xi^1}
\omega_{\gamma\varepsilon}^{M_*} K_\gamma(t,\tau) e^{-\xi t}
\beta(\boldsymbol{x},s,M_*e^{\xi t})\,d\boldsymbol{x}dsdt.
\end{multline}
Now, we use arbitrariness of $\xi$ and choose it such that the
second integral in the right-hand side of \eqref{e11.40} vanishes.
Recall that $\mathrm{supp\,}
K_\gamma(\cdot,\tau)\subset[\tau-\gamma,\tau]$. Hence
$$K_\gamma(t,\tau)\equiv0\quad \mbox{for } 0\leqslant t\leqslant
t_0:=\tau-\frac{\gamma_0}{2}.$$
On the strength of condition
\eqref{u1.07v}, in order to secure that
$\beta(\boldsymbol{x},s,M_*e^{\xi t_0})\equiv 0$
$\forall\,(\boldsymbol{x},s)\in \Xi^1$, it is sufficient to take $\xi$
such that
\begin{equation}
\label{e11.41}%
M_*e^{\xi t_0}\geqslant b_1+1.
\end{equation}
Remark that $M_*$ depends on $\xi$, and we have
\begin{equation*}
M_*=\max\{\|u_0^{(1)}\|_{L^\infty(\Xi^1)},\,\|u_0^{(2)}
e^{-\xi t}\|_{L^\infty(\Xi^2)},\, \|u_S^{(2)} e^{-\xi
t}\|_{L^\infty(\Xi^2)}\}\geqslant \|
u_0^{(1)}\|_{L^\infty(\Xi^1)}.
  \end{equation*}
Therefore, \eqref{e11.41} holds true if $\Vert u_0^{(1)}\Vert e^{\xi
t_0}\geqslant b_1+1$. The latter holds with $\xi=0$ for $\Vert
u_0^{(1)}\Vert\geqslant b_1+1$ and
$\displaystyle \xi=\xi_*=\frac1{t_0}\ln\frac{b_1+1}{\Vert
u_0^{(1)}\Vert_{L^\infty(\Xi^1)}}$ for $\Vert
u_0^{(1)}\Vert_{L^\infty(\Xi^1)}<b_1+1$.

Thus, the second integral in the right-hand side of \eqref{e11.40}
vanishes, if $\xi$ is defined by \eqref{e11.31}.

Further, notice that
\[
\int\limits_0^{t'}\int\limits_{\Xi^1}(\omega_{\gamma\varepsilon}^{M_*}+M_*)\omega_{\gamma\varepsilon}^{M_*}\xi_*\,d\boldsymbol{x}dtds
\geqslant0
\]
and discard this integral along with
$\Vert\nabla_x\omega_{\gamma\varepsilon}^{M_*}\Vert_{L^2(\Xi^1\times(0,t'))}^2$
and
$\varepsilon\Vert\partial_s\omega_{\gamma\varepsilon}^{M_*}\Vert_{L^2(\Xi^1\times(0,t'))}^2$
from the left-hand side of \eqref{e11.39}. As the result of above
considerations, we derive the inequality
\begin{equation}
\Vert\omega_{\gamma\varepsilon}^{M_*}(t')\Vert_{L^2(\Xi^1)}^2
\leqslant\int\limits_0^{t'} b_0K_\gamma(t,\tau)
\Vert\omega_{\gamma\varepsilon}^{M_*}(t')\Vert_{L^2(\Xi^1)}^2dt,\quad
t'\in(0,T].
\end{equation}
Applying Gr\"onwall's lemma, we conclude that
$\Vert\omega_{\gamma\varepsilon}^{M_*}(t')\Vert_{L^2(\Xi^1)}^2=0$,
which implies that
\begin{multline}
\label{e11.42}%
u_{\gamma\varepsilon}(\boldsymbol{x},t',s)\leqslant
M_*e^{\xi_*t'}\leqslant e^{\xi_*t'}\max\left\{\Vert
u_0^{(1)}\Vert_{L^\infty(\Xi^1)},\,\Vert
u_0^{(2)}\Vert_{L^\infty(\Xi^2)},\,\Vert
u_S^{(2)}\Vert_{L^\infty(\Xi^2)}\right\}\\
\forall\,(\boldsymbol{x},t',s)\in
G_{T,S},\quad
\forall\,\gamma\in(0,\gamma_0/2],\quad \forall\,\varepsilon\in(0,1].
  \end{multline}
Analogously, we also establish that
\begin{equation}
\label{e11.43}%
u_{\gamma\varepsilon}(\boldsymbol{x},t',s)\geqslant
-M_*e^{\xi_*t'}\geqslant -e^{\xi_*t'}\max\left\{\Vert
u_0^{(1)}\Vert_{L^\infty(\Xi^1)},\,\Vert
u_0^{(2)}\Vert_{L^\infty(\Xi^2)},\,\Vert
u_S^{(2)}\Vert_{L^\infty(\Xi^2)}\right\}.
  \end{equation}
Combining \eqref{e11.42} and \eqref{e11.43} we establish the maximum
principle \eqref{e11.30} (with $\xi_*$ defined by \eqref{e11.31}),
which completes the proof of Lemma $\ref{lem.11.4}$. \qed

\begin{corollary}
\label{cor.11.1} The family of kinetic
solutions $\{u_\gamma\}_{\gamma\in(0,\gamma_0/2]}$ of Problem
$\Pi_\gamma$ with $Z_\gamma=K_\gamma\beta$ satisfies the maximum
principle
\begin{equation}
\label{e11.44}%
\Vert u_\gamma(\cdot,t',\cdot)\Vert_{L^\infty(\Xi^1)}\leqslant
M_7(t')\leqslant M_7(T),\quad
\forall\,t'\in(0,T],\quad\forall\,\gamma\in(0,\gamma_0/2].
  \end{equation}
\end{corollary}
\proof Bound \eqref{e11.44} directly follows from \eqref{e11.30} and the
limiting relation
$u_\gamma=\mbox{s-}\lim\limits_{\varepsilon\to0}u_{\gamma\varepsilon}$, see
Remark \ref{rem.7.2}. \qed

\begin{corollary}
\label{cor.11.2} The family of
$\chi$-functions
$\{\chi(\lambda;u_\gamma)\}_{\gamma\in(0,\gamma_0/2]}$ is uniformly
in $\gamma$ bounded in $L^2(G_{T,S}\times[-M_7(T),M_7(T)])$.
\end{corollary}
\proof From \eqref{e11.44} and \eqref{chi-f}, it immediately follows
that
\[
\Vert\chi(\cdot,u_{\gamma\varepsilon})\Vert_{L^2(G_{T,S}\times[-M_7(T),M_7(T)])}
\leqslant\sqrt{T\, S\, M_7(T)\,\mbox{meas\,}\Omega},
\]
which completes the proof. \qed

\begin{corollary}
\label{cor.11.3} The supports of the
measures $m_\gamma$ and $n_\gamma$ lay in the layer
$\{-M_7(T)\leqslant\lambda\leqslant M_7(T)\}$ for all
$\gamma\in(0,\gamma_0/2]$.
\end{corollary}
\proof The assertion directly follows from Corollary \ref{cor.11.1},
Remark \ref{rem.7.2} and representation \eqref{e4.12iv}. \qed

\section{Derivation of the impulsive kinetic equation} \label{ImpulsiveKin}
In this section, from the kinetic equation \eqref{e3.01a} with
$Z_\gamma=K_\gamma\beta$ (or, equivalently, from
\eqref{e11.26}--\eqref{e11.26ii}) we derive the limiting impulsive
kinetic equation, as $\gamma\to0+$. We expose this procedure in the
form of the sequence of several lemmas.

\begin{lemma}
\label{lem.12.1} The family of kinetic solutions
$\{u_\gamma\}_{\gamma\in(0,\gamma_0/2]}$ of Problem $\Pi_\gamma$ (in
the sense of Definition \ref{def.3.02}) is relatively compact in
$L^1(G_{T,S})$.
\end{lemma}
\proof Kinetic equation \eqref{e3.01a} with $Z_\gamma=K_\gamma\beta$
or, equivalently, kinetic equation \eqref{e11.26} with
$\mathcal{Z}_\gamma$ defined by \eqref{e11.26ii}, has the form of
the kinetic equation in requirement (iv) in Proposition
\ref{prop.4.1} with $f_\nu=\chi(\lambda;u_{\gamma_\nu})$,
$g_\nu\equiv0$ and
$k_\nu=m_{\gamma_\nu}+n_{\gamma_\nu}+\mathcal{Z}_{\gamma_\nu}$. Here
$\{\gamma_\nu\}_{\nu\in\mathbb N}$ is an arbitrary sequence such
that $\gamma_\nu\underset{\nu\to\infty}{\longrightarrow}0+$. On the
strength of Corollaries \ref{cor.11.2} and \ref{cor.11.3} and assertions (ii) and  (iii) of Lemma
\ref{lem.11.3}, there exist a subsequence of $\{\gamma_\nu\}$,
still denoted by $\{\gamma_\nu\}$, a limiting function $f_*$ and a
limiting measure $k_*$ such that
\begin{equation}
\label{e12.01}
\chi(\lambda;u_{\gamma_\nu})\underset{\nu\to\infty}{\longrightarrow}f_*\text{
weakly in } L^2(G_{T,S}\times[-M_7(T),M_7(T)]),
  \end{equation}
\begin{equation}
\label{e12.02}
k_\nu\underset{\nu\to\infty}{\longrightarrow}k_*\text{
weakly* in }\mathcal{M}(G_{T,S}\times[-M_7(T),M_7(T)]).
  \end{equation}
Since $\mathcal{M}(G_{T,S}\times[-M_7(T),M_7(T)])$ is compactly
embedded into
$W_{\mathrm{loc}}^{-1,p'}(G_{T,S}\times[-M_7(T),M_7(T)])$ for any
$p'\in\left[1,\frac{d+3}{d+2}\right)$ due to Sobolev's embedding
theorem \cite[Chapter I, Section 8]{Sob-2008}, we also have that
\begin{equation}
\label{e12.03}
k_\nu\underset{\nu\to\infty}{\longrightarrow}k_*\text{
strongly in
}W_{\mathrm{loc}}^{-1.p'}(G_{T,S}\times[-M_7(T),M_7(T)]).
  \end{equation}
On the strength of Proposition \ref{prop.4.1}, from the limiting
relations \eqref{e12.01}--\eqref{e12.03} it follows that

\begin{equation}
\label{e12.04}%
\int\limits_{-M_7(T)}^{M_7(T)}\chi(\lambda;u_{\gamma_\nu})\,d\lambda
\underset{\nu\to\infty}{\longrightarrow} u_*:=
\int\limits_{-M_7(T)}^{M_7(T)}f_*\,d\lambda \quad \text{strongly in
}L^2(G_{T,S}).
  \end{equation}
On the strength of Lemma \ref{lem.3.02} and item (i) in Lemma
\ref{lem.3.01}, this yields that
\begin{equation}
\label{e12.05}%
u_{\gamma_\nu}\underset{\nu\to\infty}{\longrightarrow}u_*\text{
strongly in }L^2(G_{T,S})\;(\text{and in }L^1(G_{T,S})),
  \end{equation}
and that $f_*$ is the $\chi$-function:
\begin{equation}
\label{e12.06}%
f_*(\boldsymbol{x},t,s,\lambda)=\chi(\lambda;u_*(\boldsymbol{x},t,s)).
\end{equation}
Lemma \ref{lem.12.1} is proved. \qed

\begin{corollary}
\label{cor.12.1} {\bf (Corollary of Lemmas \ref{lem.12.1}, \ref{lem.11.3} and
\ref{lem.11.4}.)} The limiting function $u_*$ belongs to
$L^\infty(G_{T,S})\cap
L^2((0,T)\times(0,S);\text{\it\r{W}}{}_2^{\,1}(\Omega))$ and
satisfies the maximum principle
\begin{equation}
\label{e12.06i}%
-M_7(T)\leqslant u_*(\boldsymbol{x},t,s)\leqslant M_7(T)\text{ for
a.e. }(\boldsymbol{x},t,s)\in G_{T,S}
  \end{equation}
and the energy estimate
\begin{equation}
\label{e12.06ii}%
\Vert\nabla_x u_*\Vert_{L^2(G_{T,S})}\leqslant C_6.
\end{equation}
\end{corollary}
\begin{lemma}
\label{lem.12.2}%
Let
$m_\gamma,n_\gamma\in\mathcal{M}^+(G_{T,S}\times\mathbb{R}_\lambda)$
be the measures having
place in formulation of Problem $\Pi_\gamma$, and function
$\mathcal{Z}_\gamma\in L^1(G_{T,S}\times\mathbb{R}_\lambda)$ be
defined by formula \eqref{e11.26ii}.

Then there exist subsequence
$\gamma_\nu\underset{\nu\to\infty}{\longrightarrow}0+$ and measures
$m_*\in\mathcal{M}^+(G_{T,S}\times\mathbb{R}_\lambda)$ and\linebreak
$\mathcal{Z}_*\in\mathcal{M}(G_{T,S}\times\mathbb{R}_\lambda)$
such that
\begin{equation}
\label{e12.07}%
m_{\gamma_\nu}+n_{\gamma_\nu}-\delta_{(\lambda=u^*)}\vert\nabla_x
u_*\vert^2\underset{\nu\to\infty}{\longrightarrow} m_*\text{
weakly}^*\text{ in }\mathcal{M}(G_{T,S}\times\mathbb{R}_\lambda),
\end{equation}
\begin{equation}
\label{e12.08}%
\mathcal{Z}_{\gamma_\nu}\underset{\nu\to\infty}{\longrightarrow}
\mathcal{Z}_*\text{ weakly}^*\text{ in
}\mathcal{M}(G_{T,S}\times\mathbb{R}_\lambda),
\end{equation}
\begin{equation}
\label{e12.09}%
\mathrm{supp\,}m_*\subset\overline{G}_{T,S}\times[-M_7(T),M_7(T)],
\end{equation}
\begin{equation}
\label{e12.10}%
\mathrm{supp\,}\mathcal{Z}_*\subset\Xi^1\times\{t=\tau\}\times\{-b_1\leqslant\lambda\leqslant
b_1\}.
\end{equation}
In particular,
\begin{equation}
\label{e12.11}%
m_{\gamma_\nu}+n_{\gamma_\nu}+\mathcal{Z}_{\gamma_\nu}
\underset{\nu\to\infty}{\longrightarrow}m_*+n_*+\mathcal{Z}_*\text{
weakly}^*\text{ in }\mathcal{M}(G_{T,S}\times\mathbb{R}_\lambda),
\end{equation}
where $n_*=\delta_{(\lambda=u_*)}\vert\nabla_x u_*\vert^2$.
\end{lemma}
\proof Firstly, on the strength of Alaoglu's theorem, \eqref{e12.07}
follows from assertion (ii) of Lemma \ref{lem.11.3}, and
\eqref{e12.08} follows from assertion (iii) of Lemma
\ref{lem.11.3}.
Secondly, notice that measure
$\lim\limits_{\nu\to\infty}(m_{\gamma_\nu}+n_{\gamma_\nu}-\delta_{(\lambda=u_*)}\vert\nabla_xu_*\vert^2)$
is nonnegative, since $m_{\gamma_\nu}$ is nonnegative for any $\gamma_\nu$ and
$\lim\limits_{\nu\to\infty}(n_{\gamma_\nu}-\delta_{(\lambda=u_*)}\vert\nabla_xu_*\vert^2)$ is
nonnegative due to lower semicontinuity property. Thus $m_*$ is nonnegative.
Finally, inclusion \eqref{e12.09} follows from Corollary
\ref{cor.11.3}, and inclusion \eqref{e12.10}
follows from representation \eqref{e11.26ii} and condition
\eqref{u1.07v}.

 Lemma \ref{lem.12.2} is proved. \qed

\begin{lemma}
\label{lem.12.3}%
The limiting quadruple $(u_*,m_*,n_*,\mathcal{Z}_*)$, defined in
Lemmas \ref{lem.12.1} and \ref{lem.12.2}, resolves the kinetic
equation
\begin{equation}
\label{e12.12}%
\partial_t\chi(\lambda;u_*)+a'(\lambda)\partial_s\chi(\lambda;u_*)
+\boldsymbol{\varphi}'(\lambda)\cdot\nabla_x\chi(\lambda;u_*)-
\Delta_x\chi(\lambda;u_*)=\partial_\lambda(m_*+n_*+\mathcal{Z}_*)
\end{equation}
in the sense of distributions.
\end{lemma}
\proof The proof is achieved by limiting passage in \eqref{e11.26}
as $\gamma_\nu\underset{\nu\to\infty}{\longrightarrow}0+$. This
passage relies on the limiting relations \eqref{e12.01},
\eqref{e12.11} and representation \eqref{e12.06}. \qed

\begin{lemma}
\label{lem.12.4}%
Let the quadruple $u_*\in L^\infty(G_{T,S})\cap
L^2((0,T)\times(0,S);\text{\it\r{W}}{}_2^{\,1}(\Omega))$,
$m_*\in\mathcal{M}^+(G_{T,S}\times[-M_7(T),M_7(T)])$,
$n_*=\delta_{(\lambda=u_*)}|\nabla_x u_*|^2$ and
$\mathcal{Z}_*\in\mathcal{M}(G_{T,S}\times[-M_7(T),M_7(T)])$
resolve the kinetic equation \eqref{e12.12} in the sense of
distributions. (As the matter of fact, we consider the case when
$u_*=\lim\limits_{\nu\to\infty}u_{\gamma_\nu}$.)

Then there exists $u_{*0}^{\mathrm{tr},(1)}\in L^\infty(\Xi^1)$ ---
the trace of $u_*$ on $\Gamma_0^1$ --- such that
\begin{equation}
\label{e12.13}%
\underset{t\to0+}{\mathrm{ess\,lim}}\int\limits_{\Xi^1} \left\vert
u_*(\boldsymbol{x},t,s)-
u_{*0}^{\mathrm{tr},(1)}(\boldsymbol{x},s)\right\vert\,d\boldsymbol{x}ds=0,
\end{equation}
\begin{equation}
\label{e12.14}%
\underset{t\to0+}{\mathrm{ess\,lim}}\int\limits_{-M_7(T)}^{M_7(T)}\int\limits_{\Xi^1}
\left\vert\chi(\lambda;u_*(\boldsymbol{x},t,s))-
\chi(\lambda;u_{*0}^{\mathrm{tr},(1)}(\boldsymbol{x},s))\right\vert\,d\boldsymbol{x}dsd\lambda=0.
\end{equation}

Furthermore, for any fixed $\tau_*\in(0,T)$ there exist the both
one-sided traces $u_{*\tau_*\pm 0}^{\mathrm{tr},(1)}\in
L^\infty(\Xi^1)$ such that
\begin{equation}
\label{e12.15}%
\underset{t\to\tau_*\pm0}{\mathrm{esslim}}\int\limits_{\Xi^1}
\left\vert u_*(\boldsymbol{x},t,s)-
u_{*\tau_{*}\pm0}^{\mathrm{tr},(1)}(\boldsymbol{x},s)\right\vert\,d\boldsymbol{x}ds=0,
\end{equation}
\begin{equation}
\label{e12.16}%
\underset{t\to\tau_*\pm0}{\mathrm{esslim}}\int\limits_{-M_7(T)}^{M_7(T)}\int\limits_{\Xi^1}
\left\vert\chi(\lambda;u_*(\boldsymbol{x},t,s))-
\chi(\lambda;u_{*\tau_{*}\pm0}^{\mathrm{tr},(1)}(\boldsymbol{x},s))\right\vert\,d\boldsymbol{x}dsd\lambda=0.
\end{equation}
In particular, \eqref{e12.15} and \eqref{e12.16} are valid for
$\tau_*=\tau$.
\end{lemma}
\noindent {\it Proof} relies on Lemma \ref{lem.12.3} and Corollary \ref{cor.12.1}.
It is just a minor modification of the proofs of Lemmas
\ref{lem.5.1} and \ref{lem.5.2}; therefore we skip it. \qed

\begin{lemma}
\label{lem.12.5}%
Let $(u_*,m_*,n_*,\mathcal{Z}_*)$ be the limiting point of the
sequence $\left\{
\left(u_{\gamma_\nu},m_{\gamma_\nu},n_{\gamma_\nu},\mathcal{Z}_{\gamma_\nu}\right)
\right\}$ (as $\nu\to +\infty$) and resolve the kinetic equation
\eqref{e12.12} in the sense of distributions.

Then $u_*\in C(0,\tau;L^1(\Xi^1))\cap C(\tau,T;L^1(\Xi^1))$. In
particular, the initial condition
\begin{equation}
\label{e12.17}%
\underset{t\to0+}{\mathrm{ess\,lim}}\int\limits_{\Xi^1} \left\vert
u_*(\boldsymbol{x},t,s)-
u_{0}^{(1)}(\boldsymbol{x},s)\right\vert\,d\boldsymbol{x}ds=0
\end{equation}
holds true.
\end{lemma}
\proof%
We keep track of the proof of Lemma \ref{lem.6.1}, with some
necessary natural modifications.

Taking $\eta(u_{\gamma_\nu})=\pm u_{\gamma_\nu}$ in \eqref{e3.04a}, we arrive at equation
\eqref{e1.01a} with $Z_{\gamma_\nu}=K_{\gamma_\nu}\beta$. In the
sense of distributions, this equation is equivalent to the integral
equality
\begin{multline}
\label{e12.18}%
\langle \partial_t
u_{\gamma_\nu}(t),\phi\rangle_{W^{-1,2}(\Xi^1),\text{\it\r{W}}{}_2^{\,1}(\Xi^1)}=\\
\int\limits_{\Xi^1}(a(u_{\gamma_\nu})\partial_s\phi+
\boldsymbol{\varphi}(u_{\gamma_\nu})\cdot\nabla_x\phi
-\nabla_xu_{\gamma_\nu}\cdot\nabla_x\phi
+K_{\gamma_\nu}(t,\tau)
\beta(\boldsymbol{x},s,u_{\gamma_\nu})\phi)\,d\boldsymbol{x}ds=0
\\ \forall\,\phi\in\text{\it\r{W}}{}_2^{\,1}(\Xi^1),\quad \forall\,
t\in[0,T],\quad \forall\, \gamma_\nu>0.
\end{multline}
Fixing arbitrarily $\tau_0\in(0,\tau)$, from \eqref{e12.18}, the
maximum principle \eqref{e11.44}, the energy estimate \eqref{e11.27}
and the localization property
$\mathrm{supp\,}K_{\gamma_\nu}\subset(\tau-\gamma_\nu,\tau]$, we
easily derive the bound
\begin{multline}
\label{e12.18-bis}%
\langle \partial_t
u_{\gamma_\nu}(t),\phi\rangle_{W^{-1,2}(\Xi^1),\text{\it\r{W}}{}_2^{\,1}(\Xi^1)}
\leqslant C_{13}(\tau_0)\Vert\phi\Vert_{W_2^{\,1}(\Xi^1)} \\
\forall\,\phi\in\text{\it\r{W}}{}_2^{\,1}(\Xi^1),\quad \forall\, t\in[0,\tau_0], \quad
\forall\, \gamma_\nu\in\left(0,\min\left\{\tau_0,\frac{\gamma_0}2\right\}\right),
\end{multline}
where
\[
C_{13}(\tau_0)=\sqrt{S\,\mathrm{meas\,}\Omega}\left(\Vert
a\Vert_{C([-M_8(\tau_0),M_8(\tau_0)])}+
\Vert\boldsymbol{\varphi}\Vert_{C([-M_8(\tau_0),M_8(\tau_0)])}
\right)+2 C_6.
\]
This means that the family $\{\partial_t
u_{\gamma_\nu}\}_{\gamma_\nu\in\left(0,\min\left\{\tau_0,
\frac{\gamma_0}2\right\}\right)}$ is uniformly bounded in
$L^\infty(0,\tau_0;W^{-1,2}(\Xi^1))$. Hence the family of functions
$\{u_{\gamma_\nu}\!:\;[0,\tau_0]\mapsto W^{-1,2}(\Xi^1)\}
_{\gamma_\nu\in\left(0,\min\left\{\tau_0,
\frac{\gamma_0}2\right\}\right)}$ is equicontinuous, and we have
\begin{equation}
\label{e12.19}%
\begin{split}
u_{\gamma_\nu}(\cdot,t,\cdot)\underset{t\to t_0}{\longrightarrow}
u_{\gamma_\nu}(\cdot,t_0,\cdot) & \text{ strongly in
}W^{-1,2}(\Xi^1)\\ & \text{ uniformly in
}\gamma_\nu\in\left(0,\min\left\{\tau_0,
\frac{\gamma_0}2\right\}\right)\text{ for a.e. }
t_0\in[0,\tau_0].
\end{split}
\end{equation}
Here, in the case $t_0=0$, the right-sided limit is meant, and, in the
case $t_0=\tau_0$,  the left-sided limit is meant.

On the other hand, due to \eqref{e11.44}, values of the mappings
$t\mapsto u_{\gamma_\nu}(\cdot,t,\cdot)$ belong to the set
\[
\mathfrak{F}\overset{\mathrm{def}}{=}\{\phi\in L^2(\Xi^1)\!:\;
\underset{(\boldsymbol{x},s)\in\Xi^1}{\mathrm{ess\,sup}}\vert\phi(\boldsymbol{x},s)\vert\leqslant
M_7(\tau)\},
\]
which is a compact subset in $W^{-1,2}(\Xi^1)$ by the Rellich
theorem. Therefore, by the Arcel theorem, the set
$\{u_{\gamma_\nu}\}_{\gamma_\nu\in\left(0,\min\left\{\tau_0,
\frac{\gamma_0}2\right\}\right)}$ is relatively compact in
$C(0,\tau_0;W^{-1,2}(\Xi^1))$. Hence
\begin{equation}
\label{e12.20}%
u_{\gamma_\nu}(\cdot,t,\cdot)\underset{t\to\infty}{\longrightarrow}
u_*(\cdot,t,\cdot)\text{ strongly in }W^{-1,2}(\Xi^1)\text{
uniformly on  the segment }[0,\tau_0].
\end{equation}
Here we recall that $\{\gamma_\nu\to0\}$ is the subsequence
extracted in the proof of Lemma \ref{lem.12.1}. Next, from
\eqref{e12.13} and \eqref{e12.15} it immediately follows that
\begin{equation}
\label{e12.21a}%
u_*(\cdot,t,\cdot)\underset{t\to0+}{\longrightarrow}
u_{*0}^{\mathrm{tr},(1)}\text{ strongly in }W^{-1,2}(\Xi^1),
\end{equation}
\begin{equation}
\label{e12.21b}%
u_*(\cdot,t,\cdot)\underset{t\to t_0\pm0}{\longrightarrow}
u_{*t_0}^{\mathrm{tr},(1)}\text{ strongly in
}W^{-1,2}(\Xi^1)\quad\forall\, t_0\in(0,\tau_0),
\end{equation}
\begin{equation}
\label{e12.21c}%
u_*(\cdot,t,\cdot)\underset{t\to\tau_0-0}{\longrightarrow}
u_{*t_0-0}^{\mathrm{tr},(1)}\text{ strongly in }W^{-1,2}(\Xi^1).
\end{equation}

From \eqref{e12.19}--\eqref{e12.21c}, by the triangle inequality we
deduce that
\begin{equation}
\label{e12.22}%
\begin{split}
& u_{*0}^{\mathrm{tr},(1)}(\boldsymbol{x},s)=u_0^{(1)}(\boldsymbol{x},s),\quad u_{*\tau_0-0}^{\mathrm{tr},(1)}(\boldsymbol{x},s)=
u_*(\boldsymbol{x},\tau_0-0,s),\\
& u_{*t_0}^{\mathrm{tr},(1)}(\boldsymbol{x},s)=u_*(\boldsymbol{x},t_0,s)\quad
\forall\,t_0\in(0,\tau_0).
\end{split}
\end{equation}
Using arbitrariness of $\tau_0\in(0,\tau)$, the localization
property $\mathrm{supp\,}K_{\gamma_\nu}\subset(\tau-\gamma_\nu,\tau]$
and the fact that the set $\mathfrak{F}$ does not depend on
$\tau_0$,  and taking the sequences
$\tau_{0\mu}\underset{\mu\to\infty}{\longrightarrow}\tau-0$ and
$\gamma_\mu\in\left(0,\min\{\tau_{0\mu},\frac{\gamma_0}2\}\right)$
($\gamma_\mu\leqslant \gamma_\nu$), we deduce that equalities
\eqref{e12.22} hold true with $\tau$ on the place of $\tau_0$.
Inserting $u_0^{(1)}(\boldsymbol{x},s)$ on the place of
$u_{*0}^{\mathrm{tr},(1)}(\boldsymbol{x},s)$ in \eqref{e12.13}, $u_*(\vx,t_0,s)$ and $u_*(\vx,\tau-0,s)$ on the places of $u_{*t_0\pm 0}^{tr,(1)}(\vx,s)$ and
$u_{*\tau-0}^{\mathrm{tr},(1)}(\boldsymbol{x},s)$ in
\eqref{e12.15}, we finally establish that $u_*\in
C(0,\tau;L^1(\Xi^1))$.

Quite analogously, we verify that $u_*\in C(\tau,T;L^1(\Xi^1))$ and
thus complete the proof of the lemma.\qed

The following lemma finalizes derivation of the
impulsive kinetic equation \eqref{e9.04a-bis}.

\begin{lemma}
\label{lem.12.6}%
The limiting measure
$\mathcal{Z}_*=\mbox{w*-}\lim\limits_{\nu\to\infty}\mathcal{Z}_{\gamma_\nu}$
admits the representation
\begin{equation}
\label{e12.23}%
\mathcal{Z}_*=\delta_{(t=\tau-0)}\boldsymbol{1}_{(\lambda\geqslant
u_*(\boldsymbol{x},t,s))}\beta(\boldsymbol{x},s,u_*(\boldsymbol{x},t,s)),
  \end{equation}
which is understood in the sense of distributions and therefore can
be equivalently written as follows:
\begin{equation}
\label{e12.24}%
\langle\mathcal{Z}_*,\phi\rangle=\int\limits_{\Xi^1}
\int\limits_{\mathbb{R}_\lambda}\mathbf{1}_{(\lambda\geqslant
u_*(\boldsymbol{x},\tau-0,s))}\beta(\boldsymbol{x},s,u_*(\boldsymbol{x},\tau-0,s))\phi(\boldsymbol{x},\tau-0,s,\lambda)\,d\boldsymbol{x}dsd\lambda
\quad \forall\,\phi\in C_0(G_{T,S}\times\mathbb{R}_\lambda).
  \end{equation}
\end{lemma}
\proof Let $\phi\in C_0(G_{T,S}\times\mathbb{R}_\lambda)$ be an
arbitrary test function. Set
$$\Phi(\boldsymbol{x},t,s,\lambda)=\int\limits_{-\infty}^\lambda\phi(\boldsymbol{x},t,s,\lambda')\,d\lambda',$$
i.e., $\Phi$ is the primitive of $\phi$ with respect to $\lambda$.
On the strength of representation \eqref{e11.26ii} and Conditions on
$K_\gamma\&\beta$, the following chain of equalities holds true:
\begin{multline}
\label{e12.25}%
\langle\mathcal{Z}_{\gamma_\nu},\phi\rangle=\int\limits_{G_{T,S}}
\int\limits_{\mathbb{R}_\lambda}\mathbf{1}_{(\lambda\geqslant
u_{\gamma_\nu})}K_{\gamma_\nu}(t,\tau)\beta(\boldsymbol{x},s,\lambda)
\phi(\boldsymbol{x},t,s,\lambda)\,d\lambda d\boldsymbol{x}dtds-
\\
\int\limits_{G_{T,S}} \int\limits_{\mathbb{R}_\lambda}
\phi(\boldsymbol{x},t,s,\lambda)\int\limits_0^\lambda
\mathbf{1}_{(\lambda'\geqslant
u_{\gamma_\nu})}K_{\gamma_\nu}(t,\tau)\partial_{\lambda'}\beta(\boldsymbol{x},s,\lambda')
\,d\lambda'd\lambda d\boldsymbol{x}dtds=
\\
\int\limits_0^T K_{\gamma_\nu}(t,\tau) \int\limits_{\Xi^1}
 \int\limits_{u_{\gamma_\nu}}^{+\infty}
\beta(\boldsymbol{x},s,\lambda)\phi(\boldsymbol{x},t,s,\lambda)\,d\lambda
d\boldsymbol{x}dtds-
\\
\int\limits_0^T K_{\gamma_\nu}(t,\tau) \int\limits_{\Xi^1}
\int\limits_{\mathbb{R}_\lambda}\phi(\boldsymbol{x},t,s,\lambda)
(\beta(\boldsymbol{x},s,\lambda)-\beta(\boldsymbol{x},s,u_{\gamma_\nu}))
\mathbf{1}_{(\lambda\geqslant u_{\gamma_\nu})}\,d\lambda
d\boldsymbol{x}dtds=
\\
\int\limits_0^T K_{\gamma_\nu}(t,\tau) \int\limits_{\Xi^1}
\beta(\boldsymbol{x},s,u_{\gamma_\nu})\int\limits_{u_{\gamma_\nu}}^{+\infty}\phi(\boldsymbol{x},t,s,\lambda)\,d\lambda
d\boldsymbol{x}dtds=
\\
-\int\limits_{\tau-\gamma_\nu}^\tau\frac{2}{\gamma_\nu}
\omega\left(\frac{t-\tau}{\gamma_\nu}\right) \int\limits_{\Xi^1}
\beta(\boldsymbol{x},s,u_{\gamma_\nu}(\boldsymbol{x},t,s))\Phi(\boldsymbol{x},t,s,u_{\gamma_\nu}(\boldsymbol{x},t,s))\,d\boldsymbol{x}dtds\equiv
I_{\gamma_\nu}.
  \end{multline}
Let us change variable $t$ for
$\displaystyle \zeta=\frac{\tau-t}{\gamma_\nu}$. Taking into account that $\omega$
is even, for $\gamma_\nu<\gamma_0$ we have
\begin{equation}
\label{e12.26}%
I_{\gamma_\nu}=-\int\limits_0^1\int\limits_{\Xi^1}2\omega(\zeta)\beta(\boldsymbol{x},s,u_{\gamma_\nu}(\boldsymbol{x},\tau-\zeta\gamma_\nu,s))
\Phi(\boldsymbol{x},\tau-\zeta\gamma_\nu,s,u_{\gamma_\nu}(\boldsymbol{x},\tau-\zeta\gamma_\nu,s))\,d\boldsymbol{x}dsd\zeta.
\end{equation}
Recall that $Z_\gamma=K_\gamma\beta\equiv0$ for
$t\in[0,\tau-\gamma]$. Therefore, $u_{\gamma_\nu}$ and $u_*$ are
kinetic and entropy solutions of the same problem on the set
$\{(\boldsymbol{x},t,s)\in\Omega\times[0,\tau-\gamma]\times[0,S]\}$,
i.e., Problem $\Pi_\gamma$ (or $\Pi_0$). Since
$u_*=\lim\limits_{\nu\to\infty}u_{\gamma_\nu}$, this implies that
\begin{equation}
\label{e12.27}%
u_{\gamma_\nu}(\boldsymbol{x},t,s)=u_*(\boldsymbol{x},t,s)\text{ on
}\Xi^1\times[0,\tau-\gamma_\nu]
\end{equation}
On the strength of Lemma \ref{lem.12.5} and identity \eqref{e12.27},
we conclude that
\begin{equation}
\label{e12.28}%
u_{\gamma_\nu}(\boldsymbol{x},\tau-\zeta\gamma_\nu,s)\underset{\nu\to\infty}{\longrightarrow}u_*(\boldsymbol{x},\tau-0,s)\quad \text{a.e. in }\Xi^1\times\{0<\zeta<1\}.
\end{equation}
Since $\beta$ and $\Phi$ are smooth and bounded, using the Lebesgue
dominated convergence theorem \cite[Theorem 1.4.48]{TAO}, we establish
the limiting relation
\begin{equation}
\label{e12.29}%
I_{\gamma_\nu}\underset{\nu\to\infty}{\longrightarrow}
-\int\limits_0^1\int\limits_{\Xi^1}2\omega(\zeta)
\beta(\boldsymbol{x},s,u_*(\boldsymbol{x},\tau-0,s))\Phi(\boldsymbol{x},\tau-0,s,u_*(\boldsymbol{x},\tau-0,s))\,d\boldsymbol{x}dsd\zeta.
\end{equation}
Taking into account that $\int\limits_0^12\omega(\zeta)\,d\zeta=1$
and
\[
\Phi(\boldsymbol{x},\tau-0,s,u_*(\boldsymbol{x},\tau-0,s))=-\int\limits_{\mathbb{R}_\lambda}
\mathbf{1}_{(\lambda\geqslant
u_*(\boldsymbol{x},\tau-0,s))}\phi(\boldsymbol{x},\tau-0,s,\lambda)\,d\lambda
\]
and combining \eqref{e12.25}, \eqref{e12.29} and
$\mathcal{Z}_*=\mbox{w*-}\lim\limits_{\nu\to\infty}\mathcal{Z}_{\gamma_\nu}$,
we arrive at the limiting relation \eqref{e12.24}, which completes
the proof of the lemma.\qed

\begin{corollary}
\label{cor.12.2} {\bf (Corollary of Lemmas \ref{lem.12.3} and \ref{lem.12.6}.)}
 The limiting triple
$(u_*,m_*,n_*)$, defined by Lemmas \ref{lem.12.1} and
\ref{lem.12.2}, resolves the kinetic equation \eqref{e9.04a-bis} in
the sense of distributions.
  \end{corollary}

\begin{corollary}
\label{cor.12.3} {\bf (Corollary of Lemmas \ref{lem.12.3} and \ref{lem.12.6}.)}
 The limiting triple
$(u_*,m_*,n_*)$, defined by Lemmas \ref{lem.12.1} and
\ref{lem.12.2}, resolves the kinetic equation \eqref{e9.04a} on
$\Omega\times((0,\tau)\cup(\tau,T))\times(0,S)\times\mathbb{R}_\lambda$
in the sense of distributions.
  \end{corollary}

\section{Derivation of the kinetic impulsive condition (\ref{e9.04e})}
\label{ImpulsiveKC}
In view of Lemma \ref{lem.12.5} and representation \eqref{e12.24},
the impulsive kinetic equation \eqref{e9.04a-bis} is equivalent in
the sense of distributions to the integral equality
\begin{multline}
\label{e13.01}%
\int\limits_{\Xi^1\times(t_0,t_1)\times\mathbb{R}_\lambda}\chi(\lambda;u_*)
\left(\partial_t\phi+a'(\lambda)\partial_s\phi+\boldsymbol{\varphi}'(\lambda)\cdot\nabla_x\phi+\Delta_x\phi\right)
\,d\boldsymbol{x}dtdsd\lambda=
\\
\int\limits_{\Xi^1\times\mathbb{R}_\lambda}\left(\chi(\lambda;u_*(\boldsymbol{x},t_1,s))
\phi(\boldsymbol{x},t_1,s,\lambda)-\chi(\lambda;u_*(\boldsymbol{x},t_0,s))
\phi(\boldsymbol{x},t_0,s,\lambda)\right)\,d\boldsymbol{x}dsd\lambda+
\\
\int\limits_{\Xi^1\times\mathbb{R}_\lambda}\mathbf{1}_{(\tau\in(t_0,t_1))}
\mathbf{1}_{(\lambda\geqslant
u_*(\boldsymbol{x},\tau-0,s))}\beta(\boldsymbol{x},s,u_*(\boldsymbol{x},\tau-0,s))
\partial_\lambda\phi(\boldsymbol{x},\tau-0,s,\lambda)\,d\boldsymbol{x}dsd\lambda+
\\
\langle
m_*+n_*,\partial_\lambda\phi\rangle_{\mathcal{M}(\Xi^1\times[t_0,t_1]\times\mathbb{R}_\lambda),\,
C(\Xi^1\times[t_0,t_1]\times\mathbb{R}_\lambda)}
\end{multline}
for all values $t_0$, $t_1\in[0,T]$, such that $t_0<t_1$,
$t_0,t_1\neq\tau$, and for all test-functions $\phi\in
C^2(G_{T,S}\times\mathbb{R}_\lambda)$ vanishing in the neighborhood
of $\partial\Xi^1$ and for sufficiently large $\vert\lambda\vert$,
say, for $\vert\lambda\vert\geqslant\max\{M_7(T),b_1\}+1$. (Here $M_7(T)$
is given by \eqref{e11.30} and $b_1$ is given by \eqref{u1.07v}).

Taking
$\phi(\boldsymbol{x},t,s,\lambda)=\phi_1(\boldsymbol{x},s)\phi_2(\lambda)$,
where $\phi_2\equiv 1$ on $[-M_7(T), M_7(T)]$, and passing to the
limit as $t_0\to\tau-0$ and $t_1\to\tau+0$, from \eqref{e13.01} we
derive the integral inequality
\begin{multline}
\label{e13.02}%
\int\limits_{\Xi^1}\phi_1(\boldsymbol{x},s)\int\limits_{-M_7(T)}^{M_7(T)}
\left(\chi(\lambda;u_*(\boldsymbol{x},\tau+0,s))-\chi(\lambda;u_*(\boldsymbol{x},\tau-0,s))\right)\,d\lambda
d\boldsymbol{x}ds=
\\
\int\limits_{\Xi^1}\phi_1(\boldsymbol{x},s)\beta(\boldsymbol{x},s,u_*(\boldsymbol{x},\tau-0,s))\left(
\int\limits_{u_*(\boldsymbol{x},\tau-0,s)}^{+\infty}
\partial_\lambda \phi_2(\lambda)\,d\lambda\right)\,d\boldsymbol{x}ds.
\end{multline}
Here we took into account that function $\chi$ and measures $m_*$
and $n_*$ are supported on the segment
$\{-M_7(T)\leqslant\lambda\leqslant M_7(T)\}$. In the right-hand
side of \eqref{e13.02} notice that
\begin{equation}
\label{e13.03}%
\int\limits_{u_*(\boldsymbol{x},\tau-0,s)}^{+\infty}\partial_\lambda\phi_2(\lambda)\,d\lambda=-\phi_2(u_*(\boldsymbol{x},\tau-0,s))+1
  \end{equation}
and, using assertion (i) of Lemma \ref{lem.3.01}, represent
\begin{equation}
\label{e13.04}%
\beta(\boldsymbol{x},s,u_*(\boldsymbol{x},\tau-0,s))=
\int\limits_{-M_7(T)}^{M_7(T)}\chi(\lambda;u_*(\boldsymbol{x},\tau-0,s))\partial_\lambda\beta(\boldsymbol{x},s,\lambda)\,d\lambda+
\beta(\boldsymbol{x},s,0).
  \end{equation}
Combining \eqref{e13.02}--\eqref{e13.04}, we arrive at the integral
equality
\begin{multline}
\label{e13.05}%
\int\limits_{\Xi^1}\phi_1(\boldsymbol{x},s)\left(
\int\limits_{-M_7(T)}^{M_7(T)}\chi(\lambda;u_*(\boldsymbol{x},\tau+0,s))\,d\lambda\right)\,d\boldsymbol{x}ds=
\\
\int\limits_{\Xi^1}\phi_1(\boldsymbol{x},s)\left(
\int\limits_{-M_7(T)}^{M_7(T)}(1+\partial_\lambda\beta(\boldsymbol{x},s,\lambda))\chi(\lambda;u_*(\boldsymbol{x},\tau-0,s))\,d\lambda+
\beta(\boldsymbol{x},s,0)\right)\,d\boldsymbol{x}ds.
  \end{multline}
Using assertion (i) of Lemma \ref{lem.3.01} again, we rewrite
\eqref{e13.05} in the equivalent form
\begin{equation}
\label{e13.06}%
\int\limits_{\Xi^1}\phi_1(\boldsymbol{x},s)u_*(\boldsymbol{x},\tau+0,s)\,d\boldsymbol{x}ds=
\int\limits_{\Xi^1}\phi_1(\boldsymbol{x},s)\left(u_*(\boldsymbol{x},\tau-0,s)
+\beta(\boldsymbol{x},s,u_*(\boldsymbol{x},\tau-0,s))\right)\,d\boldsymbol{x}ds.
  \end{equation}
Since $\phi_1$ is arbitrary, \eqref{e13.06} immediately yields the
impulsive condition \eqref{e9.02b} for $u_*$.

In turn, \eqref{e9.02b} yields that integration in \eqref{e13.05} with respect to
$\lambda$ is fulfilled over interval $[-M_3,M_3]$. ($M_3$ is
defined in \eqref{e9.06}). Thus, we have established the
following result.

\begin{lemma}
 \label{lem.13.1}
The limiting function $u_*=\lim\limits_{\nu\to\infty}u_{\gamma_\nu}$
satisfies the impulsive condition \eqref{e9.02b} and the kinetic
impulsive condition \eqref{e9.04e}.
\end{lemma}

Notice that in this section we also have justified the claim of
Remark \ref{rem.9.01-bis}, as a byproduct.

\section{Derivation of the kinetic boundary conditions (\ref{e9.04c}) and (\ref{e9.04d}).
Completion of the proof of Theorem \ref{theo.9.1}}
\label{Limited} 
Recall the sequence $\{u_{\gamma_\nu}\}_{\nu\in{\mathbb N}}$ that
was introduced in the proof of Lemma \ref{lem.12.1}, so that the
limiting relation \eqref{e12.05} holds. Let us consider the traces
$u_{\gamma_\nu 0}^{tr,(2)}$ and $u_{\gamma_\nu S}^{tr,(2)}$ of
$u_{\gamma_\nu}$ on the sets $\Xi^2\times \{s=0+\}$ and $\Xi^2\times
\{s=S-0\}$, respectively. We start with the observation that, by the
Tartar theorem \cite{TAR-1975}, \cite[Chapter 3]{MNRR-1996}, there
exist a subsequence still denoted by
$\{u_{\gamma_\nu}\}_{{}^{\nu\to\infty}_{(\gamma_\nu \to 0+)}}$ and
two families of probability Radon measures $\Lambda_{x,t}^{(0)}$ and
$\Lambda_{x,t}^{(S)}$ supported uniformly on $[-M_3,M_3]$ such that
\begin{subequations} \label{e14.00}
\begin{eqnarray} \displaystyle \phi(u_{\gamma_\nu 0}^{tr,(2)}) \underset{\nu\to\infty}{\longrightarrow} \bar{\phi} \mbox{ weakly* in } L^\infty(\Xi^2),& & \displaystyle \bar{\phi}(\vx,t)=\int\limits_{{\mathbb R}_\lambda} \phi(\lambda) d\Lambda_{x,t}^{(0)}(\lambda), \label{e14.00a}\\
\displaystyle \zeta(u_{\gamma_\nu S}^{tr,(2)}) \underset{\nu\to\infty}{\longrightarrow} \bar{\zeta} \mbox{ weakly* in } L^\infty(\Xi^2),& & \displaystyle \bar{\zeta}(\vx,t)=\int\limits_{{\mathbb R}_\lambda} \zeta(\lambda) d\Lambda_{x,t}^{(S)}(\lambda) \label{e14.00b}
\end{eqnarray}
\end{subequations}
for all $\phi,\zeta\in C({\mathbb R}_\lambda)$.

Measures $\Lambda_{x,t}^{(0)}$ and $\Lambda_{x,t}^{(S)}$ are the
Young measures associated with the subsequences\linebreak
$\{u_{\gamma_\nu 0}^{tr,(2)}\}_{{}^{\nu\to\infty}_{(\gamma_\nu \to
0+)}}$ and $\{u_{\gamma_\nu
S}^{tr,(2)}\}_{{}^{\nu\to\infty}_{(\gamma_\nu \to 0+)}}$,
respectively. The mappings $(\vx,t)\mapsto \Lambda_{x,t}^{(0)}$
and\linebreak $(\vx,t)\mapsto \Lambda_{x,t}^{(S)}$ are weakly*
measurable. More precisely, they belong to the space
$L_w^\infty(\Xi^2;{\mathcal M}({\mathbb R}_\lambda))$ (for details,
see \cite[Chapter 3, Definition 2.7]{MNRR-1996}).

Now we are in a position to formulate and prove the following assertion. \begin{lemma} \label{lem.14.1}
The limiting function $u_*=\lim\limits_{\nu\to\infty}
u_{\gamma_\nu}$ satisfies the entropy boundary conditions
\eqref{e9.07d} and \eqref{e9.07e}.
\end{lemma}
\proof {\bf (1)} Since the kinetic equation \eqref{e9.04a} on $\Omega\times ((0,\tau)\cup(\tau,T))\times (0,S)\times {\mathbb R}_\lambda$ is exactly the kinetic equation \eqref{e3.01a} with $Z_\gamma \equiv 0$, similarly to Lemma \ref{lem.5.3} we conclude that there exist $u_{*0}^{tr,(2)},u_{*S}^{tr,(2)}\in L^\infty(\Xi^2)$ --- the traces of $u_*$ on $\Gamma_{0+}^2$ and $\Gamma_{S-0}^2$, respectively --- satisfying the limiting relations
\begin{subequations} \label{e14.01}
\begin{eqnarray}
& \displaystyle \underset{s\to 0+}{\mbox{ess\,lim}} \int\limits_{\Xi^2} |u_*(\vx,t,s)-u_{*0}^{tr,(2)}(\vx,t)|d\vx dt=0, \label{e14.01a}\\
& \displaystyle \underset{s\to S-0}{\mbox{ess\,lim}} \int\limits_{\Xi^2} |u_*(\vx,t,s)-u_{*S}^{tr,(2)}(\vx,t)|d\vx dt=0. \label{e14.01b}
\end{eqnarray}
\end{subequations}
{\bf (2)} Firstly, we consider the case $t\leqslant \tau$. Since $\supp K_\gamma(\cdot,\tau) \subset [\tau-\gamma,\tau]$, we deduce that
\begin{equation} \label{e14.02}
u_*(\vx,t,s)=u_\gamma(\vx,t,s) \quad \mbox{a.e. in} \quad \Omega \times (0,\tau-\gamma)\times (0,S).
\end{equation}
Along with Lemma \ref{lem.5.3}, this implies that
\begin{equation} \label{e14.03}
u_{*0}^{tr,(2)}(\vx,t)=u_{\gamma 0}^{tr,(2)}(\vx,t),\quad
u_{*S}^{tr,(2)}(\vx,t)=u_{\gamma S}^{tr,(2)}(\vx,t)\quad \mbox{for a.e. } (\vx,t)\in\Omega \times (0,\tau-\gamma).
\end{equation}
Hence,
\begin{subequations} \label{e14.04}
\begin{eqnarray}
& \displaystyle u_{\gamma 0}^{tr,(2)} \underset{\gamma\to 0+}{\longrightarrow} u_{*0}^{tr,(2)} \quad \mbox{a.e. in } \Omega\times (0,\tau), \label{e14.04a}\\
& \displaystyle u_{\gamma S}^{tr,(2)} \underset{\gamma\to 0+}{\longrightarrow} u_{*S}^{tr,(2)} \quad \mbox{a.e. in } \Omega\times (0,\tau). \label{e14.04b}
\end{eqnarray}
\end{subequations}
Since functions $a=a(\lambda)$ and $q_a=q_a(\lambda)$ are continuous, using \eqref{e14.04} we pass to the limit as $\gamma\to 0+$ in the entropy boundary conditions \eqref{e3.04d} (with $u_0^{tr,(2)}=u_{\gamma 0}^{tr,(2)}$) and \eqref{e3.04e} (with $u_S^{tr,(2)}=u_{\gamma S}^{tr,(2)}$) and derive the entropy boundary conditions \eqref{e9.07d} and \eqref{e9.07e} on $\Omega\times (0,\tau)$.\\
{\bf (3)} Secondly, we consider the case $t>\tau$. Since $\supp K_\gamma(\cdot,\tau) \subset [\tau-\gamma,\tau]$, the inequality \eqref{e3.04a} has the form \eqref{e9.07a} on $\Omega \times (\tau,T) \times (0,S)$ (with $u=u_\gamma$). Taking into account existence of traces $u_{\gamma 0}^{tr,(2)}$ and $u_{\gamma S}^{tr,(2)}$ on $\Omega \times (\tau,T)\times \{s=0+\}$ and $\Omega\times (\tau,T)\times \{s=S-0\}$, respectively, this inequality is equivalent in the sense of distributions to the integral inequality
\begin{multline} \label{e14.05}
\int\limits_\tau^T \int\limits_0^S \int\limits_\Omega \Bigl[\eta(u_\gamma) \partial_t \phi + q_a(u_\gamma) \partial_s \phi + \vq_\varphi(u_\gamma)\cdot\nabla_x\phi + \eta(u_\gamma) \Delta_x \phi - \eta''(u_\gamma) |\nabla_x u_\gamma|^2 \phi \Bigr] d\vx dsdt\\
-\int\limits_\tau^T \int\limits_\Omega \Bigl[q_a(u_{\gamma S}^{tr,(2)}) \phi(\vx,t,S) - q_a(u_{\gamma 0}^{tr,(2)}) \phi(\vx,t,0) \Bigr]d\vx dt \geqslant 0,
\end{multline}
where $\phi=\phi(\vx,t,s)$ is an arbitrary smooth nonnegative test-function vanishing in the neighborhood of the sections $\{t=\tau\}$ and $\{t=T\}$.\\
{\bf (4)} Due to sufficient arbitrariness of the test-function $\phi$, for $\eta(u_\gamma)=\pm u_\gamma$ inequality \eqref{e14.05} reduces to the integral equality
\begin{multline} \label{e14.06}
\int\limits_\tau^T \int\limits_0^S \int\limits_\Omega \Bigl[u_\gamma \partial_t \psi + a(u_\gamma) \partial_s \psi + \vvarphi(u_\gamma)\cdot\nabla_x\psi + u_\gamma \Delta_x \psi \Bigr] d\vx dsdt\\
-\int\limits_\tau^T \int\limits_\Omega \Bigl[a(u_{\gamma S}^{tr,(2)}) \psi(\vx,t,S) - a(u_{\gamma 0}^{tr,(2)}) \psi(\vx,t,0) \Bigr]d\vx dt = 0,
\end{multline}
where $\psi=\psi(\vx,t,s)$ is an arbitrary smooth (not necessarily nonnegative) test-function vanishing in the neighborhood of the sections $\{t=\tau\}$ and $\{t=T\}$. Substitute the test-function $\psi(\vx,t,s) =\theta(\vx,t) \rho_\delta^0(s)$ into \eqref{e14.06}, where $\theta$ is an arbitrary smooth test-function satisfying the above stated finiteness demands, and $\rho_\delta^0$ is defined by \eqref{e6.05} and \eqref{e6.07-Phi}. Passing to the limit in \eqref{e14.06} as $\gamma\to 0+$ along the subsequence $\{\gamma_\nu\}_{\nu\to\infty}$, using \eqref{e12.05} and \eqref{e14.00a} we establish that
\begin{multline*}
\int\limits_0^T \int\limits_0^\delta \int\limits_\Omega \Bigl[u_* \rho_\delta^0(s) \partial_t \theta + a(u_*)(\rho_\delta^0)'(s) \theta + \vvarphi(u_*) \rho_\delta^0(s)\cdot \nabla_x \theta + u_* \rho_\delta^0(s) \Delta_x\theta \Bigr] d\vx ds dt+\\
\int\limits_\tau^T \int\limits_\Omega \Biggl[\,\int\limits_{\RR_\lambda} a(\lambda) d\Lambda_{x,t}^{(0)}(\lambda) \Biggr] \theta d\vx dt=0.
\end{multline*}
From this, passing to the limit as $\delta\to 0+$, we derive the  relation
\begin{equation*}
\int\limits_\tau^T \int\limits_\Omega \Biggl[\,\int\limits_{\RR_\lambda} a(\lambda) d\Lambda_{x,t}^{(0)}(\lambda) \Biggr] \theta d\vx dt = \int\limits_\tau^T \int\limits_\Omega a(u_{*0}^{tr,(2)})\theta d\vx dt,
\end{equation*}
or, equivalently,
\begin{subequations} \label{e14.07}
\begin{equation} \label{e14.07a}
\int\limits_{\RR_\lambda} a(\lambda) d\Lambda_{x,t}^{(0)}(\lambda)  = a(u_{*0}^{tr,(2)}(\vx,t))\quad \mbox{for a.e. } (\vx,t)\in \Omega \times (\tau,T),
\end{equation}
due to arbitrariness of $\theta$.

Quite analogously, using the test-function $\rho_\delta^S(s)=\rho_\delta^0(S-s)$ in \eqref{e14.06} and limiting relations \eqref{e12.05} and \eqref{e14.00b}, we derive the limiting relation
\begin{equation} \label{e14.07b}
\int\limits_{\RR_\lambda} a(\lambda) d\Lambda_{x,t}^{(S)}(\lambda)  = a(u_{*S}^{tr,(2)}(\vx,t))\quad \mbox{for a.e. } (\vx,t)\in \Omega \times (\tau,T).
\end{equation}
\end{subequations}
{\bf (5)} Now substitute $\phi(\vx,t,s)=\theta(\vx,t) \rho_\delta^0(s)$ into \eqref{e14.05}. Here $\theta$ is nonnegative, so that $\phi$ is a valid test-function. Passing to the limit in \eqref{e14.05} as $\gamma\to 0+$ along the subsequence $\{\gamma_\nu\}_{\nu\to\infty}$, using \eqref{e12.05} and \eqref{e14.00a}, we establish that
\begin{multline*}
\int\limits_\tau^T \int\limits_0^\delta \int\limits_\Omega \Bigl[\eta(u_*) \rho_\delta^0(s) \partial_t \theta + q_a(u_*) (\rho_\delta^0)'(s) \theta + \vq_\varphi(u_*) \rho_\delta^0(s)\cdot\nabla_x\theta + \eta(u_*) \rho_\delta^0(s) \Delta_x \theta  \Bigr] d\vx dsdt +\\
\int\limits_\tau^T \int\limits_\Omega \Biggl[\,\int\limits_{\RR_\lambda} q_a(\lambda) d\Lambda_{x,t}^{(0)}(\lambda) \Biggr] \theta d\vx dt \geq \underset{\nu\to\infty}{\liminf} \int\limits_\tau^T \int\limits_0^\delta \int\limits_\Omega \eta''(u_{\gamma_\nu}) |\nabla_x u_{\gamma_\nu}|^2 \rho_\delta^0(s) \theta d\vx ds dt \geqslant 0.
\end{multline*}
From this, passing to the limit as $\delta\to 0+$, we derive the inequality
\begin{equation*}
\int\limits_\tau^T \int\limits_\Omega \Biggl[\,\int\limits_{\RR_\lambda} q_a(\lambda) d\Lambda_{x,t}^{(0)}(\lambda) \Biggr] \theta d\vx dt \geqslant \int\limits_\tau^T \int\limits_\Omega q_a(u_{*0}^{tr,(2)})\theta d\vx dt,
\end{equation*}
or, equivalently,
\begin{subequations} \label{e14.08}
\begin{equation} \label{e14.08a}
\int\limits_{\RR_\lambda} q_a(\lambda) d\Lambda_{x,t}^{(0)}(\lambda) \geqslant  q_a(u_{*0}^{tr,(2)}(\vx,t)) \quad \mbox{for a.e. } (\vx,t)\in \Omega\times (\tau,T),
\end{equation}
due to arbitrariness and non-negativeness of $\theta$.

Quite analogously, using the test-function $\rho_\delta^S(s) =\rho_\delta^0(S-s)$ in \eqref{e14.05} and the limiting relations \eqref{e12.05} and \eqref{e14.00b}, we derive the relation
 \begin{equation} \label{e14.08b}
q_a(u_{*S}^{tr,(2)}(\vx,t)) \geqslant \int\limits_{\RR_\lambda} q_a(\lambda) d\Lambda_{x,t}^{(S)}(\lambda)    \quad \mbox{for a.e. } (\vx,t)\in \Omega\times (\tau,T).
\end{equation}
\end{subequations}
{\bf (6)} Passing to the limit in the entropy boundary condition \eqref{e3.04d} (with $u_0^{tr,(2)}=u_{0\gamma}^{tr,(2)}$) on $\Omega\times (\tau,T)$ as $\gamma\to 0+$ along the subsequence $\{\gamma_\nu\}_{\nu\to\infty}$, using \eqref{e12.05} and \eqref{e14.00a}, we get
\begin{multline} \label{e14.09a}
\int\limits_{\RR_\lambda} q_a(\lambda) d\Lambda_{x,t}^{(0)}(\lambda) - q_a(u_0^{(2)}(\vx,t))-
\eta'(u_0^{(2)}(\vx,t))\Biggl(\,\int\limits_{\RR_\lambda} a(\lambda) d\Lambda_{x,t}^{(0)}(\lambda) - a(u_0^{(2)}(\vx,t)) \Biggr) \leqslant 0 \\ \mbox{for a.e. } (\vx,t)\in \Omega\times (\tau,T).
\end{multline}
Combining \eqref{e14.09a} with \eqref{e14.07a} and \eqref{e14.08a}, we establish \eqref{e9.07d} on $\Omega\times (\tau,T)$.

Analogously, using \eqref{e12.05}, \eqref{e14.00b}, \eqref{e14.07b} and \eqref{e14.08b}, from \eqref{e3.04e} we deduce \eqref{e9.07e} on $\Omega\times (\tau,T)$, which completes the proof of the lemma. \qed

\begin{lemma} \label{lem.14.2}
The limiting function $u_*=\lim\limits_{\nu\to\infty}
u_{\gamma_\nu}$ satisfies the kinetic boundary conditions
\eqref{e9.04c} and \eqref{e9.04d} with some nonnegative Radon
measures $\mu_{*0}^{(2)},\mu_{*S}^{(2)}\in\mathcal{M}^+(\Xi^2\times\mathbb{R}_\lambda)$ such
that $\mathrm{supp\,}\mu_{*0}^{(2)}, \,\mathrm{supp\,}\mu_{*S}^{(2)}\subset\Xi^2\times[-M_3,M_3]$.
\end{lemma}
 \proof It is sufficient to notice that the kinetic boundary conditions \eqref{e9.04c} and \eqref{e9.04d} are equivalent to the entropy boundary conditions \eqref{e9.07d} and \eqref{e9.07e}, respectively. The proof of this claim is quite analogous to justification of equivalency of \eqref{e3.01c}--\eqref{e3.01d} to \eqref{e3.04d}--\eqref{e3.04e} in the proof of Lemma \ref{lem.6.3} (see formulas \eqref{e6.10}--\eqref{e6.15}).

 Lemma \ref{lem.14.2} is proved. \qed\\

 {\it Completion of the proof of Theorem \ref{theo.9.1}.} Collecting altogether the results of Corollary \ref{cor.12.3} and Lemmas \ref{lem.12.5}, \ref{lem.13.1} and \ref{lem.14.2}, we conclude that the limiting function $u_*$ of the extracted in the beginning of Section \ref{Limited} subsequence $\{u_{\gamma_\nu}\}_{{}^{\nu\to\infty}_{(\gamma_\nu \to 0+)}}$ is a kinetic solution of Problem $\Pi_0$ in the sense of Definition \ref{def.9.01}.
 Due to item 2 of Theorem \ref{theo.9.1}, function $u_*$ is the unique kinetic solution of Problem $\Pi_0$. Therefore, {\bf the whole family} $\{u_\gamma\}_{\gamma>0}$ tends to $u_*$ strongly in $L^1(G_{T,S})$ and weakly in $L^2((0,T)\times(0,S);\text{\it\r{W}}{}_2^{\,1}(\Omega))$, as $\gamma\to 0+$.

 Finally, due to item 3 of Theorem \ref{theo.9.1}, function $u_*$ is the unique entropy solution of Problem $\Pi_0$ in the sense of Definition \ref{def.9.02}. This observation completes the proof of item 1 of Theorem \ref{theo.9.1}.
 Thus, all assertions in Theorem \ref{theo.9.1} are proved. \qed


\section*{Acknowledgment}
\nocite{} The work was supported by the RSCF grant no. 19-11-00069.

\end{document}